\documentclass[12pt]{article}
\usepackage{graphicx}
\usepackage{amssymb,amsmath,amsthm}
\usepackage{mathrsfs}
\usepackage{bbold}
\usepackage{enumitem}

\DeclareMathAlphabet{\mathpzc}{OT1}{pzc}{m}{it}

\def\a{{\mathrm a}}

\def\K{{\Bbb K}}
\def\bc{\mathbb{c}}
\def\ff{\mathfrak{f}}
\def\fN{\boldsymbol{N}}
\def\fJ{\boldsymbol{J}}
\def\fZ{\boldsymbol{Z}}

\def\bh{\mathbb{h}}
\def\bp{\mathbb{p}}

\def\bN{\mathbb{N}}

\def\bC{\mathbb{C}}
\def\bR{\mathbb{R}}

\def\RPt{{\mathbb{RP}^2}}

\def\CPo{{\mathbb{CP}^1}}
\def\Rt{{\mathbb{R}^2}}

\def\cA{{\cal A}}

\def\cB{{\cal B}}
\def\cC{{\cal C}}

\def\cE{{\cal E}}

\def\cI{{\cal I}}

\def\cH{{\cal H}}

\def\cK{{\cal K}}

\def\cP{{\cal P}}

\def\cF{{\cal F}}

\newtheorem{definition}{Definition}
\newtheorem{theorem}{Theorem}

\newtheorem{corollary}{Corollary}

\newtheorem{remark}{Remark}

\newtheorem{proposition}{Proposition}
\newtheorem{conjecture}{Conjecture}
\usepackage{array}
\newcolumntype{S}{>{\centering\arraybackslash} m{.475\linewidth}}
\newcolumntype{T}{>{\arraybackslash} m{10.5cm}}
\newcolumntype{U}{>{\centering\arraybackslash} m{2cm}}
\title{%
  Julia sets of Newton maps of real quadratic polynomial maps on the plane.
}
\author{Roberto De Leo\\ \small Howard University, Washington DC 20059 (USA)}
\begin{document}
\maketitle
\begin{abstract}
  We study numerically the $\alpha$- and $\omega$-limits of the Newton maps of two of the most elementary families of polynomial transformations
  on the plane: those with a linear component and those with both components of degree two. Our results are fully consistent with some conjectures
  we posed in a recent work about the dynamics of Newton maps.
\end{abstract}
\section{Introduction}
The dynamics of Newton maps of planar holomorphic maps has been extensively studied over the last hundred years as a chapter of the
more general problem of the discrete dynamics of maps in one {\em complex} variable, after the seminal works of G.~Julia~\cite{Jul18} and
P.~Fatou~\cite{Fat19,Fat20a,Fat20b}. This field saw a strong acceleration when computers became powerful enough to visualize
the intricacies of their attractors and retractors -- this moment is usually identified with the publication of the celebrated article
of B. Mandelbrot on his homonymous fractal~\cite{Man80}.

Despite the large number of articles and books dedicated to this topic and to its extension to general rational functions (here we just mention
the recent book by Milnor~\cite{Mil06} and a rich survey by Lyubich~\cite{Lyu86} and refer the interested reader to~\cite{DL18} for
an extensive bibliography), very little attention has been dedicated so far, on the contrary, to Newton maps associated with {\em general}
(as opposed to {\em holomorphic}) maps on the plane. The most noteworthy exception is a series of articles by H.-O. Peitgen
{\em et al.}~\cite{PR86,PPS88,PPS89} where, in particular, they study some case relative to polynomial maps in two variables
with quadratic components. More recently, Newton maps associated to such maps, although in the {\em complex} context, have
been the focus of an article by J.H.~Hubbard and P. Papadopol~\cite{HP08} and of a Ph.D. thesis of Hubbard's pupil
R.K. Roeder~\cite{Roe05}.

The works in the previous paragraph are ideally the roots of the present article. Here, indeed, we classify Newton maps associated
to polynomial maps in two real variables with quadratic components -- as expected, such classification is slightly more articulated
than in the complex case -- and we study numerically their attractors and retractors. In particular, we verify numerically for this
case some conjecture we made in~\cite{DL18} on the behavior of Newton maps of general polynomial homomorphisms of the real plane.
We are also able to prove part of the conjectures in some degenerate case, for example when one of the components has degree 1.
%into itself and supported it with the numerical analysis of several cases. In this articlewe focus on the two most elementary cases
%
\section{Newton maps}
The concept of Newton map arises from the so-called Newton method, which is at the same time the most well-known method
for solving numerically nonlinear systems of equations and one of the most important analytical tools to prove the
existence of solutions to non-linear equations in infinite-dimensional Banach spaces.
%We recall here the main definitions and properties of Newton maps.
%
\begin{definition}
  Let $F:D\subset B_1\to B_2$, be a continuous map between the Banach spaces $B_1$ and $B_2$.
  We call {\em Newton map} associated to $F$ the map $N_F:E\subset B_1\to B_2$ defined by
  $$
  N_F(x)=x-[F'(x)]^{-1}\left(F(x)\right),
  $$
  where $E$ is the subset of $B_1$ over which the Fr\'echet derivative $F'(x)$ exists and is invertible.
\end{definition}
Clearly all zeros of $F$ are fixed points for $N_F$. Newton's method is based on the fact that all such
points are attracting (in fact, super-attracting when $F$ is a Morse function), so that the iterates of
any point ``close enough'' to a root of $F$ will converge to it:
\begin{theorem}[Kantorovich, 1949]
  Let $F:D\subset B_1\to B_2$, be a continuous map between the Banach spaces $B_1$ and $B_2$
  %  :D\subset B_1\to B_2$ be a map betweeen two Banach spaces $B_1$ and $B_2$
  and assume that:
  \begin{enumerate}
  \item $F$ is Fr\'echet differentiable over some open convex subset $D_0\subset D$;
  \item $F'$ is Lipschitz on $D_0$.
  \end{enumerate}
  Then, for every $x_0\in D_0$ such that $[F'(x_0)]^{-1}$ exists and is defined on the whole $B_2$, there is a
  neighborhood $S$ of $x_0$ and an isolated root $x_*$ of $F$ such that $\lim_{n\to\infty}N_F^n(x)=x_*\in S$
  for every $x\in S$.
\end{theorem}
As it often happens in Mathematics, Newton's method was not really introduced by Isaac Newton but rather arose naturally
from the works of several authors, including of course Newton himself -- the interested reader is referred to the review
article by T.J. Ypma~\cite{Ypm95} for a detailed discussion on the history of this method.
According to Ypma, the first appearance in literature of the Newton method in the form above, formulated for to the one-dimensional
real case only, goes back to the {\em Trait\'e de la r\'esolution des \'equations num\'eriques}~\cite{Lag98}, published by Lagrange
in 1798, while the above generalization to Banach spaces by Kantorovich arrived only 151 years later~\cite{Kan49}.

The point of view of discrete dynamics is, in some sense, complementary with respect to the Kantorovich theorem, namely
its main question is what happens to the iterates under $N_F$ when the starting point is {\em far enough} from every root of $F$.
The following concepts and theorems are central in this regard.
\begin{definition}
  Let $M$ be a compact manifold with a measure $\mu$ and $f$ a surjective continuous map of a manifold $M$ into itself.
  %We denote by $f^n$ the $n$^{th} iterate of $f$
  The $\omega$-limit of a point $x\in M$ under $f$ is the (closed) set of all points to whom $x$ iterate are eventually close, namely
  %accumulation points of the sequence $\{f^n(x)\}$, namely
  $$
  \omega_f(x) = \bigcap_{n\geq0}\;\overline{\bigcup_{m\geq n}\{f^m(x)\}}\,,
  $$
  while its $\alpha$-limit is the (closed) set of points to whom $x$ iterated counterimages are eventually close, namely
  %accumulations points of the sequence $\{f^{-n}(x)\}$, namely
  $$
  \alpha_f(x) = \bigcap_{n\geq0}\;\overline{\bigcup_{m\geq n}\{f^{-m}(x)\}}\,.
  $$
  The $\omega$ and $\alpha$ limits of a set is the union of the $\omega$ and $\alpha$ limits of all of its points.
  The {\em forward (\hbox{\sl resp.} backward) basin} $\cF_f(C)$ (resp. $\cB_f(C)$) under $f$ of a closed invariant subset $C\subset M$ is the set
  of all $x\in M$ such that $\omega_f(x)\subset C$ (resp. $\alpha_f(x)\subset C$).
  Following Milnor~\cite{Mil85}, we say that a closed subset $C\subset M$ is an {\em attractor} (resp. {\em repellor}) for $f$ if:
  \begin{enumerate}
    \item $\cF_f(C)$ (resp. $\cB_f(C)$) has strictly positive measure;
    \item there is no closed subset $C'\subset C$ such that $\cF_f(C)$ (resp. $\cB_f(C)$) coincides with $\cF_f(C')$ (resp. $\cB_f(C')$) up to a null set.
  \end{enumerate}
\end{definition}
\begin{definition}
  Given a compact manifold $M$ and a continuous map $f:M\to M$, the {\em Fatou set} $F_f\subset M$ of $f$ is the largest open set
  over which the family of iterates $\{f^n\}$ is {\em normal}, namely the largest open set over which there is a subsequence of the iterates of
  $f$ that converges locally uniformly. The complement of $F_f$ in $M$ is the {\em Julia set} $J_f$ of $f$.
  Finally, we denote by $Z_f$ the set of points $x\in M$ where the Jacobian $D_xf$ is degenerate.
\end{definition}
\begin{theorem}[Julia, Fatou, 1918]
  \label{thm:JF}
  Let $f$ be a rational complex map in one variable of degree larger than 1. Then:
  \begin{enumerate}
  \item $F_f$ contains all basins of attractions of $f$;
  \item both $J_f$ and $F_f$ are forward and backward invariant and $J_f$ is the smallest
    closed set with more than 2 points with such property;
%  \item $J_f$ is the smallest forward and backward closed set with more than 2 points;
  \item $J_f$ is a perfect set;
  \item $J_f$ has interior points iff $F_f=\emptyset$; % (e.g. $F_f=\emptyset$ for $f(x)=\left((z-2)/z\right)^2$, see Corollary 6.2.4 in~\cite{HP88});
  \item $J_f=J_{f^n}$ for all $n\in\bN$;
  \item $J_f$ is the closure of all repelling cycles of $f$;
  \item there is an open dense $U\subset J_f$ s.t. $\omega_f(z)=J_f$ for all $z\in U$;
  \item for every $y\in J_f$, $J_f=\overline{\{x\in\CPo\,|\;f^k(x)=y\hbox{ for some $k\in\bN$}\}}$;
  \item for any attracting periodic orbit $\gamma$, we have that $\partial\cF_f(\gamma)=J_f$.
  \item the dynamics of the restriction of $f$ to its Julia set is highly sensitive to the initial conditions, namely $f|_{J_f}$ is {\em chaotic}.
  \end{enumerate}
\end{theorem}
\begin{definition}
  A Iterated Function System (IFS) $\cI$ on a metric space $(X,d)$ is a semigroup generated by some finite number of continuous functions
  $f_i:X\to X$, $i=1,\dots,n$. We say that $\cI$ is {\em hyperbolic} when the $f_i$ are all contractions. The Hutchinson operator associated to
  $\cI$ is defined as $\cH(A)=\cup_{i=1}^n f_i(A)$, $A\subset X$.
\end{definition}
\begin{theorem}[Hutchinson~\cite{Hut81}, 1981; Barnsley and Demko~\cite{BD85}, 1985]
  Let $\cI$ be a hyperbolic IFS on $X$. Then there exists a unique non-empty compact set $K\subset X$ such that $\cH(K)=K$. Moreover,
  $\lim_{n\to\infty}\cH^n(A)=K$ for every non-empty compact set $A\subset X$.
\end{theorem}
\begin{theorem}[Barnsley, 1988]
  \label{thm:Bar}
  Let $(Y,d)$ be a complete metric space and $X$ a compact non-empty proper subset of $Y$.
  Denote by $\cK(X)$ the set of the non-empty compact subsets of $X$ endowed with the Hausdorff distance $h$
  (recall that $h$ makes $\cK(X)$ a complete metric space).
  Assume that one of the following conditions is satisfied:
%  $f$ is either a continuous open map $X\to Y$ such that  $f(X)\supset X$ or a continuous open map $Y\to Y$
  \begin{enumerate}
    \item $f:X\to Y$ is an open map such that $f(X)\supset X$;
    \item $f:Y\to Y$ is an open map such that $f(X)\supset X$ and $f^{-1}(X)\subset X$.
  \end{enumerate}
%  such that both $f(X)\supset X$ and $f^{-1}(X)\subset X$,
  Then the map $F:\cK(X)\to\cK(X)$ defined by $F(K) = f^{-1}(K)$ is continuous, $\{F^n(K)\}$ is a Cauchy sequence, its
  %(closed, compact, invariant)
  limit $K_0=\lim F^n(X)\in\cK(X)$ is a repellor for $f$ and it is equal to the set of points that never leave $X$ under
  the action of $f$.
\end{theorem}
\medskip

The main result on the dynamics of real Newton maps is the following result by B. Barna~\cite{Bar53} in one dimension:
\begin{theorem}[Barna, 1953]
  Let $p$ be a generic real polynomial of degree $n\geq 4$ without complex roots and denote its roots by $r_1,\dots,r_n$.
  Then:
  \begin{enumerate}
  \item $F_{N_p}=\cup_{i=1}^n\cF(c_i)$;
  \item $F_{N_p}$ has full Lebesgue measure;
  \item $N_p$ has no attractive $k$-cycles with $k\geq2$;
  \item $N_p$ has repelling $k$-cycles of any order $k\geq2$;
  \item $J_{N_p}$ is equal, modulo a countable set, to a Cantor set $\cE_{N_p}$ of Lebesgue measure zero.
  \end{enumerate}
\end{theorem}
In 1984 this result was independently generalized by several authors in three different directions:
\begin{theorem}[Wong, 1984~\cite{Won84}]
  A sufficient condition for Barna's theorem to hold is that the polynomial $p$ has no complex root and at least 4 distinct real roots,
  possibly repeated.
\end{theorem}
\begin{theorem}[Saari and Urenko, 1984~\cite{SU84}]
  Let $p$ be a generic polynomial of degree $n\geq3$, $A_p$ the collection of all bounded intervals in $\bR\setminus Z_{p}$
  and $\cA_p$ the set of all sequences of elements of $A_p$. Then the restriction of $N_p$ to the Cantor set $\cE_{N_p}$ is semi-conjugate to the
  one-sided shift map $S$ on $\cA_f$, namely there is a surjective homomorphism $h_p:\cE_{N_p}\to A_p$ such that $T\circ h_p=h_p\circ N_p$.
\end{theorem}
\begin{theorem}[Hurley and Martin, 1984~\cite{HM84}]
  Let $p$ be a generic polynomial of degree $n\geq3$. Then $N_p$ has at least $(n-2)^k$ $k$-cycles for each $k\geq1$ and the
  topological entropy of $N_p$ is at least $\log(n-2)$.
\end{theorem}
In our knowledge, no further generalization of Barna's result has been published since the three above and no general theorem has
been proved for Newton maps coming from real maps in more than one dimension. In a recent article~\cite{DL18},
based on several numerical observations and in analogy with Barna's theorem, we posed the following ``simple dynamics''
conjectures for the 2-dimensional case:
\begin{definition}
  We say that a point $p$ of the Julia set $J_F$ of a rational map $F:\Rt\to\Rt$ is {\em regular} if there is a neighborhood
  $U$ of $p$ such that $J_F\cap U$ is a connected 1-dimensional submanifold and $U$ contains points from two different basins.  
\end{definition}
\begin{conjecture}
  \label{conj:alpha}
  Let $f:\Rt\to\Rt$ be a generic polynomial map of degree $n\geq3$.
  Then there is some non-empty open subset $A\subset f(\RPt)$ such that $\alpha_{N_f}(x)$ is equal to the set of non-regular points
  of the boundary of $J_{N_f}$ for all $x\in A$.
\end{conjecture}
\begin{conjecture}
  \label{conj:J}
  Let $f:\Rt\to\Rt$ be a polynomial map of degree $n\geq3$ with $n$ distinct real roots $\{c_i\}$. Then:
%  and let $N_f:\RPt\to\RPt$ be its Newton map. Then:
  %
  \begin{enumerate}
  \item $J_{N_f}$ is the countable union of wedge sums of countable number of circles and of Cantor sets of circles of measure zero;
  \item $F_{N_f}$ has no wandering domains;
  \item the union of the basins of attraction $\cF_{N_f}(c_i)$ has full Lebesgue measure;
%    union of a a piecewise smooth Cantor set of Lebesgue measure zero (so, in particular, $N_f$ has no wandering domains);
%  \item the complement of $\cB(J_{N_f})$ has Lebesgue measure zero;
%  \item $\cB(J_{N_f})=\RPt\setminus\{r_i\}$;
  \item every neighborhood of any point of $J_{N_f}$ contains points from at least two distinct basins of attractions;
  \item unlike the holomorphic case:
    \begin{enumerate}
    \item basins of attractions are not necessarily simply connected;
    \item immediate basins of attraction are not necessarily unbounded;
    \item $J_{N_f}$ can have interior points without being equal to the whole $\RPt$.
    \end{enumerate}
%  \item  $J_{N_f}$ is connected -- equivalently, every connected component of $F_{N_f}$ is simply connected;
%  \item the {\em immediate basins} of the roots of $f$, unlike in the complex case, are not necessarily unbounded.
  \end{enumerate}
\end{conjecture}
In the next two sections we study numerically in some detail the dynamics of the Newton maps of two of the simplest
families  of polynomial homomorphisms of the plane. Our results fully support both our conjectures.
In the remainder of this section we recall some well known properties of the Newton maps (e.g. see~\cite{HP08,Hub15}) that
we will use below.
%that we present below in local coordinates (we assume throughout the Einstein convention of summation over repeated indices).

Let $(x^\alpha)$, $\alpha=1,\dots,n$, be a coordinate system in $\K^n$, where $\K=\bR$ or $\bC$,
and let us denote by $f^\a$, $\a=1,\dots,n$, the components of the map $f:\K^n\to \K^n$, by $\partial_\alpha f^\a$ the
entries of its Jacobian matrix and by $\partial_\a \bar f^\alpha$ those of the Jacobian of $f$'s inverse,
namely the inverse matrix of $(\partial_\alpha f^\a)$. Then, in coordinates, the components of $N_f$ are
given by
$$
N_f^\alpha(x) = x^\alpha - \partial_\a \bar f^\alpha(x) f^\a(x)
$$
%The first is the expression for its Jacobian, for which we provide a proof in coordinates:
%
\begin{proposition}
  The entries of the Jacobian matrix $D_xN_f$ are given by
  $$
  \partial_\gamma N_f^\alpha(x) = \partial_b\bar f^\alpha(x)\,\partial^2_{\beta\gamma} f^b(x)\,\partial_\a \bar f^\beta(x)\,f^\a(x).
  $$
%  $D_xN_f = -(D_xf)^{-1}D^2_xf
\end{proposition}
\begin{proof}
  A direct calculation shows that
  $$
  \partial_\gamma N_f^\alpha
  =
  \partial_\gamma\left(x^\alpha-\partial_\a \bar f^\alpha f^\a\right)
  =
  \delta^\alpha_\gamma - \partial_\gamma\left(\partial_\a \bar f^\alpha\right) f^\a - \partial_\a \bar f^\alpha\partial_\gamma f^\a.
    $$
  Since $\partial_\alpha f^\a\partial_\a \bar f^\beta=\delta_\alpha^\beta$, then
  $\partial^2_{\gamma\alpha}f^\a\partial_\a \bar f^\beta+ \partial_\alpha f^\a\partial_\gamma\left(\partial_\a \bar f^\beta\right) $.
  Hence
  $$
  \partial_\gamma\left(\partial_\a \bar f^\alpha\right)  = - \partial_\a \bar f^\beta\partial^2_{\gamma\beta}f^b\partial_b \bar f^\alpha
  $$
  and the claim follows.
\end{proof}
\begin{corollary}
  The set $C_f=\{x\in KP^n\,|\,\det D_xN_f\neq0\}$ of all degenerate points of $N_f$ is given
  by all points $x$ such that the matrix
  $$
  (A_{\gamma}^b)=\left(\partial^2_{\gamma\beta}f^b(x)\,\partial_\a\bar f^\beta(x)\,f^\a(x)\right)
  $$
  is degenerate. Equivalently, $C_f$ is the set of all $x$ such that either $f(x)=0$ or the vector
  $\partial_b \bar f^\alpha(x)\,f^\a(x)$ is a critical direction of the quadratic map
  $v^\alpha\mapsto\partial^2_{\gamma\beta}f^\a(x)v^\gamma v^\beta$.
\end{corollary}
Points where $DN_f$ is degenerate play an important role in the dynamics of $N_f$,
for example in the one-dimensional case $J_{N_f}$ is equal to $\alpha_f(C_f)$.
Moreover, the fact that $DN_f$ is identically zero at every root of $f$ gives us immediately
the following fundamental property:
\begin{corollary}
  All simple roots of $f$ are super-attractive fixed points for $N_f$.
\end{corollary}
\begin{remark}
  Note that, while in the one-dimensional case the fixed point at infinity of a polynomial is always repelling,
  fixed points at infinity in the multi-dimensional case can be attractive (e.g. see Example~4 in~\cite{DL18}).
\end{remark}
\begin{proposition}
  \label{prop:psi}
  If $\psi,\phi:\K^n\to\K^n$ are, respectively, {\em affine} and {\em linear} automorphisms of $\K^n$,
  namely $\psi^\alpha(x)=A_\beta^\alpha x^\beta+u^\alpha$ and $\phi^\a(y)=B_\a^b y^\a$ for some
  $A,B\in GL_n(\K)$ and $u,v\in\K^n$, then
  $$
  N_{\phi\circ f\circ\psi}=\psi^* N_f\,
  $$
  namely the Newton map $N_f$ has the same dynamics as $N_{\phi\circ f\circ\psi}$.
\end{proposition}
\begin{proof}
  In this case $\partial_\beta\psi^\alpha=A_\beta^\alpha$ and $\partial_b\phi^\a=B_b^\a$, so that
  $$
  \psi^* N_f^\alpha(x)
  =
  \bar A^\alpha_\beta \left[N_f^\beta(Ax+b) - b^\beta\right]
  = 
  $$
  $$
  =
  \bar A^\alpha_\beta \left[A^\beta_\gamma x^\gamma+b^\beta - \partial_b\bar f^\beta\big|_{Ax+b}\,f^b(Ax+b) - b^\beta\right]
  =
  $$
  $$
  =
  \delta^\alpha_\gamma x^\gamma - \bar A_\beta^\alpha\partial_b\bar f^\beta\big|_{Ax+b}\,\bar B_\a^b\,B^\a_c\,f^c(Ax+b) - b^\beta
  =
  $$
  $$
  =
  x^\alpha - \partial_\a(\overline{\phi\circ f\circ\psi})^\alpha(x)\,(\phi\circ f\circ\psi)^\a(x)
  =
  N_{\phi\circ f\circ\psi}^\alpha(x)
  $$
  since 
  $$
  \partial_\a(\overline{\phi\circ f\circ\psi})^\alpha(x)
  =
  \partial_\a(\bar \psi\circ\bar f\circ\bar\phi)^\alpha(x)
  =
  \partial_\a\bar\phi^b\,\partial_b\bar f^\beta\big|_{Ax+b}\,\partial_\beta\bar\psi^\alpha.
%
 % \partial_\beta\bar\psi^\alpha(x)\,  \partial_\a \bar f^\beta(Ax+b)
 % =
 %\bar  A_\beta^\alpha\partial_\a \bar f^\beta(Ax+b).
%  \partial_\a\bar\psi^b\partial_b f^\alpha
%\bar A^\alpha_\beta \partial_\a\bar f^\beta(Ax+b)\,(f\circ\psi)^\a(x)
 % =
  $$
\end{proof}
\begin{definition}
  Let $f=(p,q)\in C^k(\Rt,\Rt)$. We call {\em pencil} $\cP_f$ associated to $f$ the
  linear subspace of $C^k(\Rt)$ generated by $p$ and $q$. If $p$ and $q$ are polynomials,
  by {\em type} of the pencil, denoted by  $\tau(\cP_f)$, we mean the pair of the highest and lowest degrees
  of polynomials in the pencil.
\end{definition}
Thanks to Proposition~\ref{prop:psi}, we can replace the components of a Newton map $N_f$ with any two
independent linear combination of them without altering its dynamics, so what really counts
is the pencil of its components rather than the components themselves.

When the pencil is of type $(m,0)$ it means that the differentials $dp$ and $dq$ are proportional, so that
$Df$ is singular at every point and the Newton map cannot be defined anywhere.
We discuss below the two simplest non-trivial cases, namely the degenerate case $\tau(\cP_f)=(m,1)$, where the Jacobian
of the Newton map is singular at every point, and the case $\tau(\cP_f)=(2,2)$, which is the lowest-degree non-trivial case.
%
%\section{Analysis of two particular cases}
%\label{sec:numerics}
%\section{Numerical results}
%\section{Study of two types of Newton maps on $\Rt$}
%\section{Numerical analysis of two types of Newton maps on $\Rt$}
%\section{Polynomial maps with a linear component}
%
%In this section we present our results on two of the simplest types of Newton
%maps on the plane.
%%For a generic pair $(p,q)$, we have that both components of $\tau(\cF)$ are equal to the
%%highest degree between $p$ and $q$, but 

%\end{document}

%\medskip
%{\bf The case $\boldsymbol{\tau(\cP_f)=(m,1)}$}.
\section{The case $\boldsymbol{\tau(\cP_f)=(m,1)}$}
%Consider the vector space $\cL$ of all polynomial maps $f=(p,q):\Rt\to\Rt$ of degree at least 2,
%namely with either $p$ or $q$ of at least degree 2, and such that there exists a
%linear combination of $p$ and $q$ of degree 1.
%All elements of $\cL$ have Newton maps degenerate at every point: indeed, we can always find an affine map $\psi$ and a
All maps $f$ with $\tau(\cP_f)=(m,1)$ have a Newton map degenerate at every point.
Indeed, we can always find an affine map $\psi$ and a linear map $\phi$ so that
$\phi\circ f\circ\psi(x,y)=(p(x,y),y)$ and so, thanks to Proposition~\ref{prop:psi},
$N_f$ is smoothly conjugate to
$$
N_{\phi\circ f\circ\psi}(x,y) = \left(\frac{xp_x(x,y)+yp_y(x,y)-p(x,y)}{p_x(x,y)},0\right).
$$
In particular, this shows that the image of $N_f$ is one-dimensional. Correspondingly, the
counterimage of a generic point is a finite set of lines. 

Nevertheless, even this simple case gives already rise to non-trivial Julia sets and, most importantly,
we are able to prove for this class of maps Conjecture~2 as a direct consequence of Barna's Theorem:
\begin{theorem}
  Conjecture~2 holds for all generic maps $f:\Rt\to\Rt$ with $\tau(\cP_f)=(m,1)$.
  %Let $f:\Rt\to\Rt$ be a generic element of $\cL$. The all four points of Conjecture~2 hold.
  % and assume that $f$ has the maximum  .
%  (x,y)=(p(a x+b y+k), cy+d y+k')$, where $a,b,c,d,k,k'\in\bR$, $ad-bc\neq0$ and $p$ is a
  %  non-linear polynomial with maximal number of real distinct roots .
%  Then the minimal invariant sets of $N_f$ are has $J_{N_f}\subset\RPt$
%  is a connected piecewise smooth set of Lebesgue measure zero and its repelling basin is the whole
%  $\RPt$ minus the set of the roots of $f$.
\end{theorem}
\begin{proof}
  We can consider without loss of generality the case $f(x,y)=(p(x,y),y)$. Let $\hat p_y(x)=p(x,y)$ and $q(x)=p_0(x)$,
  so that in particular the roots of $q$ are the left coordinates of the roots of $p$.
  A direct calculation shows that $N_f(x,y) = (N_{p_y}(x)+y\partial_y p(x,y)/\partial_x p(x,y),0)$ and, in particular,
  $N_f(x,0) = (N_{q}(x),0)$. By hypothesis, $q$ satisfies the conditions of Barna's theorem and so the
  intersection of $J_f$ with the $x$ axis is a Cantor set of 1-dimensional Lebesgue measure. Given any other
  horizontal line $y=y_0$, the iterates of $(x,y_0)$ under $N_f$ converge to a root $(r_x,r_y)$ of $f$ iff
  the left coordinate $N_f(x,y_0)_x$ of the point $N_f(x,y_0)$ converges to $r_x$ under $N_q$ and, analogously,
  $(x,y_0)\in J_f$ iff $N_f(x,y_0)\in J_q\times\{0\}$. In turn, this means that $J_f$ is the collection of all
  level sets of $N_f(x,y)_x$ passing through $J_q\times\{0\}$, which immediately implies that it is the union
  of a countable set together with a Cantor set of (2-dimensional) zero Lebesgue measure and that each connected
  component of it is piecewise smooth.

  When one of these level sets corresponds to an isolated point of $J_q$,
  then every regular point of that level set has a neighborhood containing points from only two connected components
  of the attracting basins corresponding to different roots. Otherwise, every such neighborhood contains points from
  countably many connected components belonging to at least two different roots.

%  Since all level sets of $N_f(x,y)_x$ are asymptotically vertical and all points at infinity of the plane belong
%  to $J_f$ since the circle at infinity is invariant under the projective extension of $N_f$ to $\RPt$, so $J_f$
%  is connected.
  
%  All level sets of $N_f(x,y)_x$ are asymptotically vertical and $[0:1:0]$ belongs to $J_f$ since the projective
%  extension of $N_f$ is $[x:y:z]\mapsto [z^{d}y\partial_y p(x/z,y/z)+z^d x\partial_x p(x/z,y/z)+z^{d}p(x/z,y/z):z^{d}\partial_x p(x/z,y/z):z^{d}]$,
%  where $d$ is the degree of $p$, so that $[x:y:0]\mapsto[1:0:0]$.
  
 % and so, as a subset of $\Rt$, $J_f=J_p\times\bR$.
 % Since $f$ is generic, we can assume that $p$ has no multiple roots, so that $N_p$ is well-defined on the whole
  %$\RPo$ and $N_f$ is well-defined on the whole $\RPt$. By Barna's Theorem, $J_p$ is a Cantor subset
  %of $\RPo$ of zero Lebesgue measure. For all points on the horizontal line $y=y_0$, the forward orbit
  %of $(x,y_0)$ under $N_f$ coincides, apart from the first element, with the one of $()$
%  and so $J_g$ is also Cantor subset of $\RPt$ of zero Lebesgue measure.
%  The restriction of $J_g$ to $\Rt$ is a union of parallel lines and therefore,
%  after adding the circle at infinity, $J_g$ becomes a Cantor bouquet of circles centered at $[0:1:0]$ and,
 % in particular, it is connected. The (simply) connected components of the (restriction to $\Rt$ of the) Fatou
 % set are $F_g$ all stripes $F_i\times\bR$, where the $F_i$ are the connected components of $F_p$.
\end{proof}
%Conjecture 1 descends trivially from Barna's Theorem in case of polynomial maps of the type $f(x,y)=(p(x),y)$
%for some polynomial in a single variable $p$.
%
%\begin{remark}
 % Just as in the one-dimensional case, when real solutions are missing the union of the basins of attraction
%  might not have full Lebesgue measure. E.g. this happens in case of $f(x,y)=(x^3-2x+2+x r(y),y)$ for any polynomial $r$
%  (see Fig.~\ref{fig:cl} for the case $r(y)=y$).
%\end{remark}
%
%                                                                                                                                                                                           
\begin{figure}
  \centering
  \begin{tabular}{cc}                                                                                                                                                                  
    \includegraphics[width=6.3cm]{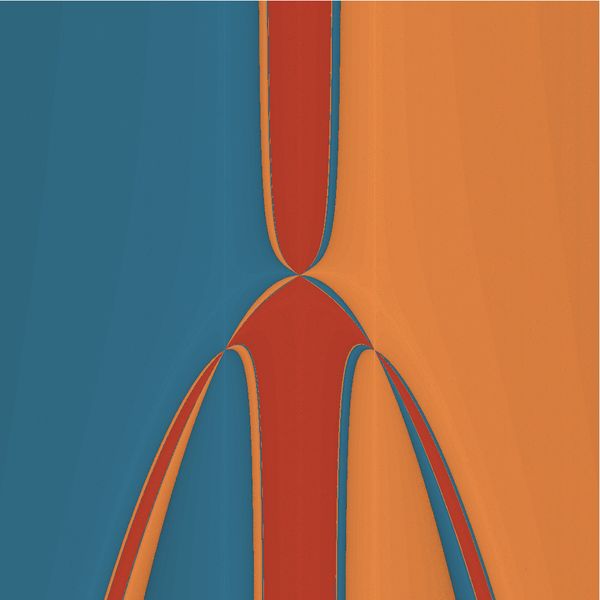}&\includegraphics[width=6.3cm]{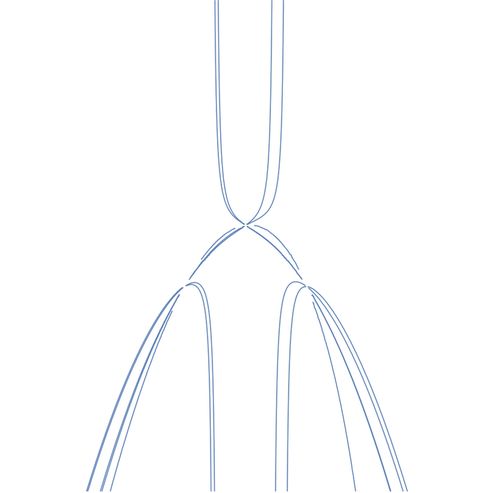}\\
    \includegraphics[width=6.3cm]{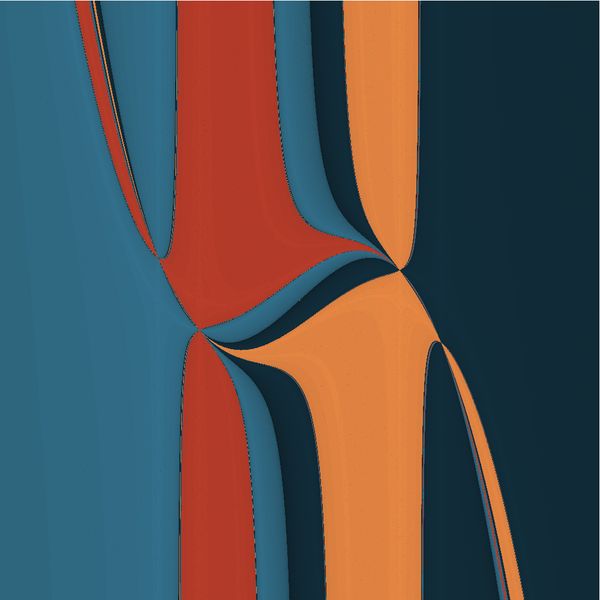}&\includegraphics[width=6.3cm]{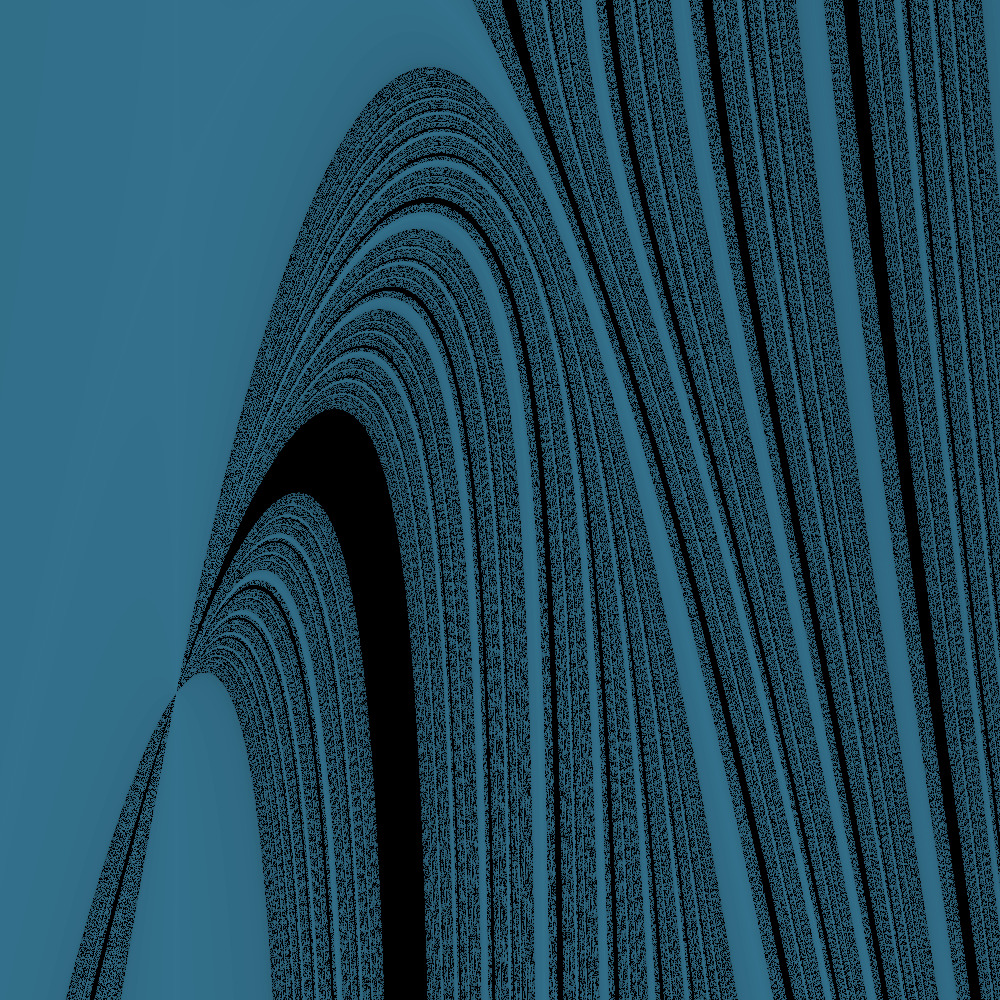}\\
  \end{tabular}
  \caption{%
    \em
    (Top) (left) Basins of attraction of $N_f$, with $f(x,y)=(x^3-x+3xy,y))$, in the square $[-4,4]^2$ and (right) the corresponding
    set $Z_{N_f}$ and a few of its counterimages (right).  Each of the three roots of $f$ corresponds to a different color and all points belonging
    to a root's basin have been given the same color of the root they converge to. The three visible nodes are the indeterminacy points of $N_f$.
    (Bottom, left) Basins of attraction of $N_g$, with $g(x,y)=(x^4-3x^2+xy+2,y)$. In this case there are four roots and four indeterminacy
    points, corresponding to the four colors and nodes in the picture. (Bottom, right)  Basin of attraction of $N_h$, with
    $h(x,y)=(x^3+xy-2x+2,y)$. This map has a single root and its basin, as the picture clearly suggests, is not of full measure because of the
    presence of an attracting $2$-cycle. 
%    On the contrary, $g$ has a single real root and its Julia set has measure larger than zero. In fact, each transversal section of $J_g$ is
%    homeomorphic to the Julia set of the  polynomial $p(x)=x^3-2x+2$, whose (1-dimensional) Lebesgue measure is known to be larger
%    than 0.
  }
  \label{fig:cl}
\end{figure}
Note that, unlike in the one-dimensional case, in the two dimensional case a generic Newton map will not extend to the whole $\RPt$
because at some finite number of points both the numerator and denominator will go to zero. These are called {\em points of indeterminacy}.
%
%\begin{example}

%{\bf In Fig.~\ref{fig:cl} (top row) we show two pictures relative to the case of the cubic polynomial map $\boldsymbol{f(x,y)=(x^3+3xy-x,y)}$},
\subsection{$\boldsymbol{f(x,y)=(x^3+3xy-x,y)}$}
In Fig.~\ref{fig:cl} (top row) we show two pictures relative to the case of the cubic polynomial map $\boldsymbol{f(x,y)=(x^3+3xy-x,y)}$,
with roots $(0,0)$, $(\pm1,0)$,
%Consider the cubic polynomial map $f(x,y)=(x^3-x+3xy,y)$,
whose Newton map is
  $$
  N_f(x,y)=\left(x\frac{2x^2+3y}{3x^2-1+3y},0\right).
  $$
  In the left picture we show the basins of attraction of $N_f$ in the square $[-4,4]^2$; 
  the three immediate basins of attraction are the largest visible basin for each of the three colors and are clearly unbounded. 
  The extension of $N_f$ on $\RPt$, in homogeneous coordinates, is the map
  $$
  [x:y:z]\mapsto[x(2x^2+3yz):0:z(3x^2-z^2+3yz)]
  $$
  that restricts at infinity to the constant map $[x:y:0]\mapsto[1:0:0]$ at all points except $[0:1:0]$, which is a point of indeterminacy.
  $N_f$ has exactly other three points of indeterminacy in $\Rt$, namely the real roots of the system of sixth degree
  $x(2x^2+3y)=0, 3x^2-1+3y=0$, whose coordinates are $(0,1/3)$, $(-1,-2/3)$ and $(1,-2/3)$.
  These four points correspond exactly to the nodal points of $J_f$, three of which are visible in
  Fig.~\ref{fig:cl} (top, left).

  The set of all counterimages of the point $(1.5,0)$ up to the third recursion level is shown in Fig.~\ref{fig:cl} (top, right)
  and clearly suggests that Conjecture~1 holds for this type of maps. Notice that all smooth connected components of $J_{N_f}$
  are segments of cubic polynomials asymptotic to the vertical direction, corresponding to the fact that they all meet at the
  nodal point at infinity $[0:1:0]$.
%\end{example}
%In the general case of a map $f(x,y)=(p(x,y),ax+by+c)$, $J_f$ is not a direct product and $N_f$ cannot
%be extended to the whole $\RPt$ because of the appearance of {\em points of indeterminacy}.
%An important feature

\subsection{$\boldsymbol{g(x,y)=(x^4-3x^2+xy+2,y)}$}
In Fig.~\ref{fig:cl} (bottom, left) we consider the case of the quartic polynomial map $\boldsymbol{g(x,y)=(x^4-3x^2+xy+2,y)}$,
with roots $(\pm1,0)$, $(\pm\sqrt{2},0)$, whose Newton map is
  $$
  N_g(x,y)=\left(\frac{3x^4-3x^2+xy-2}{4x^3-6x+y},0\right).
  $$
  Its projective extension in homogeneous coordinates writes
  $$
  [x:y:z]\mapsto[3x^4-3x^2z^2+xyz^2-2z^4:0:z(4x^3-6xz^2+yz^2)]
  $$
  and $[0:1:0]$ is a point of indeterminacy. The other fours points of indeterminacy of $N_g$ are
  $(-1,-2)$, $(1,2)$, $(-\sqrt{2},2\sqrt{2})$ and $(\sqrt{2},-2\sqrt{2})$.
  In the picture we show the basins of attraction in the rectangle $[-3,3]\times[-20,20]$. Even in this case
  the immediate basins are the largest of their own color visible in the picture and are unbounded and all
  unbounded branches of $J_{N_g}$ meet at the nodal point at infinity.

%  Finally, {\bf in Fig.~\ref{fig:cl} (bottom, right) we consider the case of the cubic polynomial map $\boldsymbol{h(x,y)=(x^3+xy-2x+2,y)}$}.
  \subsection{$\boldsymbol{h(x,y)=(x^3+xy-2x+2,y)}$}
  \label{ss:ac}
  In Fig.~\ref{fig:cl} (bottom, right) we consider the case of the cubic polynomial map $\boldsymbol{h(x,y)=(x^3+xy-2x+2,y)}$.
  This map has a single zero at about $(-1.77,0)$ and its Newton map is
  $$
  N_h(x,y)=\left(\frac{2x^3+xy-2}{3x^2+y-2},0\right).
  $$
  $N_h$ has only two points of indeterminacy: the point at infinity $[0:1:0]$ and, approximately, the point $(-1.77:-7.39)$.
  Unlike the two cases above, and correspondingly to what happens in case of the complex polynomial $p(z)=z^3-2z+2$,
  the basin of attraction of $h$'s only root has not full Lebesgue measure because of the presence of the super-attracting $2$-cycle
  $\{(0,0),(1,0)\}$.
  \begin{remark}
    In fact, this property of $p$ is universal in the sense that, if $q$ is a complex polynomial of degree
    three such that $N_q$ has a super-attracting 2-cycle, then $N_q$ is conjugated to $N_p$~\cite{CC13}.
  \end{remark}

  \medskip
   {Whether Conjecture~1 holds or not for these maps is a one-dimensional problem, namely it is true if and only
   if the same happens for real Newton maps in one variable. We are not aware of any results in this direction to date.}

%  In Fig.~\ref{fig:cl} we show some numerical study 
%An important fact to keep in mind while studying numerically the Julia set a function is that, in general,
   %
   %  \medskip
   %{\bf The case $\boldsymbol{\tau(\cP_f)=(2,2)}$}.
   \section{The case $\boldsymbol{\tau(\cP_f)=(2,2)}$}
%\subsection{Polynomial maps with both components quadratic}
%
Now we consider the case $f=(p,q)$ where both $p$ and $q$ are quadratic polynomials
in general position. Hubbard and Papadopulos~\cite{HP08} showed that, in the complex case, one can reduce this problem
to the study of the Newton maps $N_g$ with
%we can always find a pair $\phi,\psi$ such that $\phi\circ f\psi(x,y)=
$$
g(x,y)=(y-x^2,(x-x_0)-(y-y_0)^2),\, (x_0,y_0)\in\bC^2.
$$
In the real case we have, instead, two distinct cases:
\begin{proposition}
  Let $f(x,y)=(p(x,y),q(x,y)):\Rt\to\Rt$, with $p$ and $q$ quadratic polynomials in generic position.
  Then $N_f$ is smoothly conjugated to $N_g$, where $g(x,y)$ has one of the following two forms:
  \begin{enumerate}
    \item $\ff_{x_0,y_0}(x,y)=(y-x^2,(x-x_0)-(y-y_0)^2)$, $x_0\geq0$, $y_0\in\bR$;
    \item $\ff_{x_0,y_0;a}(x,y)=(xy-1,(x-x_0)^2-a(y-y_0)^2-1)$, $y_0\geq x_0\geq0,\,a>0$.
  \end{enumerate}
\end{proposition}
\begin{proof}
  We say that a quadratic polynomial is elliptic, hyperbolic or parabolic depending on the type of its generic level set.
  Through an affine transformation, we can always reduce $p$ to one of the following quadratic polynomials:
  \begin{enumerate}
    \item $\bc(x,y)=x^2+y^2-1$ (elliptic);
    \item $\bh(x,y)=xy-1$ (hyperbolic);
    \item $\bp(x,y)=y-x^2$ (parabolic).
  \end{enumerate}
  When $p$ is elliptic, given the rotational symmetry of the circle, we can always find a second affine transformation
  that reduces $q$ to one of the following quadratic polynomials:
  \begin{enumerate}[label=\alph*.]
  \item $q_e(x,y)=a(x-x_0)^2+b(y-y_0)^2-1$, with $y_0\geq x_0\geq0$ and $a>b>0$;
  \item $q_h(x,y)=a(x-x_0)^2-b(y-y_0)^2-1$, with $y_0\geq x_0\geq0$ and $a>b>0$;
  \item $q_p(x,y)=y-y_0-a(x-x_0)^2$, with $a>0$, $x_0\geq0$ and $y_0\in\bR$.
  \end{enumerate}
  In the first two cases we can easily find two linear combinations that give us two parabolic polynomials,
  in the last case we can use $q_p$ to cancel one of the two quadratic terms of $\bc$. In all three cases
  then we can find an affine map $\psi$ and a linear map $\phi$ such that $f=\phi\circ\hat f\circ\psi$
  where both components of $\hat f$ are parabolic. With yet another affine transformation, we can finally
  transform $\hat f$ in the final form $\ff_{x_0,y_0}$, namely $N_f$ is conjugated to $N_{\ff_{x_0,y_0}}$.

  When $p$ is hyperbolic, we can assume that $q$ is either parabolic or hyperbolic (the elliptic case is covered
  above). In the first case, $q(x,y)=ax^2+2bxy+cy^2+\ell(x,y)$, where $\ell$ is linear and $ac-b^2=0$.
  by adding $2\lambda xy$ to $q$ we get a parabola iff $ac-(b+\lambda)^2=0$, namely $\lambda=0,-2b$,
  so again we reduce to the case of two parabolic polynomials and so to $\ff$.
  In the second case, $q(x,y)=ax^2+2bxy+cy^2+\ell(x,y)$ with $ac-b^2<0$. After adding $2\lambda xy$ to $q$,
  its type is determined by the sign of $d(\lambda)=ac-(b+\lambda)^2=ac-b^2-2b\lambda-\lambda^2$. Since
  the discriminant of $d(\lambda)$ is $4ac$, the sign of $d$ is strictly negative for all $\lambda$ when $a$
  and $c$ have opposite signs. In this case, therefore, all linear combinations of $p$ and $q$ are hyperbolic
  and $N_f$ is conjugate to some $N_{\ff_{x_0,y_0;a}}$.
\end{proof}
\begin{figure}
  \centering
    \includegraphics[width=13.5cm]{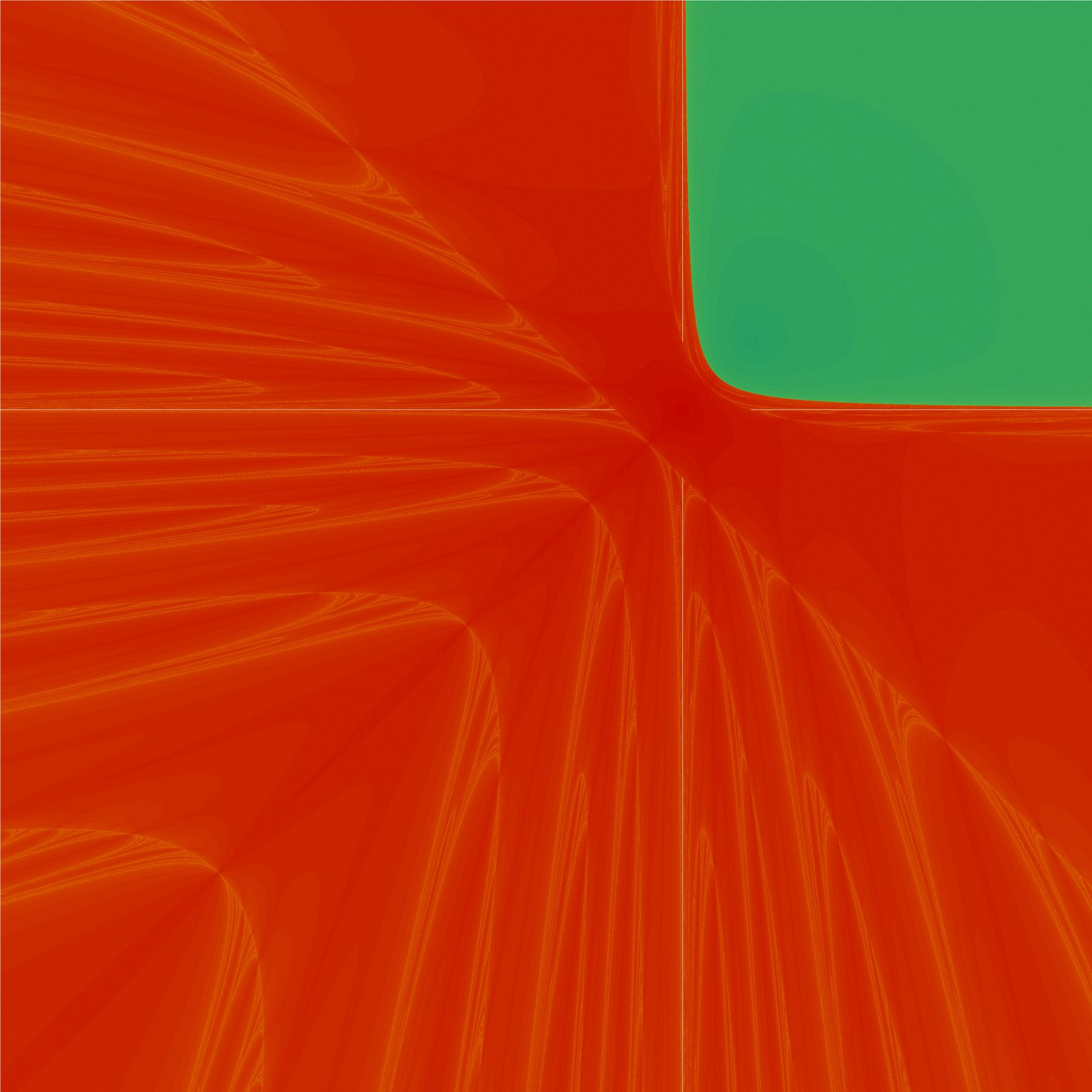}
  \caption{%
    \em
    Basins of attraction of $\fN_{0,0}$ 
%    the Newton map of $f(x,y)=(y-x^2,x-y^2)$
    in the rectangle $[-6,10]\times[-10,6]$.
    The only real roots of $\ff_{0,0}$ are the points $(0,0)$ (whose basin is colored in red) and $(1,1)$ (in green).
    The yellow points are those for which the Newton method converges only after a large number of iterations --
    as the numerical and analytical data shown in Fig.~\ref{fig:qqa1} suggest, these are all  points belonging to
    the counterimages of $\fZ_{0,0}$. Lighter tints correspond to higher convergence time. 
  }
  \label{fig:qqa}
\end{figure}
\begin{figure}
  \centering
  \begin{tabular}{ccc}                                                                                                                                                                  
    \includegraphics[width=4.25cm]{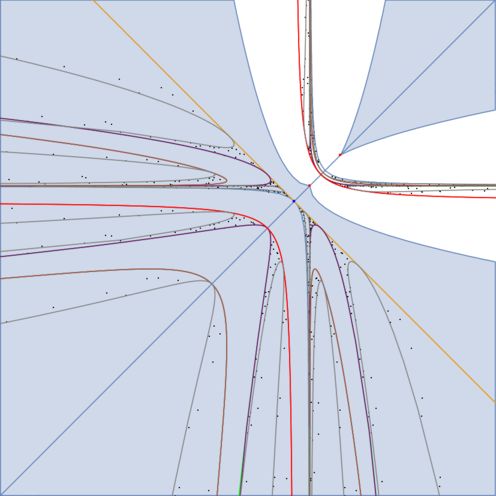}&\includegraphics[width=4.25cm]{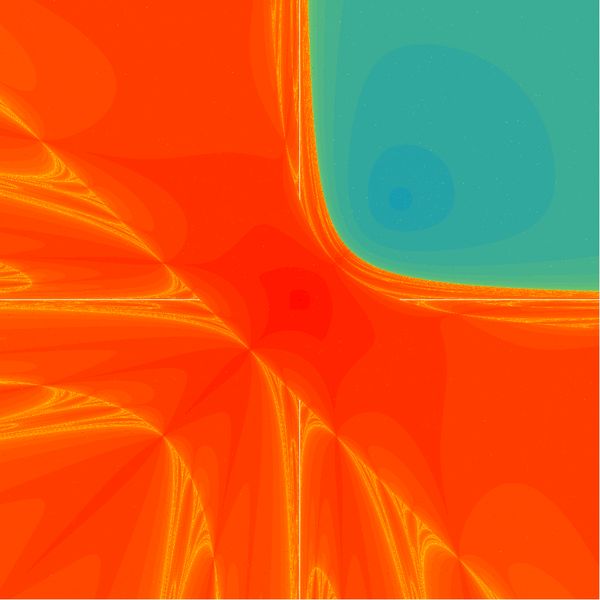}&\includegraphics[width=4.25cm]{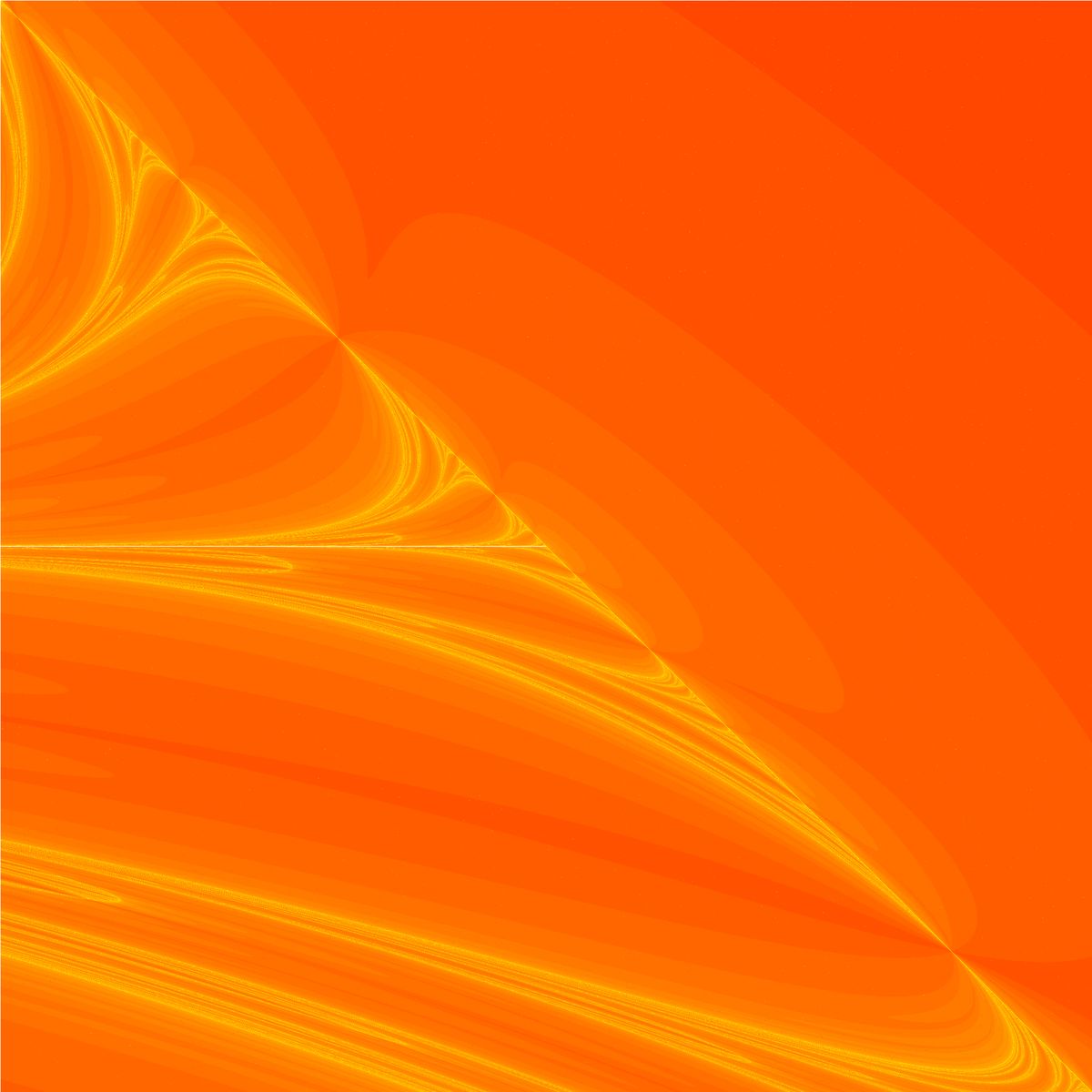}\\
    \includegraphics[width=4.25cm]{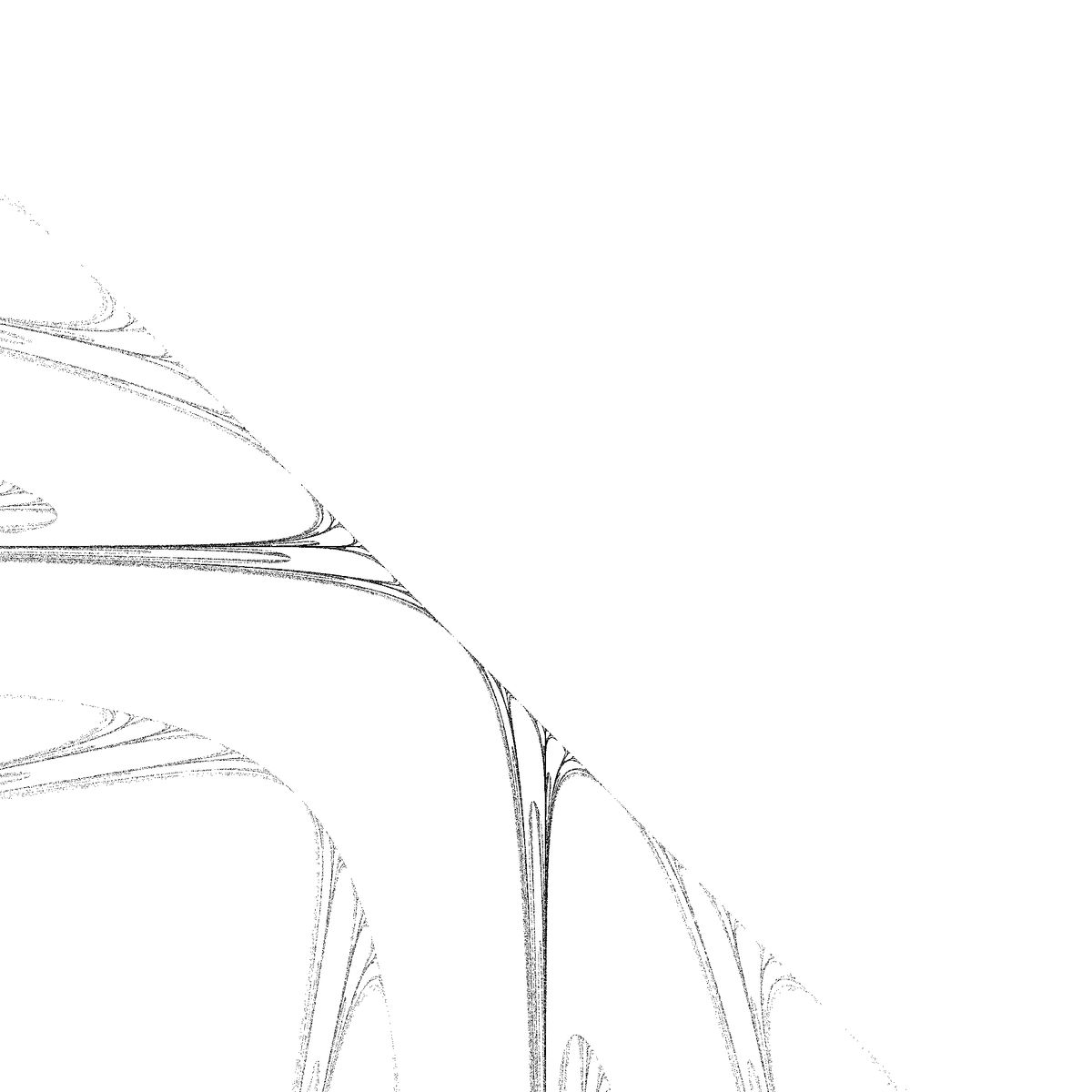}&\includegraphics[width=4.25cm]{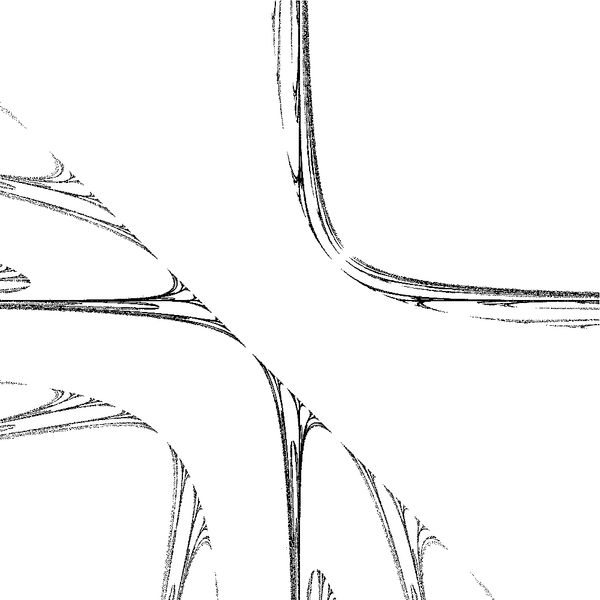}&\includegraphics[width=4.25cm]{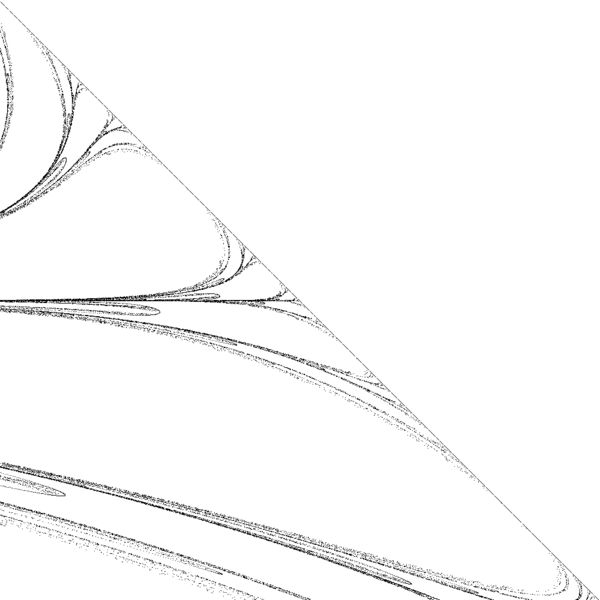}\\
    \includegraphics[width=4.25cm]{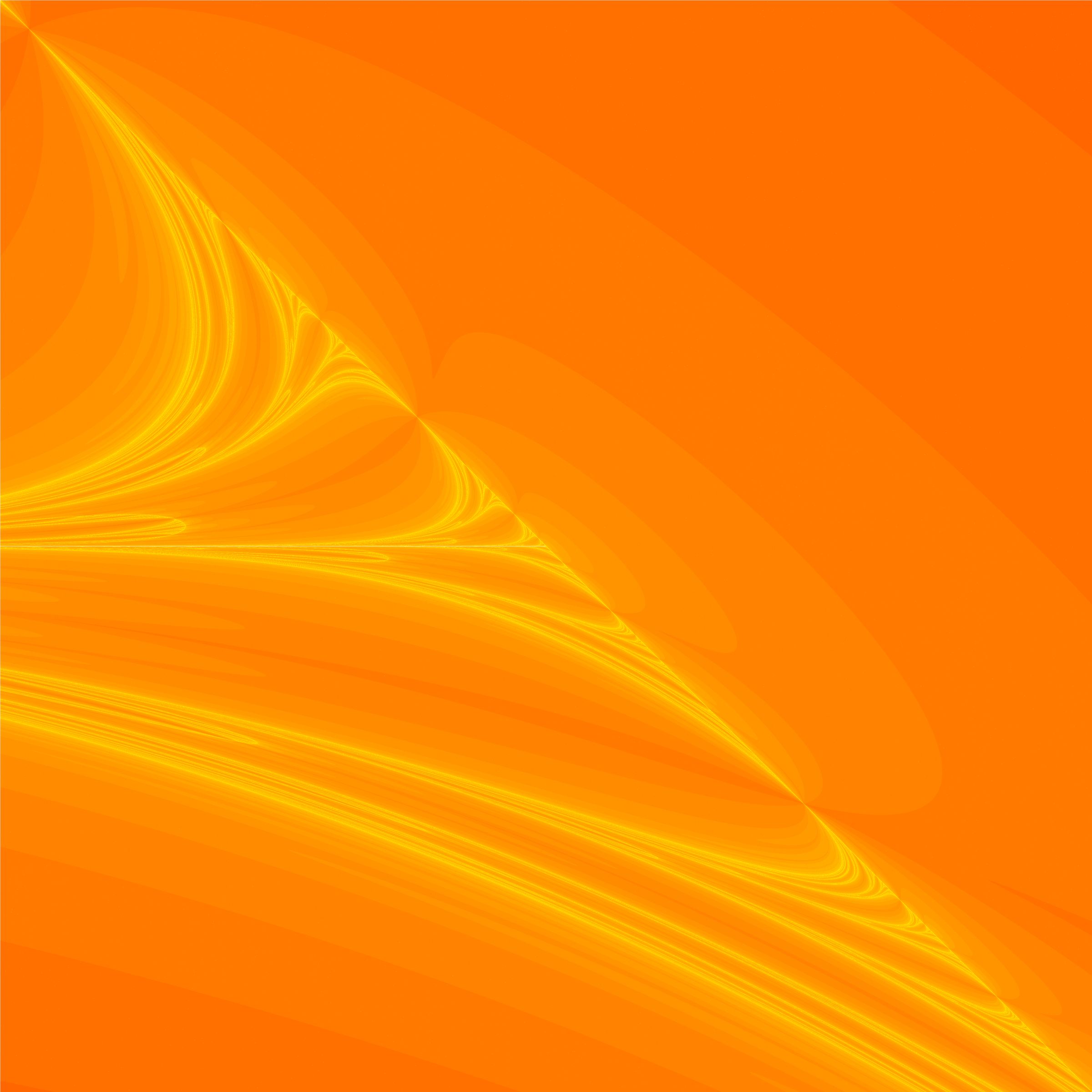}&\includegraphics[width=4.25cm]{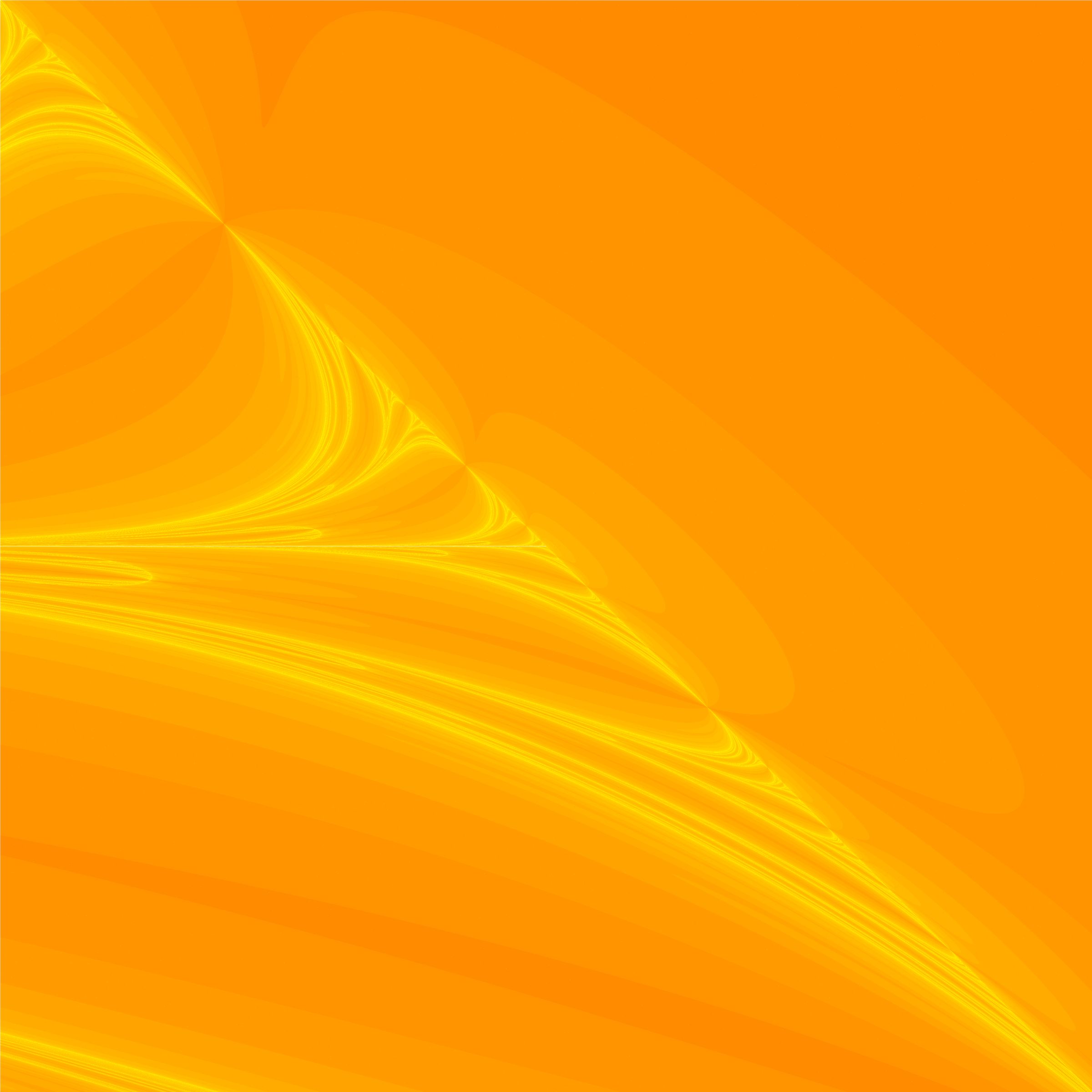}&\includegraphics[width=4.25cm]{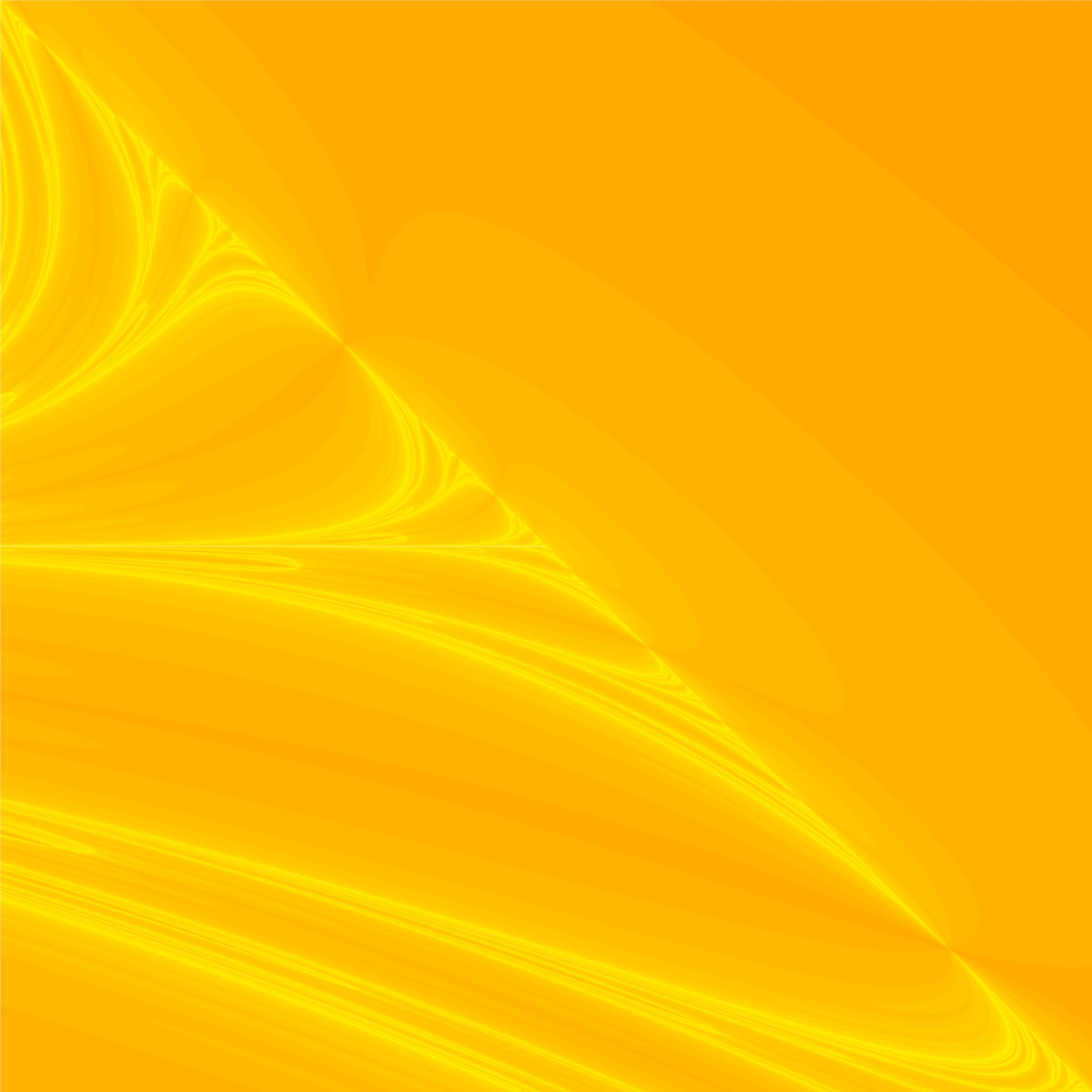}\\
  \end{tabular}
  \caption{%
    \em
    (Top, left) Fixed points (red), indeterminacy point (blue) and invariant lines of $\fN_{0,0}$ and set $\fZ_{0,0}$ together with some of their
    counterimages: the ghost line is in orange and its first and second counterimages in red and brown resp., $\fZ_{0,0}$ in
    light blue and its counterimages in purple and gray. The black points are the counterimages via $\fN_{0,0}$ of a single point of the plane
    up to the 5th level of iteration. (Top, center \& right) Basins of attraction and Julia set of $\fN_{0,0}$ in the rectangles
    $[-3,3]^2$ and $[0.95,1.05]\times[-0.05,0.05]$.
    (Center) Approximations of $\fJ_{0,0}$ obtained as the first $3\cdot10^5$ points of a random backward orbit
    (left) and as the points of the set $\fN_{0,0}^{-11}(1,-3.6)$ (center, right).
    (Bottom) Basins and Julia set of $\fN_{0,0}$ in the squares of side respectively $10^{-3}$ (left), $10^{-4}$ (center)
    and $10^{-5}$ centered at $(-1,0)$. Lighter colors correspond to higher convergence time. 
  }
  \label{fig:qqa1}
\end{figure}
\begin{definition}
  We denote respectively by $\fN_{x_0,y_0}$ and $\fN_{x_0,y_0;a}$ the Newton maps associated with $\ff_{x_0,y_0}$ and $\ff_{x_0,y_0;a}$
  and, analogously, we use the symbols $\fJ_{x_0,y_0}$, $\fJ_{x_0,y_0;a}$, $\fZ_{x_0,y_0}$ and $\fZ_{x_0,y_0;a}$ for their Julia sets and for
  the sets of degeneracy of their Jacobians.
\end{definition}
Below we present our numerical analysis of some Newton maps $\fN_{x_0,y_0}$ and $\fN_{x_0,y_0;a}$ and compare
it against our conjectures.
%In Fig.~\ref{fig:qqa} through~\ref{fig:qq3} we show the basins of attraction and the Julia sets
%of the Newton maps of several polynomial maps $f$ of type $(2,2)$.
Our pictures are of two main types: basins of attractions, obtained by considering a regular lattice of
points inside a square (of the order of $10^3\times10^3$ points), evaluating the iterates of all such points
under $N_f$ and assigning to them different colors depending to which root of $f$ they appear to converge; 
%, shaded according to the number of iterations needed to get ``close enough'' to the corresponding root;
basins of retraction, obtained by choosing an arbitrary point {\em within a suitable open set},
evaluating its backward iterates until some fixed recursion order (usually between 10 and 20)
and plotting all points obtained at the last recursion step.

%
%{\bf In Fig.~\ref{fig:qq1} we consider the case of $\boldsymbol{f(x,y)=(y-x^2,x-y^2)}$.}
%\subsection{$\boldsymbol{f(x,y)=(y-x^2,x-y^2)}$}
\subsection{The map $\boldsymbol{\fN_{0,0}}$}%(x,y)=(y-x^2,x-y^2)}$}
\label{sec:f00}
The Newton map
$$
\fN_{0,0} = \left(y\frac{2x^2+y}{4xy-1},x\frac{2y^2+x}{4xy-1}\right)
$$ 
was first studied, in this context, in great detail by Peitgen, Pr\"ufer and Schmitt~\cite{PPS88,PPS89}
and later, in the complex setting, by Hubbard and Papadopol~\cite{HSS01}.
%We believe that this map should
%receive much more attention than it had so far but in this article we just describe in short some
%of its main known properties.
%The corresponding Newton map is
%Here we show some high definition picture of the Julia set
%Given the last thirty years' enormous improvements in computing hardware, we were available to obtain a definition much higher than the one avaialable
%in literature so far for the Julia set of the corresponding Newton map
In homogeneous coordinates, this map writes as %the extension of $N_f$ to $\RPt$ is the map
$$
[x:y:z]\to[y(2x^2+yz):x(2y^2+zx):z(4xy-z^2)],
$$
from which it can be seen that $\fN_{0,0}$ has the following three points of indeterminacy: one bounded, $[1:1:-2]$,
and two at infinity, $[1:0:0]$ and $[0:1:0]$.
%On the hyperbola $Z_{N_f}=\{4xy-z^2=0\}$ (in light blue in Fig.~\ref{fig:qqa1} top, left), lie
%three points of indeterminacy of $N_f$, namely $[1:1:-2]$ and the two points at infinity  $[1:0:0]$ and $[0:1:0]$.
%Notice that all these three points belong to the hyperbola $\fh=\{4xy-z^2=0\}=\{\det D_{(x,y)}N_f=0\}$ (in blue in
%Fig.~\ref{fig:qq1} top, left), whose other points are all sent to the circle at infinity.
The restriction of $\fN_{0,0}$ to the circle at infinity, where defined, is the identity. On each point at infinity, the eigenvalues
of $D\fN_{0,0}$ are equal to 1 in the direction tangent to infinity and to 2 in the eigendirection transversal to it, namely the points
at infinity where  $D\fN_{0,0}$ is defined are all repelling with respect to points lying outside of the circle at infinity.
Nevertheless, the iterates of the points on the half-lines $\{x=0,|y|>1\}$ and $\{|x|>1,y=0\}$ diverge while bouncing
back and forth between the two half-lines just as if $[1:0:0]$ and $[0:1:0]$ were an attracting cycle -- this is not a contradiction
because  $\fN_{0,0}$ is not defined at these points.

Rational maps of degree three can have at most 6 invariant lines. As Hubbard and Papadopol showed in~\cite{HP08},
in the complex case the Newton map of a generic polynomial map with quadratic components has the maximal number
of invariant lines, namely the lines joining all pairs of its 4 roots and it restricts to each such line to the one-dimensional
Newton method for a quadratic polynomial with those roots. 
%all (complex) lines joining two roots of a polynomial with
%quadratic components are invariant under the corresponding Newton map and its restriction to each such line equals the
%one-dimensional Newton method for a quadratic polynomial with those roots.
Correspondingly, in the real case the invariant (real) lines are all those joining pairs of distinct roots plus, in case not all
roots are real,
%we might get also ``ghost'' invariant lines which are
the intersection with the real plane of complex lines passing through pairs of distinct complex roots. For a generic such map
therefore there will be exactly 6 invariant straight lines when all roots are real and only 2 otherwise. In this last case, we call
{\em ghost line} the invariant line that does not contain any root.

The two lines invariant under $\fN_{0,0}$ are
%In case of $f$, we get exactly two invariant line:
the bisetrix of the first and third quadrant, joining the two real roots, and the ghost line $\ell=\{y+x+z=0\}$,
intersection with the real plane of the complex line passing through the two complex roots $e^{2\pi i/3}$ and $e^{4\pi i/3}$
(in brown in Fig.~\ref{fig:qqa1} top, left). $\fN_{0,0}$ restricts on $\ell$ to the Newton map of the polynomial
$p(x)=x^2+x+1$ and, correspondingly to the absence of roots of this polynomial, $J_p$ is the whole line and
$N_p$ has non-trivial repelling cycles, for instance $N_p(0)=-1$ and $N_p(-1)=0$. On the bisetrix, instead,
it coincides with the Newton map of the polynomial $q(x)=x^2-x$, whose dynamics, correspondingly to the fact that $q$ has a
maximal number of real roots, is known to be trivial: its Julia set is the single point $x=1/2$, the only point where $N'_q$
is zero, that has no counterimages and this point separates the two basins of attraction corresponding to the
two roots. In particular, $N_q$ has no non-trivial cycles.

Now, let $C=\{(x,y)|y>x^2\hbox{ or }x>y^2\}$ (white region in Fig~\ref{fig:qqa1}, top left),
%Denote by $C$ the union of the convex  of the two parabolas $y=x^2$ and $x=y^2$ (see Fig???)
denote the two disjoint connected components of $\Rt\setminus C$ by $A$ (the one contained
in the first quadrant) and $B$, by $C_0$ and $C_1$ the two disjoint connected components
of $C\setminus\fZ_{0,0}$ containing respectively the roots $(0,0)$ and $(1,1)$ and by $D\subset B$
the half-plane below the invariant line $\ell$.
%so that $\Rt\setminus C=A\sqcup B$,
%where $A$ is the connected component entirely contained in the first quadrant. 
%Note first that $N_f(\Rt)\subset\Rt$
%is the disjoint union of the two closed sets $A$ (entirely contained in the first quadrant) and $B$
%equal to the complement of the union of the interior $C$ of the two parabolas $y=x^2$ and $x=y^2$ (see Fig???).
The (formal) counterimages of a point $(x_0,y_0)$ are the four points $w_{m,n} = (x_0+(-1)^m\sqrt{x_0^2-y_0},y_0+(-1)^n\sqrt{y_0^2-x_0})$,
$m,n=0,1$, so that:
\begin{enumerate}
\item $\fN_{0,0}(\Rt)=A\sqcup B$;
\item $\fN_{0,0}(C_1)\subset A$, $\fN_{0,0}(C_0)\subset B$; 
\item three of the counterimages of each point in the interior of $A$ belong to $C$ and one to $A$;
\item the counterimages of each point in the interior of $B$ belong to $B\cup C$ and no more than
  two belong to $C$;
\item the three counterimages $w_{1,0},w_{0,1},w_{1,1}$ of each point $D$ belong to $D$ and the fourth
  to $C_0$.
%\item the counterimage of any neighborhood of $\partial C$ 
\end{enumerate}

%The other invariant line of $N_f$ is the one joining the two roots, $\ell'=\{y-x=0\}$. $N_f$ restricts on $\ell'$
The branch of $\fZ_{0,0}$ in the first quadrant is the boundary between the two basins, while the other one is tangent
at $[1:1:-2]$ to $\ell$.
%=\{y+x+z=0\}$. $N_f$ restricts on $\ell$ to the Newton map of the polynomial
Some other component of the Julia set of $\fN_{0,0}$ is plotted in Fig.~\ref{fig:qqa1} top, left.
The counterimage of $\ell$ under $\fN_{0,0}$ has three connected components, namely $\ell$ itself and
the two red curves, one inside $C$ and one inside $A$. The one in $C$ of course has no counterimage while
the one in $A$ has four of them, plotted in purple, three of which inside $A$ and one inside $C$.
Similarly, the branch of $\fZ_{0,0}$ in $C$ has no counterimage while the other one has four, plotted in green,
three of which inside $A$ and one inside $C$. Moreover, two of them (corresponding to the solutions
$w_{00}$ and $w_{11}$) are tangent to the two non-trivial counterimages of $\ell$ while the other two are
tangent to $\ell$.
%The numerical evidence is that $J_{N_f}$ is obtained by repeating recursively this procedure.
In the limit, infinitely many components of $\fJ_{0,0}$ will be tangent to $\ell$, forming a shape similar to the delta of
a river on the shore represented by the straight line.
%The counterimage of $\ell$ under $N_f$ has three connected components, namely $\ell$ itself and
%the curves denoted by $\ell_1$ and $\ell_2$ in Fig.~\ref{fig:qq1}.
The numerical evidence therefore suggests that, just as in the complex case, $\fJ_{0,0}$ is the $\alpha$-limit of 
$Z_{0,0}$. % (Thm~\ref{thm:JNp}, point~(5)).
%, the numerical results seems the suggest that $J_{N_f}$ is the closure of all counterimages
%of the set $\fh$ of the points where $DN_p$ is degenerate.

Another way of looking at $\fJ_{0,0}$ is through Theorem~\ref{thm:Bar}.
Let $U$ be a small enough neighborhood of the boundary of $C$ and set $X=\RPt\setminus U$.
%Then $N_f(X)\supset X$, since every point inside $U$ close enough to the boundary has a counterimage
Then $\fN_{0,0}(X)\supset X$, since every point inside $U$ has a counterimage
in $A\sqcup B$ (take $w_{0,0}$ for points in $U\cap A$ and $w_{1,1}$ for points in $U\cap B$),
and it is an open map on $X$ because we are away from the zeros of the radicands.
Hence $\lim_{n\to\infty}\fN_{0,0}^{-n}(X)$ defines a repellor for $\fN_{0,0}$ which is exactly the set of points that
do not leave $X$ under $\fN_{0,0}$ (see~\cite{Bar88}). In Fig.~\ref{fig:qqa1}, center, we show the set
$\fN_{0,0}^{-11}(10,-3.6)$, suggesting that the counterimage of a single point, rather than the whole $X$,
can be enough to give the whole $\fJ_{0,0}$. Numerical experiments also show that every point of
$B$ gives rise to similar counterimage sets while the corresponding sets for points in $A$ are empty --
in this case indeed the only counterimage of each point lying in $A$ goes to infinity.
Now, since the Julia set is invariant and numerical evidence suggests that the only attractors of $\fN_{0,0}$
are the roots of $f$, the only points that leave $X$ are those in the basins of attractions of the two roots
and so it must be $\fJ_{0,0}=\lim_{n\to\infty}\fN_{0,0}^{-n}(X)$. Notice that the presence of points of
indeterminacy inside $X$ is not a problem because these singularities can be eliminated
by blow-ups~\cite{HP08}.

%The counterimage of each point in the interior of $A$ and $B$ is four distinct points: in case of $A$,
%one point belongs to $A$ and the other three to $C$; in case of $B$, three of them belong to $B$ and
%one to $C$.
%Now, let $U$ be a small enough open neighborhood of the three points of indeterminacy
%of $\fN_{0,0}$ and of the boundary of $C$ and let $X=\RPt\setminus U$. Then the restriction of $\fN_{0,0}$
%to $X$ satisfies the conditions of Theorem ??? and the corresponding invariant set is the component
%of $J_{\fN_{0,0}}$ outside of $U$.

A third way is to see $\fJ_{0,0}$ as the invariant compact set of the IFS $\cI$ generated
by the restriction of the maps $\{w_{10},w_{01},w_{11}\}$ to the set $D=\{x+y\leq-1\}\subset B$.
%, where
%$$
%\{w_1,w_2,w_3\}=\{(x+\sqrt{x^2-y},y-\sqrt{y^2-x}),(x-\sqrt{x^2-y},y-\sqrt{y^2-x}),(x-\sqrt{x^2-y},y+\sqrt{y^2-x})\}.
%w_\pm(x,y) = (x+\sqrt{x^2-y},y-\sqrt{y^2-x})
%$$
%and similarly for the others.
Indeed $D$ is invariant under the action of $\cI$ and
$$
w_{10}(D)\cup w_{01}(D)\cup w_{11}(D) = \fN_{0,0}^{-1}(D)\cap D
$$
so that the Hutchinson operator $H(S)=w_{10}(S)\cup w_{01}(S)\cup w_{11}(S)$, $S\subset D$ associated to this IFS,
based on the discussion in the previous paragraph, has a fixed point given by $K=\lim_{n\to\infty} \fN_{0,0}^{-n}(D)=\fJ_{0,0}\cap D$.
Note that this point of view also suggests that the component of $\fJ_{0,0}$ inside $C$ is simply $w_{00}(K)$,
namely $\fJ_{0,0}=K\sqcup w_{00}(K)$.
\begin{figure}
  \centering
    \includegraphics[width=13.5cm]{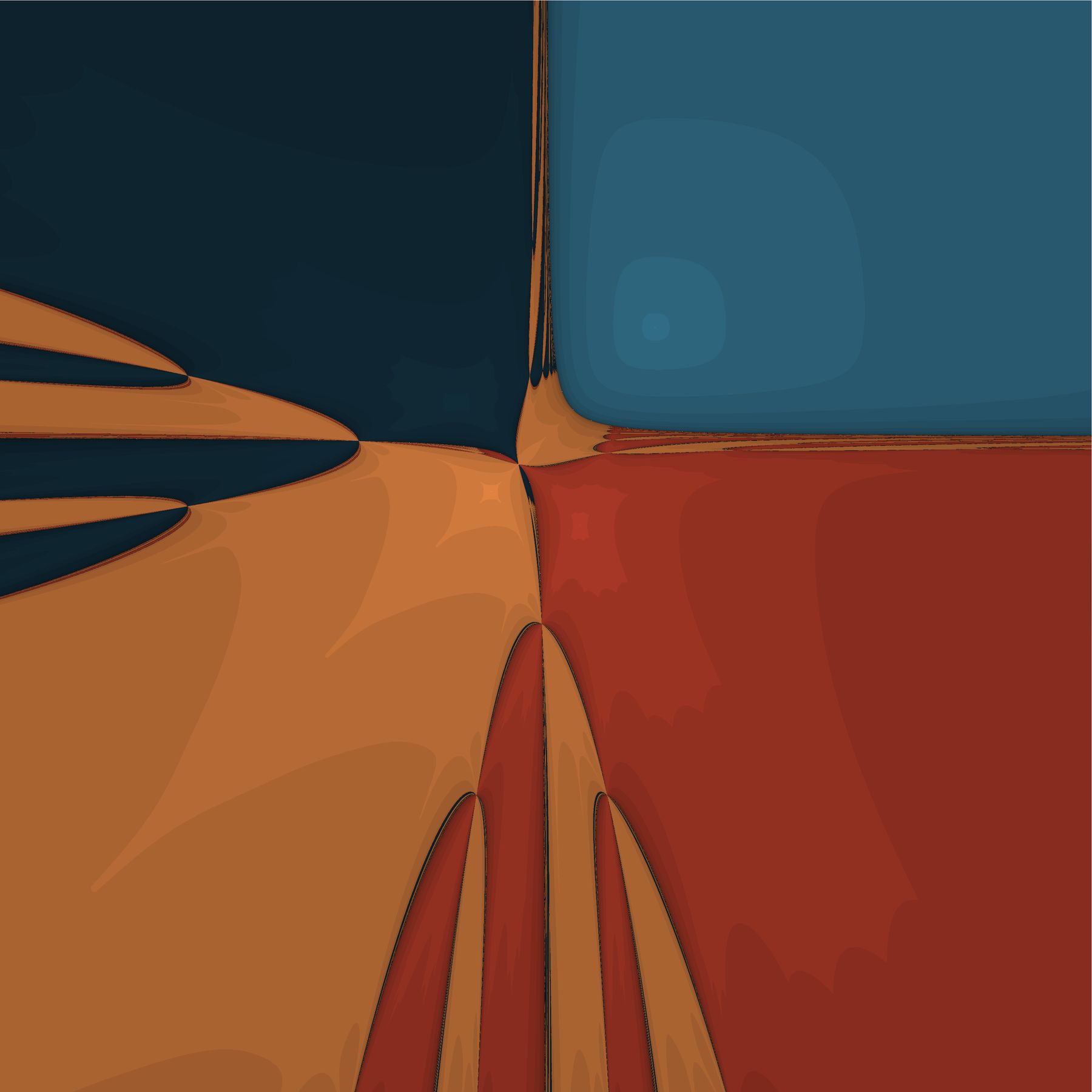}
  \caption{%                                                                                                                                                                                
    \em
    Basins of attraction of $\fN_{-2,2}$
%    $g(x,y)=(y-x^2,x+2-(y-2)^2)$
    in the square $[-10,10]^2$.
    The four real roots of $g$ are the points $(0,0)$, $(2,4)$ and, approximately, $(-1.62,2.62)$ and $(0.62,0.38)$
    and each basin corresponds to a different color. In this picture it is possible to see a qualitative difference among the
    boundary points of the basins: the ones between the orange and the light blue region close to the
    symmetry axis of the basins and the ones close to the nodal point on the same axis are regular,
    namely a small enough neighborhood of them is cut by the boundary in just two regions; all other ones
    instead seem irregular, namely each of their neighborhood intersects infinitely many basins components.
    Darker shades correspond to higher convergence time.
  }
  \label{fig:qqb}
\end{figure}
\begin{figure}
  \centering
  \begin{tabular}{cc}                                                                                                                                                                  
    \includegraphics[width=6cm]{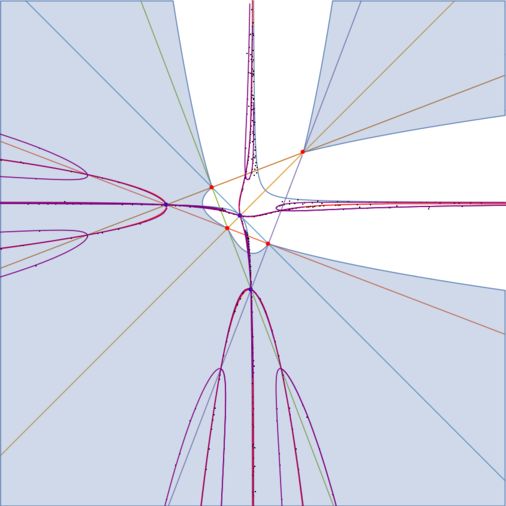}&\includegraphics[width=6cm]{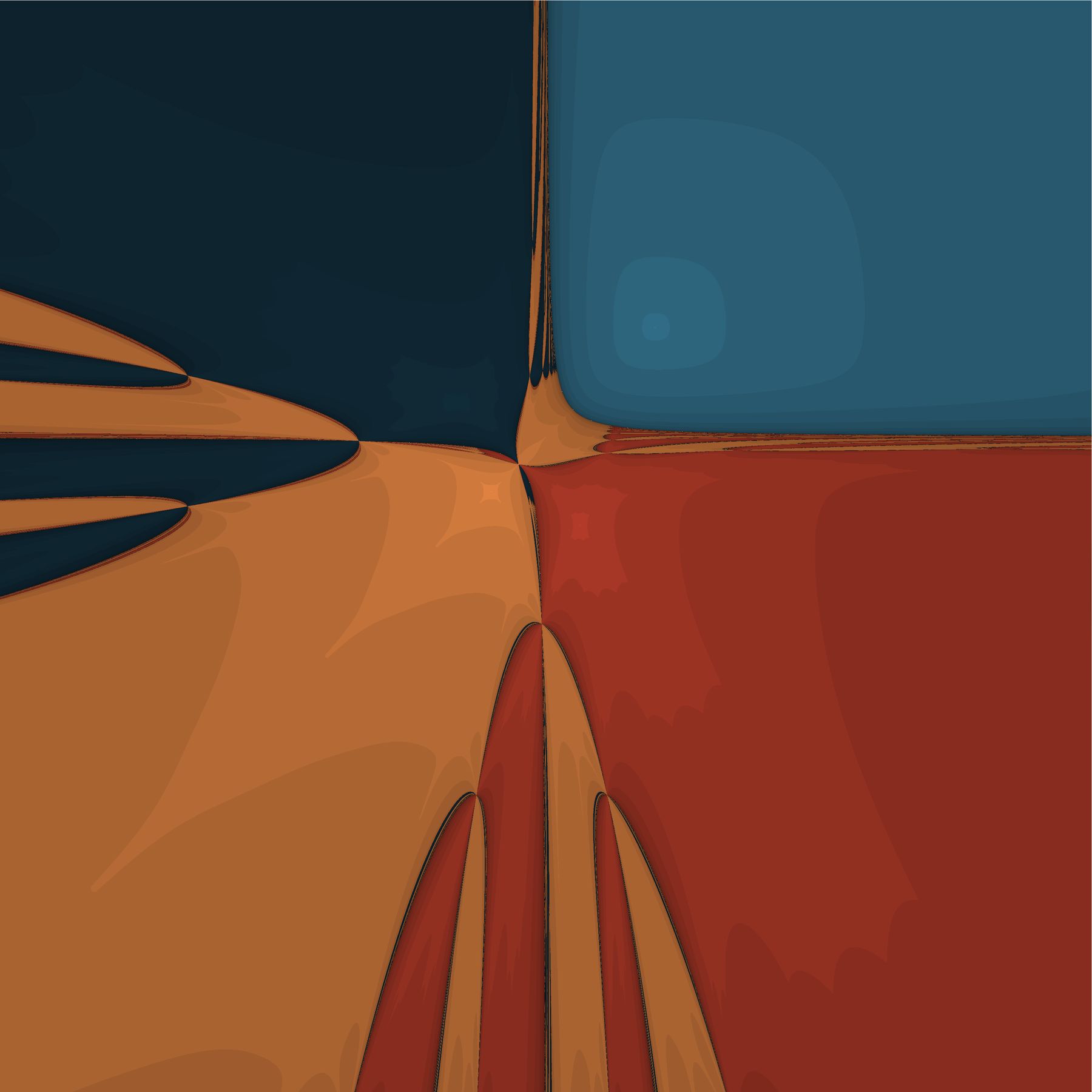}\\
    \includegraphics[width=6cm]{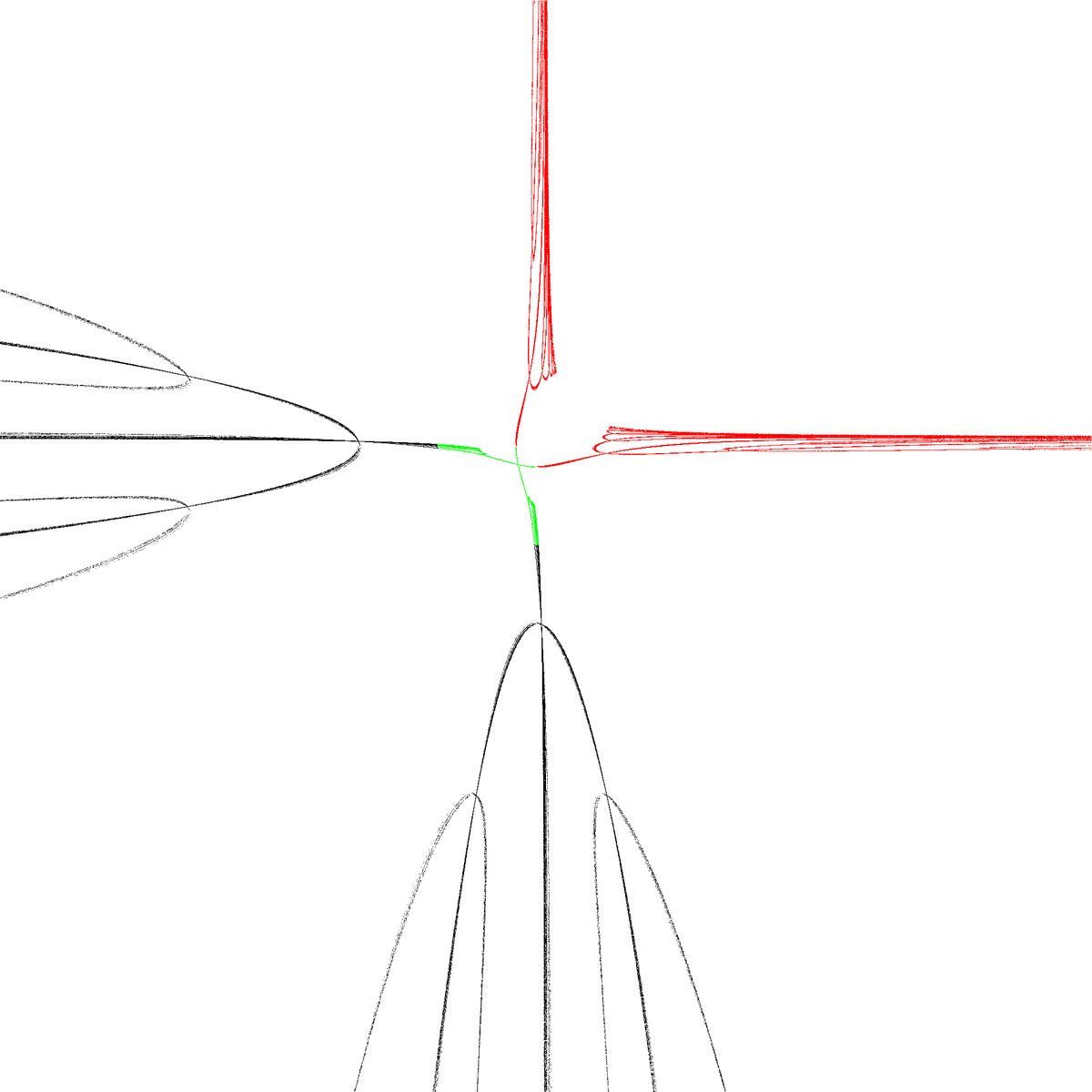}&\includegraphics[width=6cm]{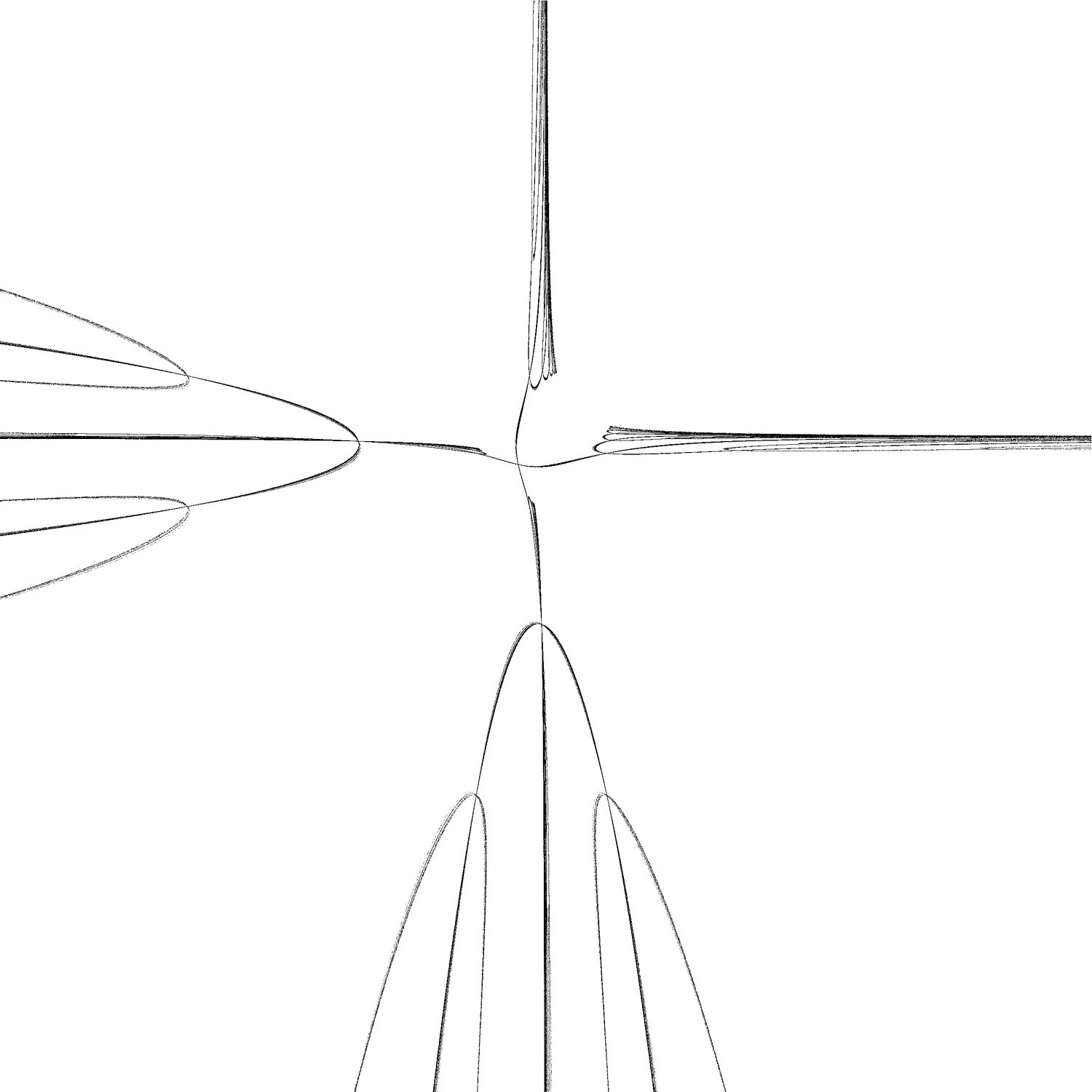}\\
  \end{tabular}
  \caption{%
    \em
    (Top, left) Some important elements of the dynamics of $\fN_{-2,2}$ in the square $[-10,10]^2$:
    fixed points (red), indeterminacy points (blue), the six invariant lines, the set $Z_{-2,2}$ (light blue)
    and its first two counterimages (resp. purple and red). The shaded region is $\fN_{-2,2}(\Rt)$ and
    the black points are the counterimages via $\fN_{-2,2}$ of a single point of the plane
    up to the 5th level of iteration.
    (Top, right) Basins of $\fN_{-2,2}$ in the square $[-10,10]^2$.
    (Bottom, left) First $3\cdot10^5$ points of a random backward orbit of a point below the light
    blue invariant line under the three branches $w_{1,0},w_{0,1},w_{1,1}$ of $\fN_{-2.2}$ (in black).
    Notice that there is a small open set below the line, when the orbits falls there we paint the point
    in green and choose a different counterimage. The red points are the image of the black points
    under $w_{0,0}$.  (Bottom, right) This picture shows the set $\fN_{-2,2}^{-13}(-20,-1.4)$. These two bottom
    pictures strongly suggest that through $\alpha$-limits we can only get the set of irregular points of the
    Julia set.
  }
  \label{fig:qqb1}
\end{figure}

Again, by Theorem~\ref{thm:Bar}, this is precisely the set of all  namely the set of points
that do not converge to $(0,0)$. The picture of the Julia set in black \& white in Fig.~\ref{fig:qqa1}(center, left)
has been obtained by plotting the first $3\cdot10^5$ points of a random backward orbit of the point $(3,-3.6)$
and the picture does not appear to depend on the initial point taken within the set $B$. On the contrary, backward
orbits starting in $A$ diverge to infinity. { We are not aware
of any systematic investigation of this fractal and, in particular, it is unknown whether its measure
is zero or not and what its Hausdorff dimension might be.}
%and the two roots of $f$ and set
%$X=\RPt\setminus U$. Then 
%Our pictures have a much higher definition of the ones available so far in literature and are obtained in two 
%pictures of the ones avaialable so far in literature and we do so in two different ways: the ones in color 

%{\bf In Fig.~\ref{fig:qq2} we consider the case of $\boldsymbol{g(x,y)=(y-x^2,x+2-(y-2)^2)}$.}
%\subsection{$\boldsymbol{g(x,y)=(y-x^2,x+2-(y-2)^2)}$}
\subsection{The map $\boldsymbol{\fN_{-2,2}}$}
The map $\ff_{-2,2}$
%$g(x,y)=(y-x^2,x+2-(y-2)^2)$
has four roots, namely the points $p_1=(-1,1)$, $p_2=(2,4)$, $p_3\simeq(-1.62,2.62)$ and
$p_4\simeq(0.62,0.38)$. Its Newton map is
$$
\fN_{-2,2} = \left(\frac{2x^2y+y^2-4x^2-1}{4xy-8x-1},x\frac{2y^2+x-4}{4xy-8x-1}\right)\,,
$$
namely, in homogeneous coordinates,  %the extension of $N_f$ to $\RPt$ is the map
$$
[x:y:z]\to[2x^2y+y^2z-4x^2z-z^3:x(2y^2+xz-4z^2):z(4xy-8xz-z^2)].
$$
In this case we have five points of indeterminacy: three bounded, namely $[-1,3,2]$, $[-7-3\sqrt{5}:1+3\sqrt{5}:4]$
and $[-7+3\sqrt{5}:1-3\sqrt{5}:4]$, and two unbounded, namely $[1:0:0]$ and $[0:1:0]$.
%, all belonging to the hyperbola $Z_{-2,2}=\{4xy-8xz-z^2=0\}$.

Just like in the previous case, the restriction of $\fN_{-2,2}$ to the circle at infinity is the identity, where it is defined,
and all these fixed points are repelling in the direction orthogonal to the circle at infinity.
%and there are only
%two invariant lines, one joining the two roots (Fig~\ref{fig:qqb1}, light blue) and the other one corresponding
%to the pair of complex conjugate solutions (orange). 
Corresponding to the 4 roots of $\ff_{-2,2}$, $\fN_{-2,2}$ has 6 invariant lines on which it restricts to the Newton map
of some quadratic polynomial in one variable (see Fig~\ref{fig:qqb1} top, left). As for $\ff_{0,0}$, we set
$\ff_{-2,2}(\Rt)=A\sqcup B$, with
$A$ the component contained in the first quadrant. We denote by $C_0$ and $C_1$ the two disjoint
connected components of $C\setminus\fZ_{-2,2}$ containing respectively the roots $(-1,1)$ and $(2,4)$
and by $H$ the half-plane below the line $x+y=1$, joining the roots $p_2$ and $p_3$.
%$C=\Rt\setminus g(\Rt)=\{y>x^2\hbox{ or } x+2>(y-2)^2 \}$.
The four (formal) preimages of a point $(x_0,y_0)$ are given by
$w_{m,n} = (x_0+(-1)^m\sqrt{x_0^2-y_0},y_0+(-1)^n\sqrt{(y_0-2)^2-x_0+2})$,
$m,n=0,1$, and satisfy properties similar to those of $f$, namely:
\begin{enumerate}
\item $\fN_{-2,2}(\Rt)=A\sqcup B$;
\item $\fN_{-2,2}(C_1)\subset A$, $\fN_{-2,2}(C_0)\subset B$; 
\item three of the counterimages of each point in the interior of $A$ belong to $C$ and one to $A$;
\item three of the counterimages of each point of the interior of $B$ belong to $B\cup H$ and one to $C$.
%\item the counterimage of any neighborhood of $\partial C$ 
\end{enumerate}
In Fig.~\ref{fig:qqb1}, top left, we show the four roots (in red) together with the lines joining them, the three bounded
indeterminacy points (in blue), the region $C$ (in white), the hyperbola $\fZ_{-2,2}$ (in blue) and its first two
counterimages (respectively in red and purple) under $N_f$. The numerical evidence suggests
that the $\alpha$-limit of the set of points where $D\fN_{-2,2}$ is degenerate coincides with the set
of irregular points of $\fJ_{-2,2}$.

Now let $U$ be a small enough neighborhood of $C$ and set $X=\RPt\setminus U$. Then,
similarly to what happens for $\fN_{0,0}$, $\fN_{-2,2}(X)\supset X$ and the restriction of $\fN_{-2,2}$ to $X$ is an open map.
Hence $\lim_{n\to\infty}\fN_{-2,2}^{-n}(X)$ defines a repellor for $\fN_{-2,2}$ which is exactly the set of points that
do not leave $X$ under $\fN_{-2,2}$. In Fig.~\ref{fig:qqb1}, bottom right, we show the set $\fN_{-2,2}^{-13}(-20,-1.4)$,
suggesting that this repellor coincides with the set of irregular points of $\fJ_{-2,2}$. As in the previous case,
The counterimages of any point in $B$ give similar results while the counterimages of points in $A$
go to infinity.

Similarly to the previous case, we can get the set of irregular points of  $\fJ_{-2,2}$ as the invariant compact
set of the IFS $\cI$ generated by the restriction of the maps $\{w_{10},w_{01},w_{11}\}$ to the set
$D=\{x+y\leq1\}\subset B$. In this case though $\cI$ is not, strictly speaking, a well-defined IFS
on $D$ since there is a bounded open subset $N\subset D$ where the maps are not defined. Nevertheless it
is reasonable to believe, supported by numerical experiments, that the concept of IFS can be extended
to this type of maps and that a unique compact subset still exist. In Fig.~\ref{fig:qqb1}, bottom left,
we show in black the first $3\cdot10^5$ points of a random backward orbit of $(3,-3.6)$ under $\cI$;
the green points are points that fell , during the backward iteration, in $N$, in which case we select
a different random counterimage in order to keep going backwards; the red points are the image of the black ones
under the map $w_{00}$. Numerical experiments suggest that the limit of such backward orbits do not
depend on the initial point inside $D\setminus N$.
\begin{figure}
  \centering
    \includegraphics[width=13.5cm]{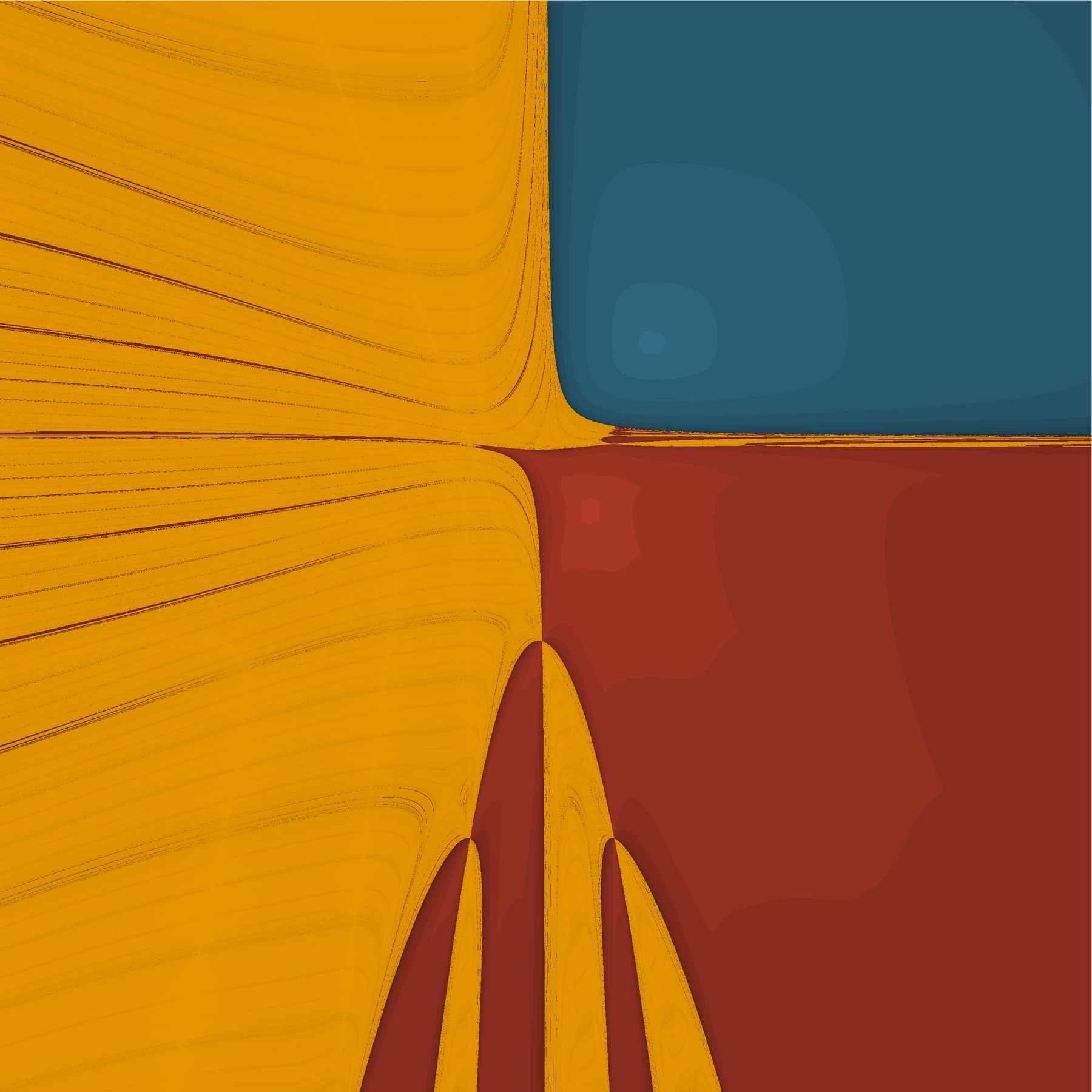}
  \caption{%                                                                                                                                                                                
    \em                                                                                                                                                                            
    Basins of attraction of $\fN_{-1,2}$
%    the Newton map of $h(x,y)=(y-x^2,x+1-(y-2)^2)$
    in the square $[-10,10]^2$.
    The two real roots of $\ff_{-1,2}$ are the points $p_1\simeq (1.9,3.7)$ and $p_2\simeq(0.81,0.65)$
    and the corresponding basins have been colored in cyan and red respectively.
    The basin in gold correspond to a chaotic attractor on the ghost line of $\fN_{-1,2}$.
    Darker shades correspond to higher convergence time.
  }
  \label{fig:qq2b}
\end{figure}
\begin{figure}
  \centering
  \begin{tabular}{cc}                                                                                                                                                                  
    \includegraphics[width=6cm]{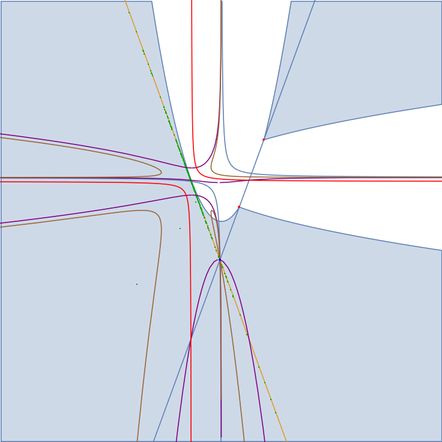}&\includegraphics[width=6cm]{qq2b}\\
    \includegraphics[width=6cm]{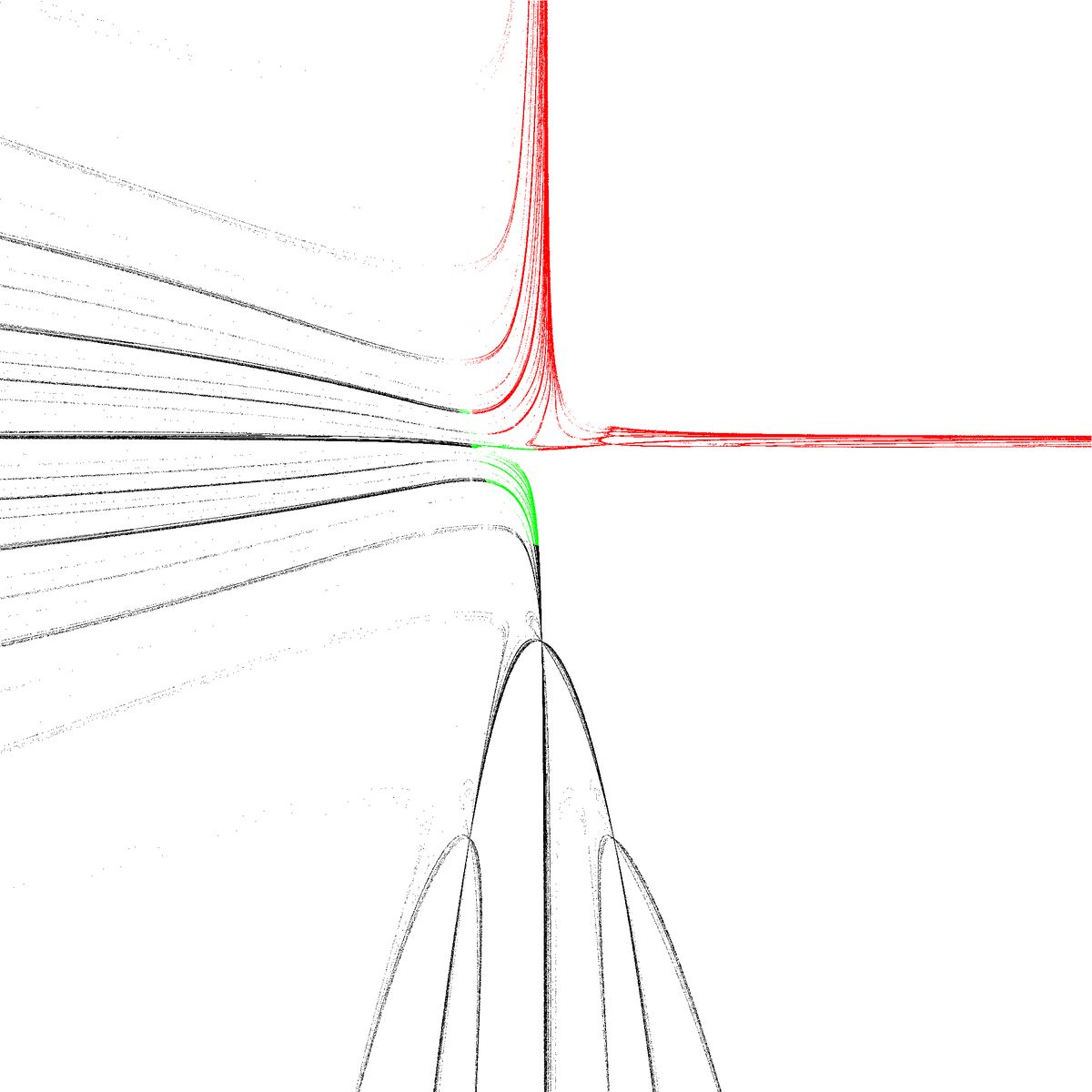}&\includegraphics[width=6cm]{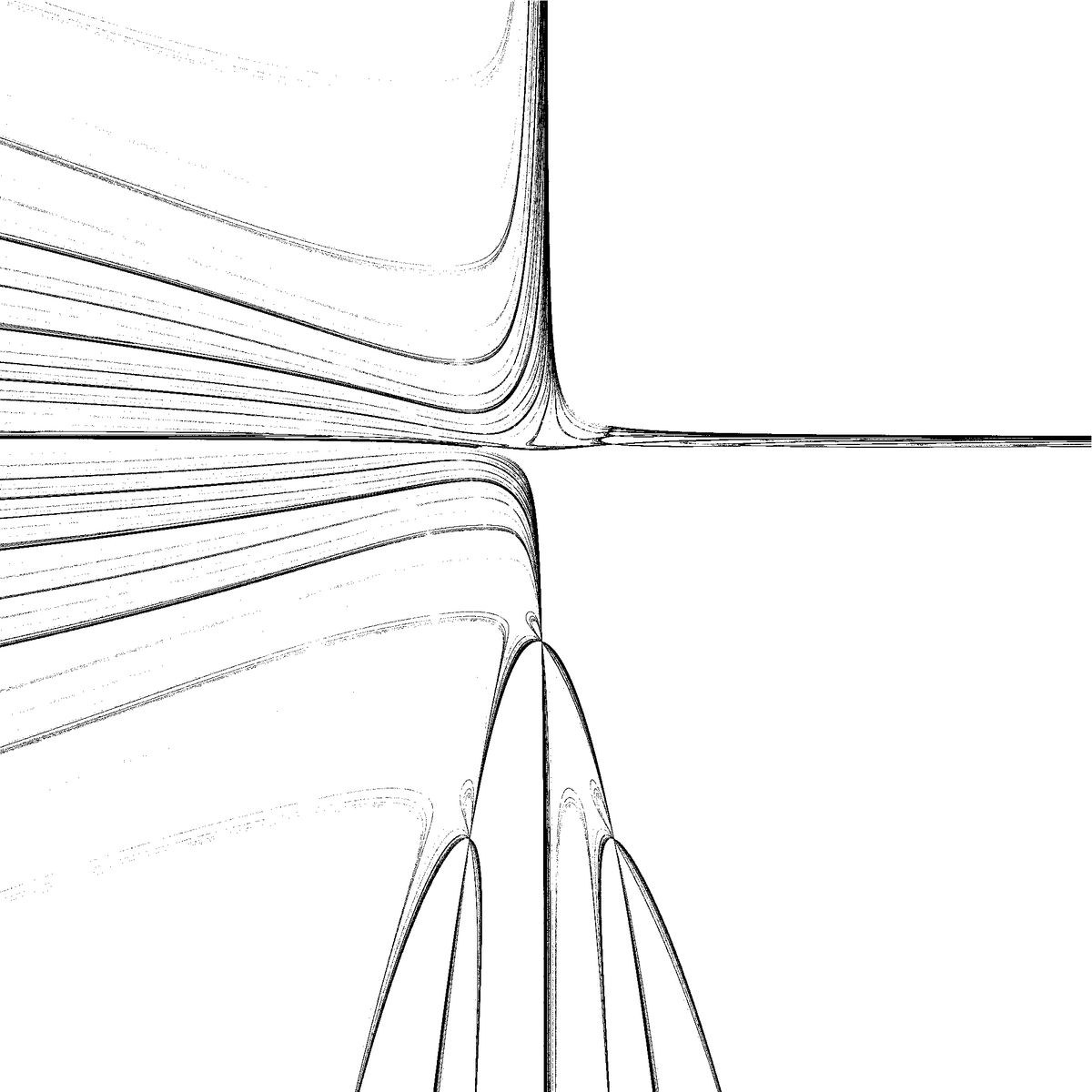}\\
  \end{tabular}
  \caption{%
    \em
    (Top, left) Main elements in the dynamics of $\fN_{-1,2}$: fixed points (red), point of indeterminacy (blue),
    invariant line joining the two roots and ghost line with its first and second counterimages (purple and gray),
    the set $\fZ_{-1,2}$ (light blue) with its first and second counterimages (red and brown) and first 500 points
    of the orbit of a random point in the black region (green). The green points suggest the existence of an
    attracting Cantor set in the ghost line.
    (Top, right) Basins of attraction of $\fN_{-1,2}$ in the square $[-10,10]^2$.
    (Bottom, left) First $3\cdot10^5$ points of a random backward orbit of a point below the light
    blue invariant line under the three branches $w_{1,0},w_{0,1},w_{1,1}$ of $N_h$ (in black).
    Notice that there is a small open set below the line, when the orbits falls there we paint the point
    in green and choose a different counterimage. The red points are the image of the black points
    under $w_{0,0}$.  (Bottom, right) This picture shows the set $\fN_{-1,2}^{-15}(-10,-0.23)$. These two bottom
    pictures suggest that through $\alpha$-limits we can only get the set of irregular points of the
    Julia set.
 %   Basins of attraction and Julia set of $f(x,y)=(y-x^2,x+2-(y-2)^2)$.
%    The only real roots of $f$ are the points $(0,0)$ and $(1,1)$.
%    In the upper left corner are shown the basins of attraction of the two roots (in shades of green for $(1,1)$, red for $(0,0)$) and
%    the Julia set (in yellow) in the square $[-3,3]^2$. The yellow points are actually points that take a long time to converge and
%    will disappear numerically if we allow a high number of iterations. Below this picture we show a numerical approximation,
%    at the 15th recursion level, of the $\alpha$-limit of the point
%    $(10,-3.6)$, confirming that the yellow set above is not just a numerical accident. In the second column we show corresponging
%    similar pictures in the square $[-1.05,-0.95]\times[-0.05,0.05]$. In the third colummn we show the basins and the Julia set
%    in the squares of side respectively $10^{-3}$ (above) and $10^{-4}$ (below) centered at $(-1,0)$.
  }
  \label{fig:qq2b1}
\end{figure}
%
%\subsection{$\boldsymbol{h(x,y)=(y-x^2,x+1-(y-2)^2)}$}
\subsection{The map $\boldsymbol{\fN_{-1,2}}$}
\label{ss:h}
The Newton map
$$
%\fN_{-1,2}([x:y:z]) = [2x^2y+y^2z-4x^2z-3z^3:x(2y^2+xz-6z^2):z(4xy-8xz-z^2)]
\fN_{-1,2}(x,y) = \left(\frac{2 x^2 y + y^2 - 4 x^2 -3 }{4xy-8x-1},x\frac{2 y^2+x-6}{4xy-8x-1}\right)
$$
shares many characteristics with $\fN_{0,0}$. 
%Like $\ff_{0,0}$, $\ff_{-1,2}$ has only two roots,
%namely the points $p_1\simeq(-1.9,3.7)$ and $p_2\simeq(0.81,0.65)$,
%and its corresponding Newton map is
%$$
%N_h = \left(\frac{2 x^2 y + y^2 - 4 x^2 -3 }{4xy-8x-1},x\frac{2 y^2+x-6}{4xy-8x-1}\right)
%$$
%which, in homogeneous coordinates, writes as %the extension of $N_f$ to $\RPt$ is the map
%$$
%[x:y:z]\to[2x^2y+y^2z-4x^2z-3z^3:x(2y^2+xz-6z^2):z(4xy-8xz-z^2)].
%$$
Like $\fN_{0,0}$, $\fN_{-1,2}$ has two attracting fixed points $p_1\simeq(-1.9,3.7)$ and
$p_2\simeq(0.81,0.65)$, corresponding to the two roots of $\ff_{-1,2}$ and three points of indeterminacy:
one bounded, approximately $[-0.067:-1.7:1]$, and two unbounded, namely $[1:0:0]$ and $[0:1:0]$.
%, all belonging to the hyperbola $\fZ_{-2,2}=\{4xy-8xz-z^2=0\}$;
The restriction of $\fN_{-1,2}$ to the circle at infinity is the identity, where defined, and all these fixed points are
repelling in a direction transversal to the circle at infinity. Even in this case, $\fN_{-1,2}(\Rt)=A\sqcup B$
and $Z_{-1,2}$ divides $C=\Rt\setminus\fN_{-1,2}(\Rt)$ into two open sets $C_0$ and $C_1$ so that
$\fN_{-1,2}(C_1)\subset A$, $\fN_{-1,2}(C_0)\subset B$
and the half-space $D$ below the ghost line is invariant under the first three of the inverse branches
$$
w_{m,n} = (x_0+(-1)^m\sqrt{x_0^2-y_0},y_0+(-1)^n\sqrt{(y_0-2)^2-x_0-1})\,,
$$
with $m,n=0,1$.

Unlike $\fN_{0,0}$, though, $\fN_{-1,2}$ shows evident signs of the presence of a third attractor $\cC$,
whose basin is shown in gold color in Fig.~\ref{fig:qq2b}, lying on its ghost line. Numerics suggest
%[add picture!!!]
that the dynamics on $\cF(\cC)$ is chaotic (namely highly dependent on the initial point).
%and that $\cC$ has a Cantor set structure.
In particular, this means that $\cF(\cC)\subset \fJ_{-1,2}$, suggesting
furthermore that $\fJ_{-1,2}$ has non-empty interior and so a non-zero measure. Consistently with Conjecture~\ref{conj:alpha}, 
%The Julia set, though, in this case appears larger and more complicated. The basins of attraction show in Fig~\ref{fig:qq2b}
%clearly suggests that either there are non-trivial attracting cycles, like in case of the cubic map in Sec.~\ref{ss:ac}, or the measure
%of $\fJ_{-1,2}$ is larger than 0. The same impression is suggested by
the numerical evaluation of the $\alpha$-limits of points in $B$ (Fig~\ref{fig:qq2b1} bottom, right) and of the invariant set of the
IFS generated by the maps $w_{1,0},w_{0,1},w_{1,1}$ (bottom, left) suggests that they are both equal to the set of non-regular points
of the boundary of $\fJ_{-1,2}$.
%{\bf The numerical evaluation of $\omega$-limits for points belonging to the black regions suggests the presence
%  of an attracting Cantor set $K$ on the ghost line and that the dynamics of $N_h$ on $K$ is chaotic} (i.e. highly depending
%  on the intial conditions. In turn, this would imply that $\fJ_{-1,2}$ has a non-empty interior.

\bigskip
\begin{figure}
  \centering
  \begin{tabular}{cccc}                                                                                                                                                                  
    \includegraphics[width=3cm]{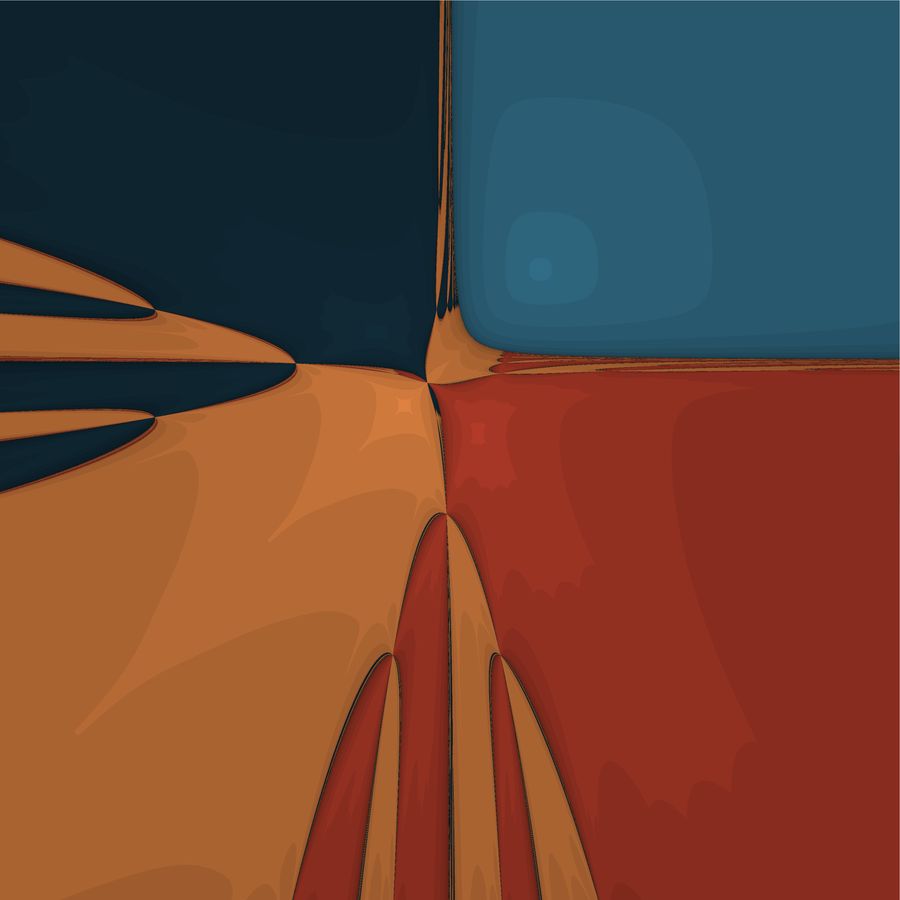}&\includegraphics[width=3cm]{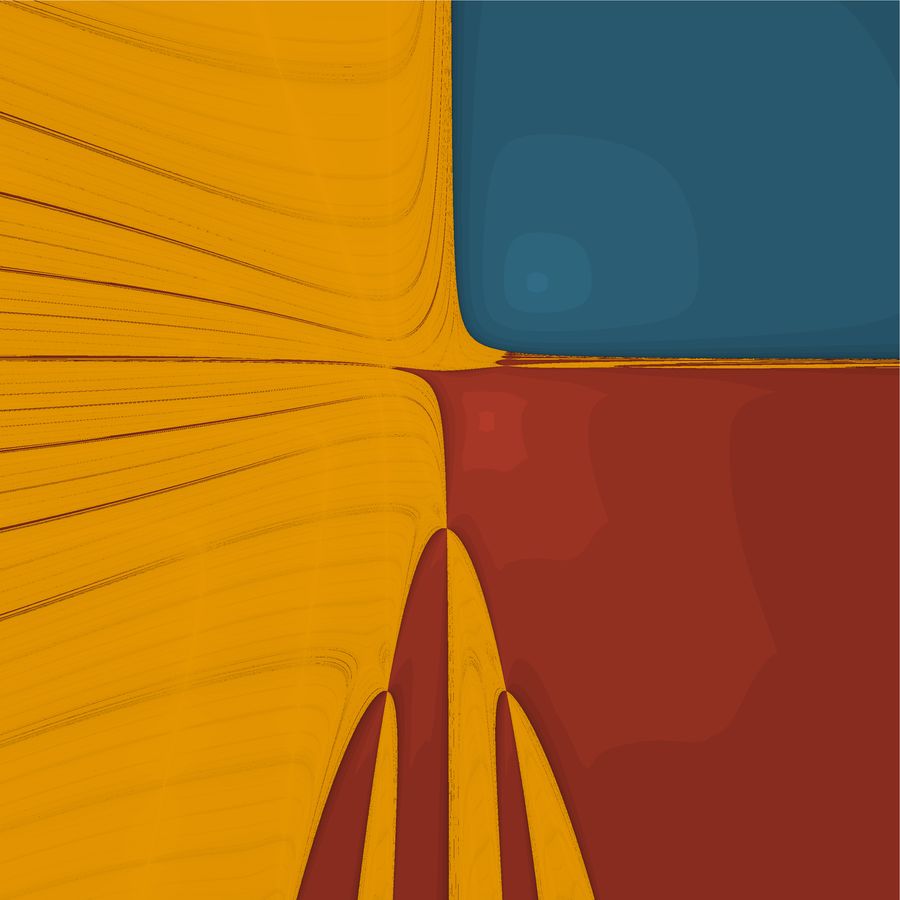}&\includegraphics[width=3cm]{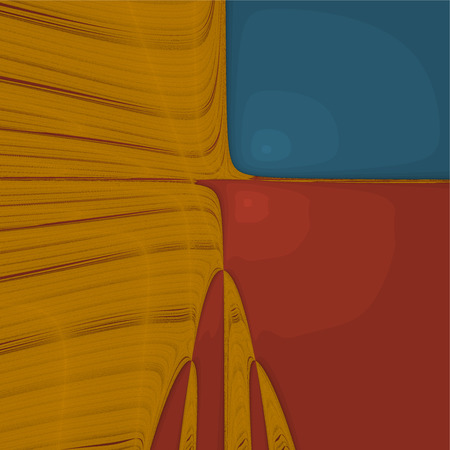}&\includegraphics[width=3cm]{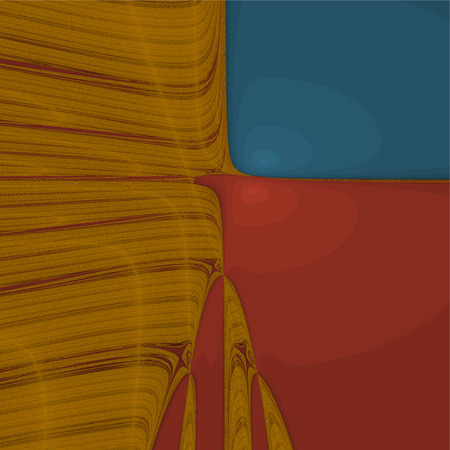}\\
    \includegraphics[width=3cm]{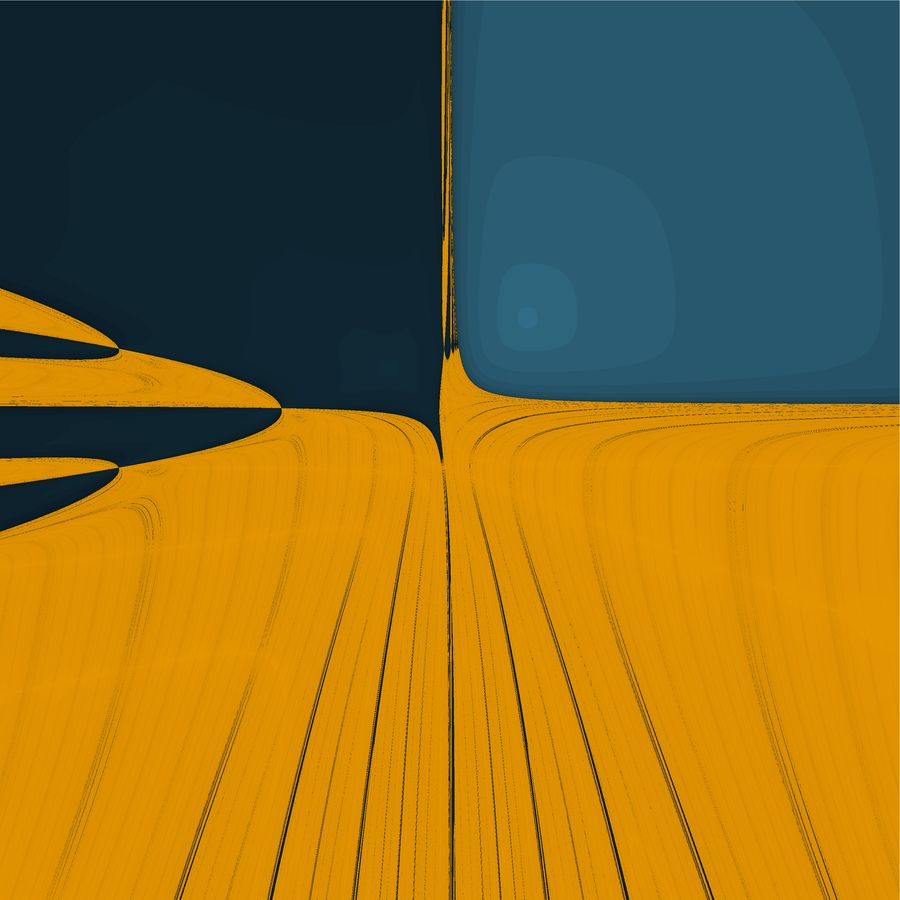}&\includegraphics[width=3cm]{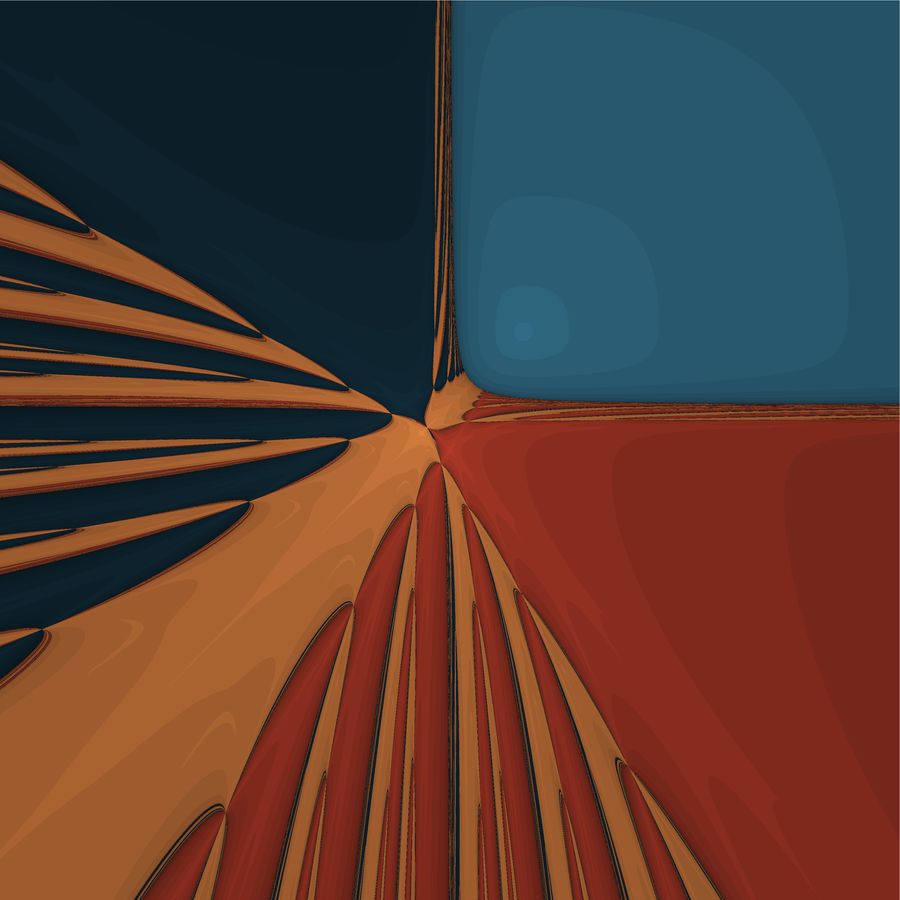}&\includegraphics[width=3cm]{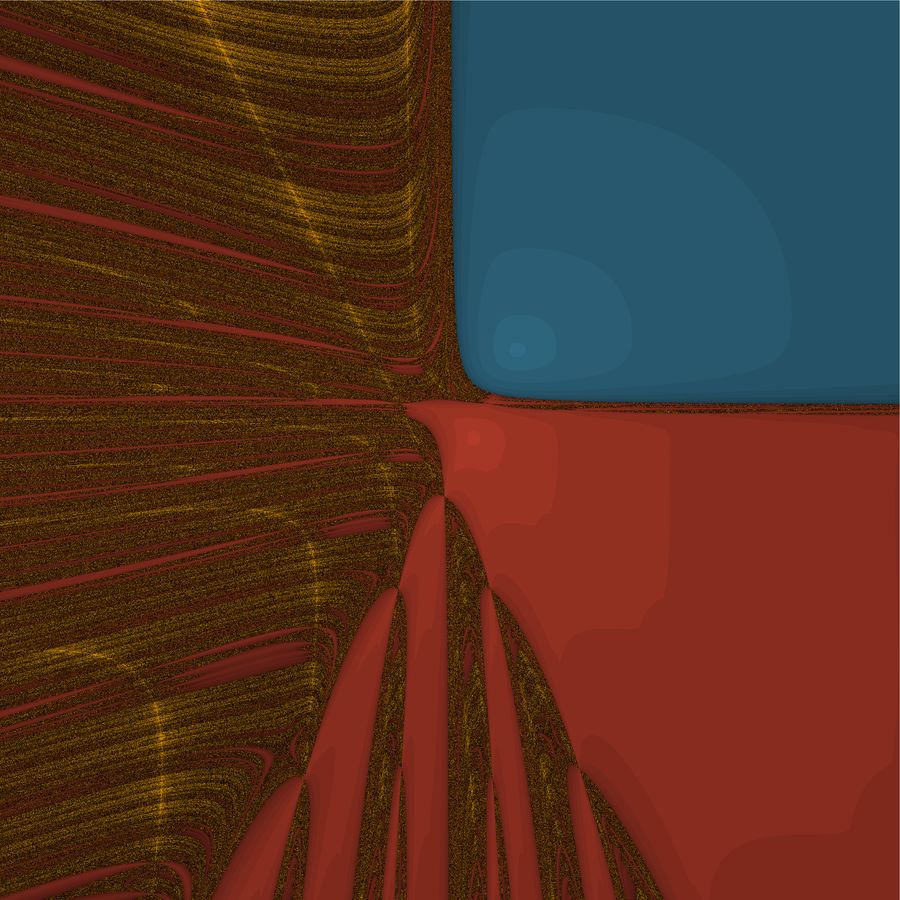}&\includegraphics[width=3cm]{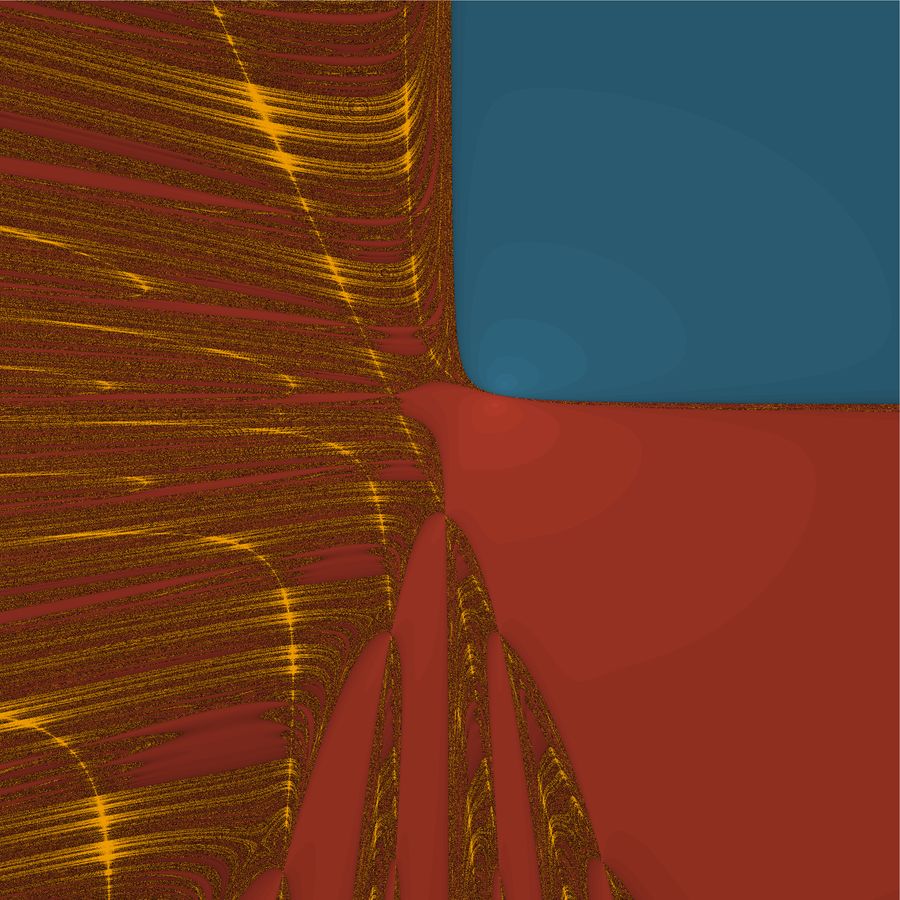}\\
    \includegraphics[width=3cm]{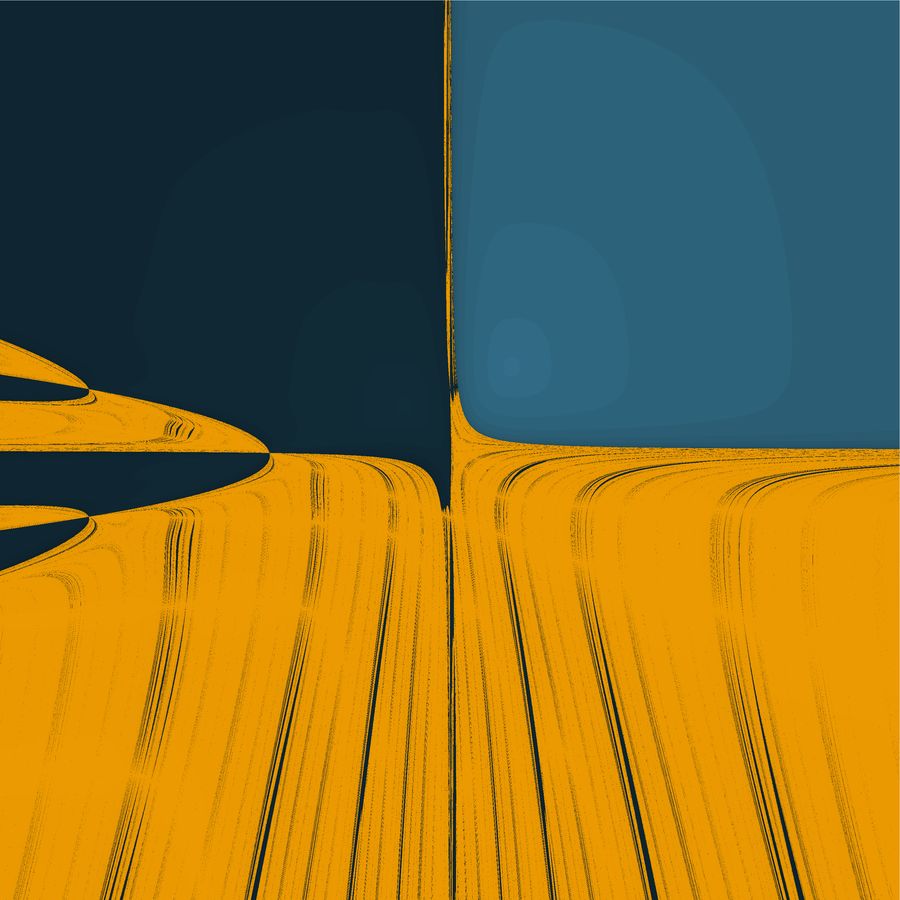}&\includegraphics[width=3cm]{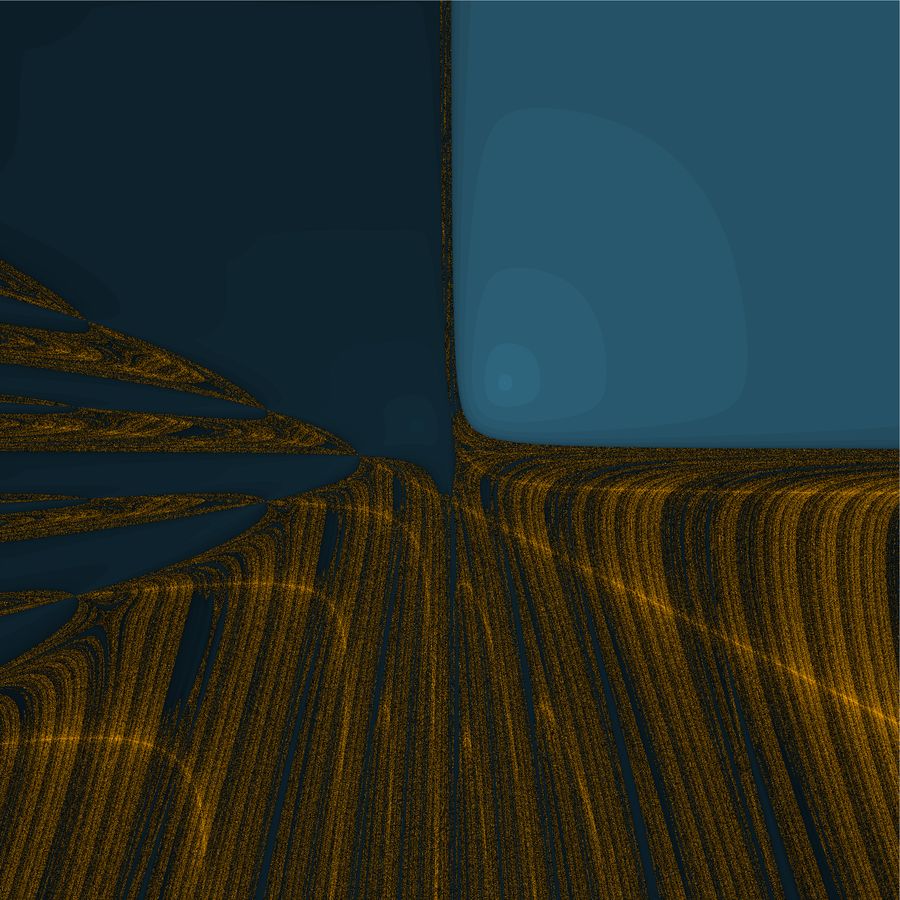}&\includegraphics[width=3cm]{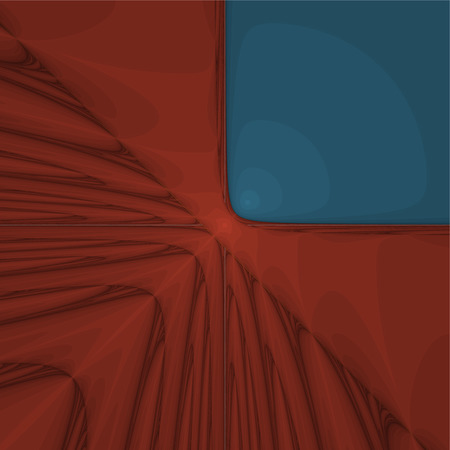}&\includegraphics[width=3cm]{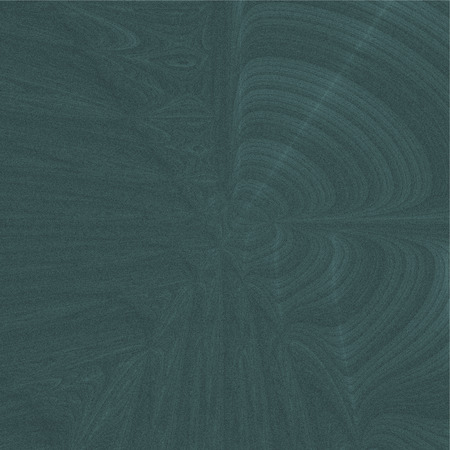}\\
    \includegraphics[width=3cm]{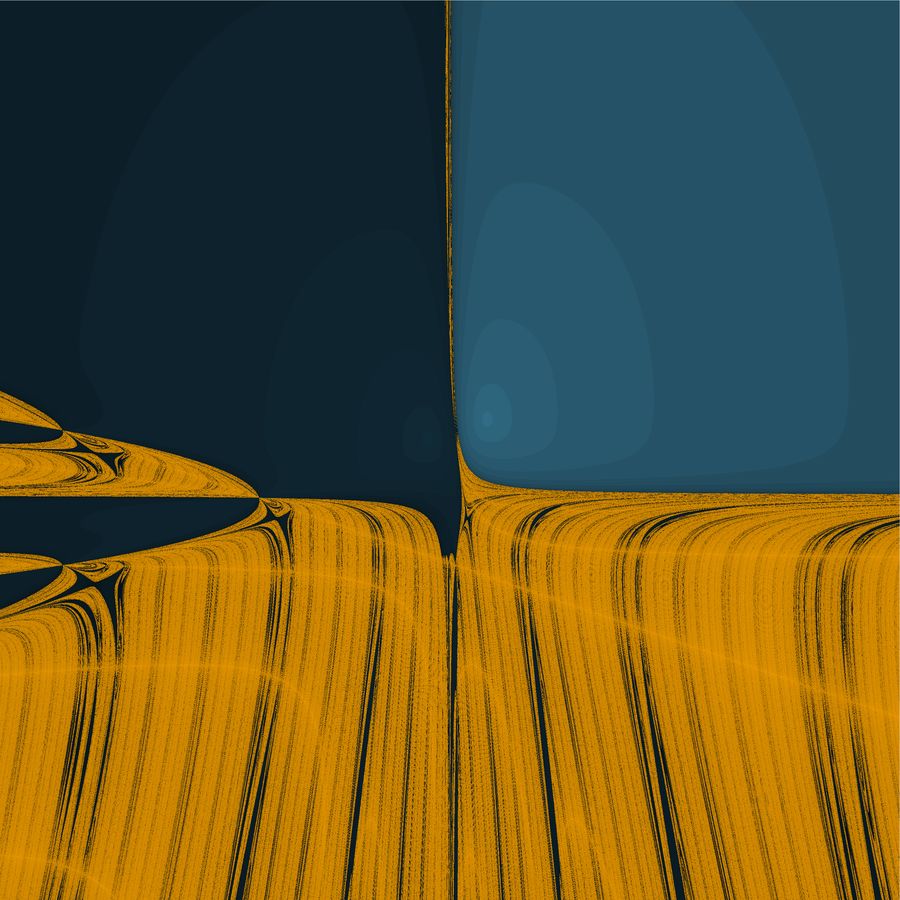}&\includegraphics[width=3cm]{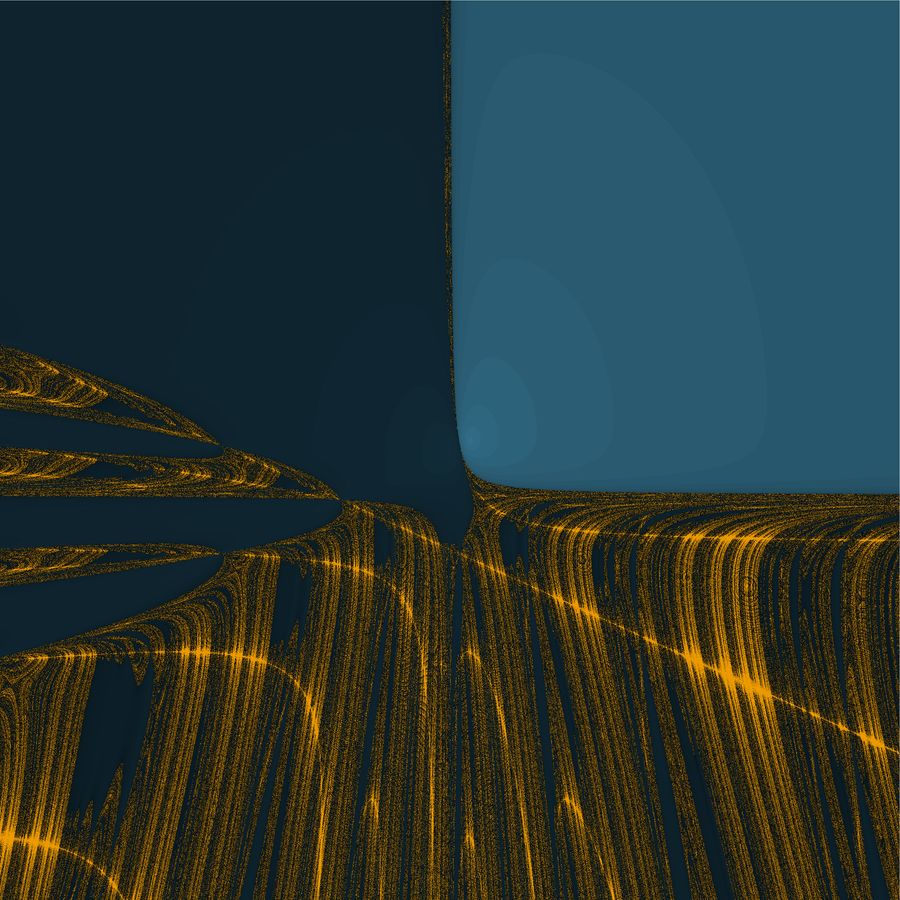}&\includegraphics[width=3cm]{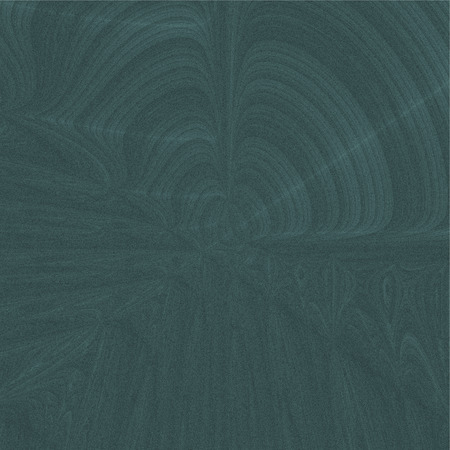}&\includegraphics[width=3cm]{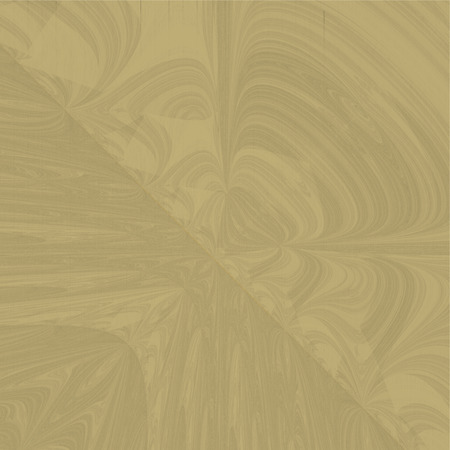}\\
    \includegraphics[width=3cm]{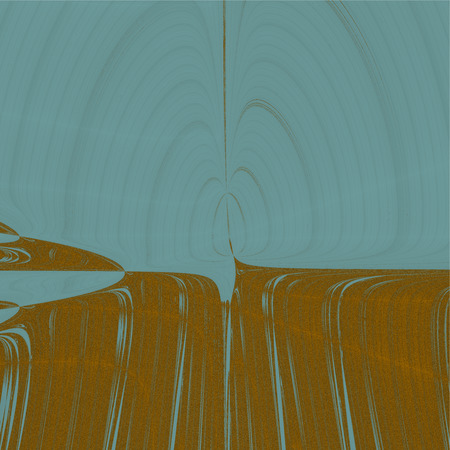}&\includegraphics[width=3cm]{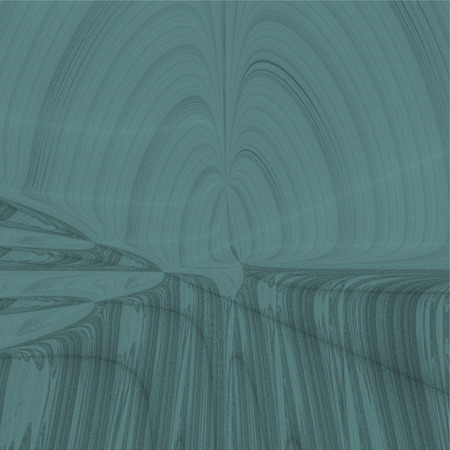}&\includegraphics[width=3cm]{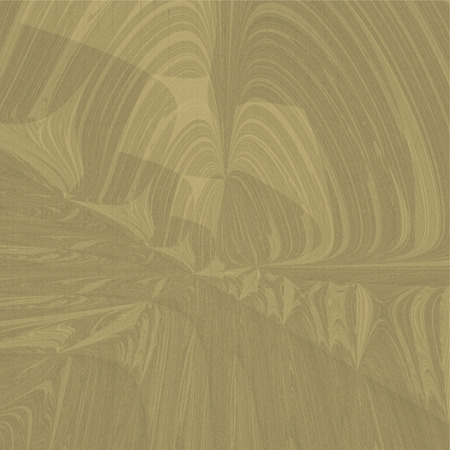}&\includegraphics[width=3cm]{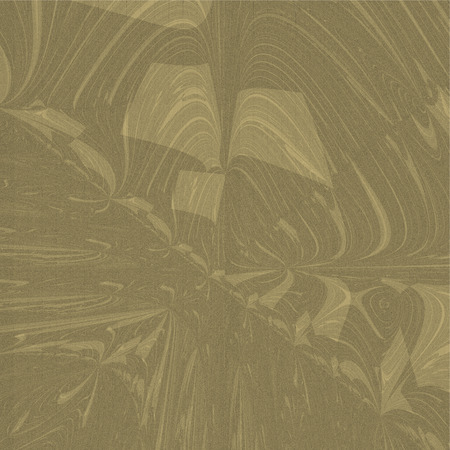}\\
  \end{tabular}
  \caption{%
    \em
    Basins of attraction for $\fN_{x_0,y_0}$ corresponding to the values $x_0=-2,-1,0,1$ and $y_0=-2,-1,0,1,2$.
    The numerical results strongly suggest that the union of the basins of attraction of the fixed points
    corresponding to the roots of $\ff_{x_0,y_0}$ has full Lebesgue measure when the map has four roots
    and that chaotic attractors, some of which lying on the map's ghost lines, arise when the number of roots
    of $\ff_{x_0,y_0}$ is not maximal. Darker shades correspond to higher convergence time.
  }
  \label{fig:qq2}
\end{figure}
\begin{figure}
  \centering
  \begin{tabular}{cccc}                                                                                                                                                                  
    \includegraphics[width=3cm]{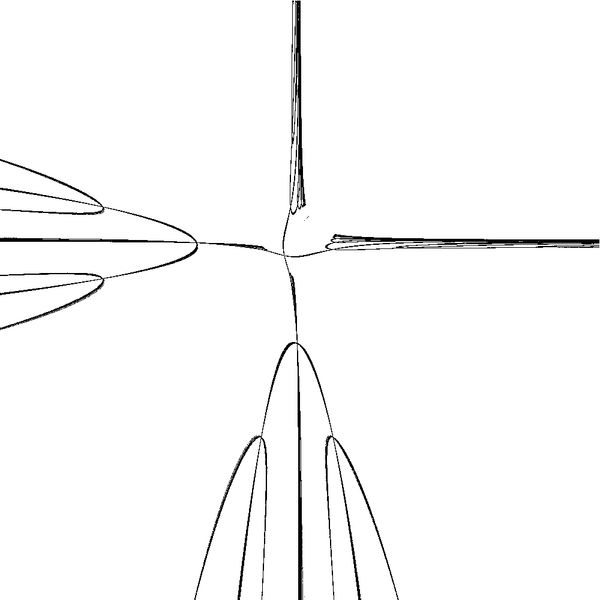}&\includegraphics[width=3cm]{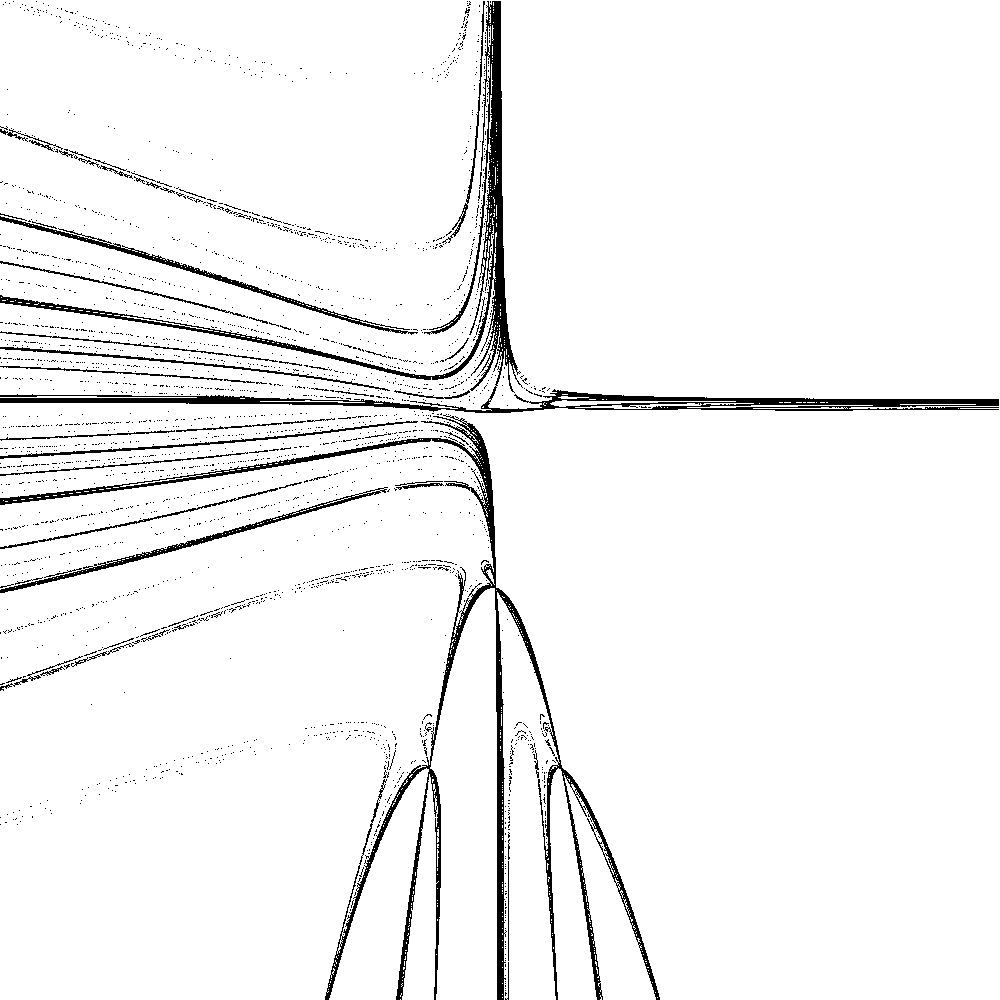}&\includegraphics[width=3cm]{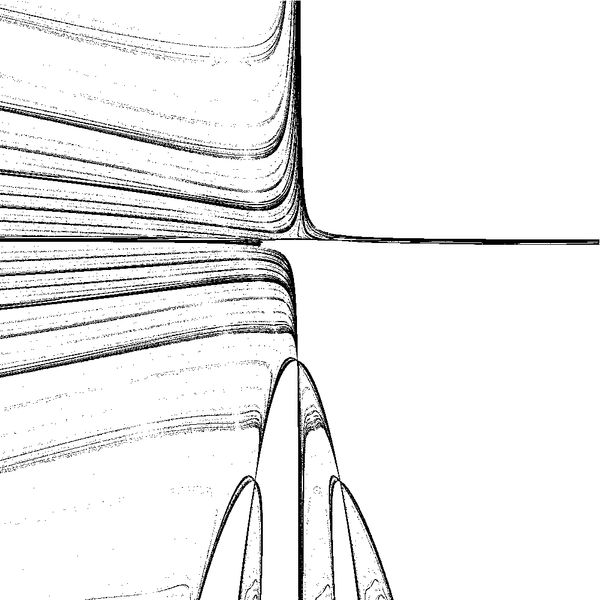}&\includegraphics[width=3cm]{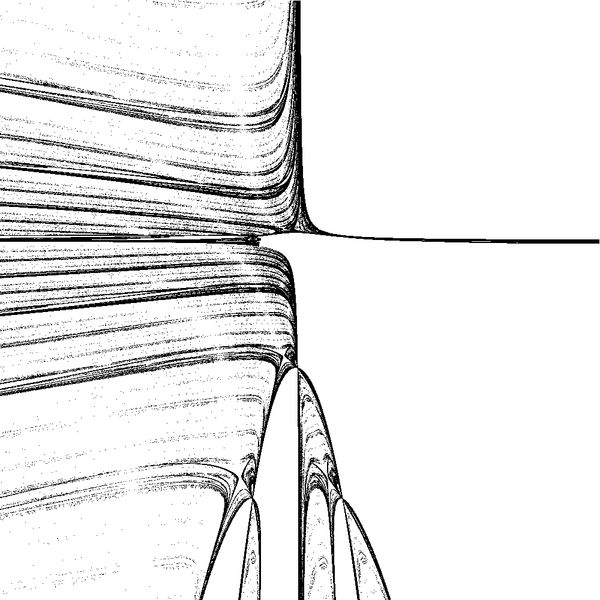}\\
    \includegraphics[width=3cm]{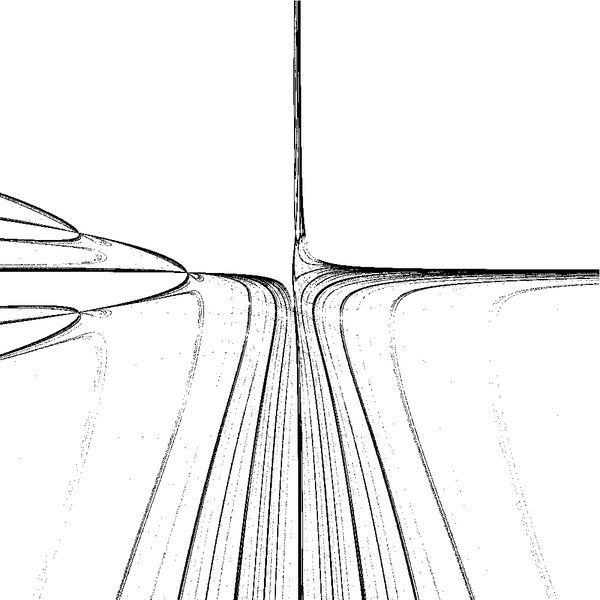}&\includegraphics[width=3cm]{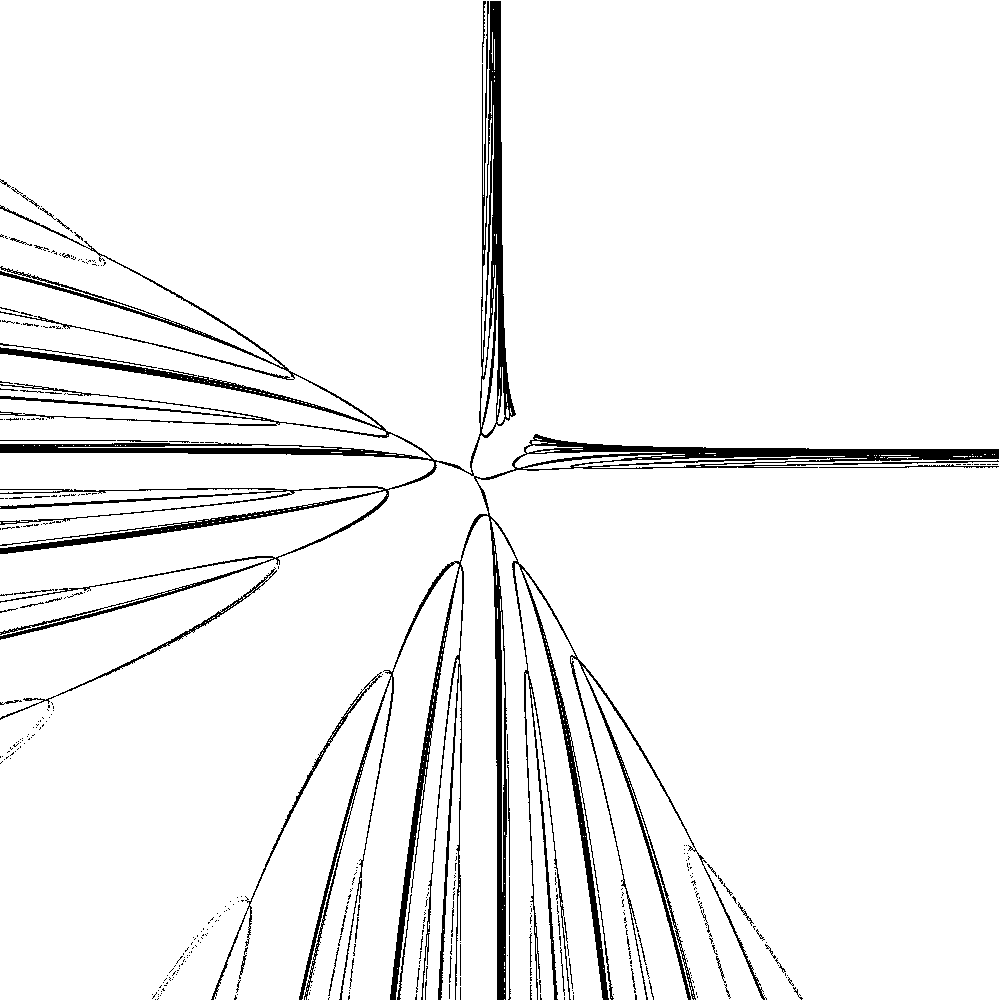}&\includegraphics[width=3cm]{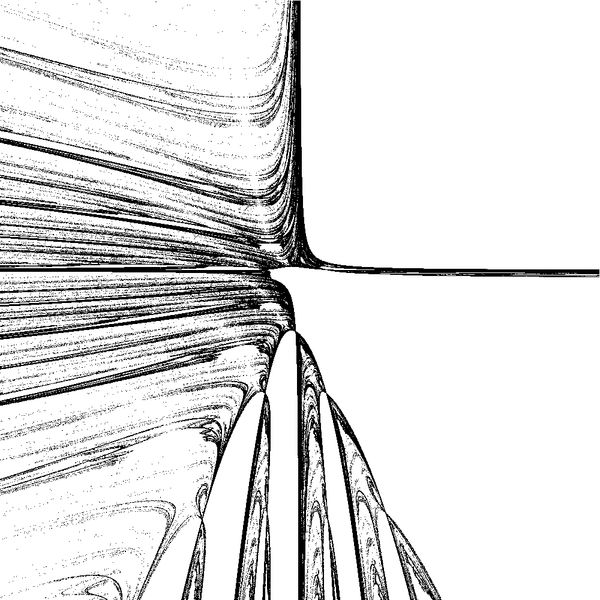}&\includegraphics[width=3cm]{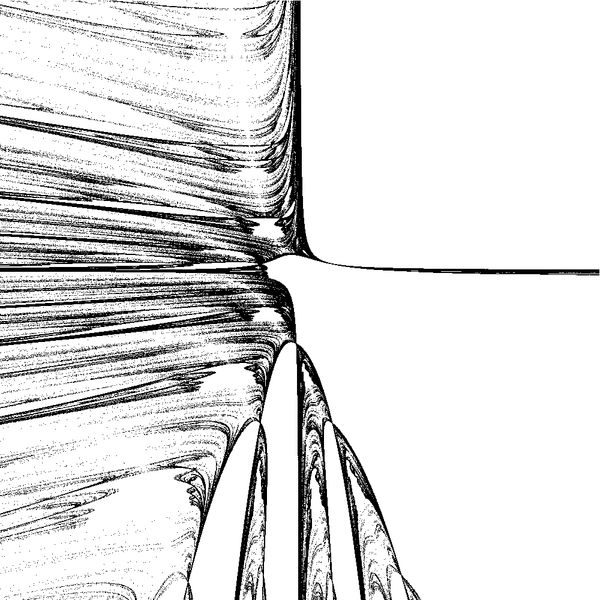}\\
    \includegraphics[width=3cm]{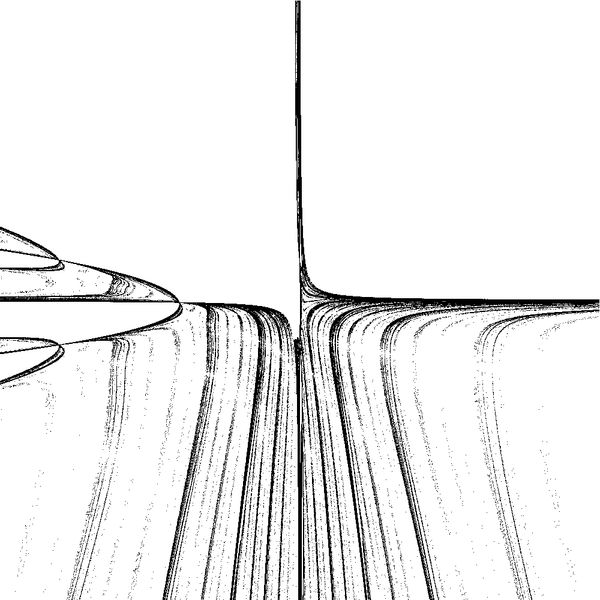}&\includegraphics[width=3cm]{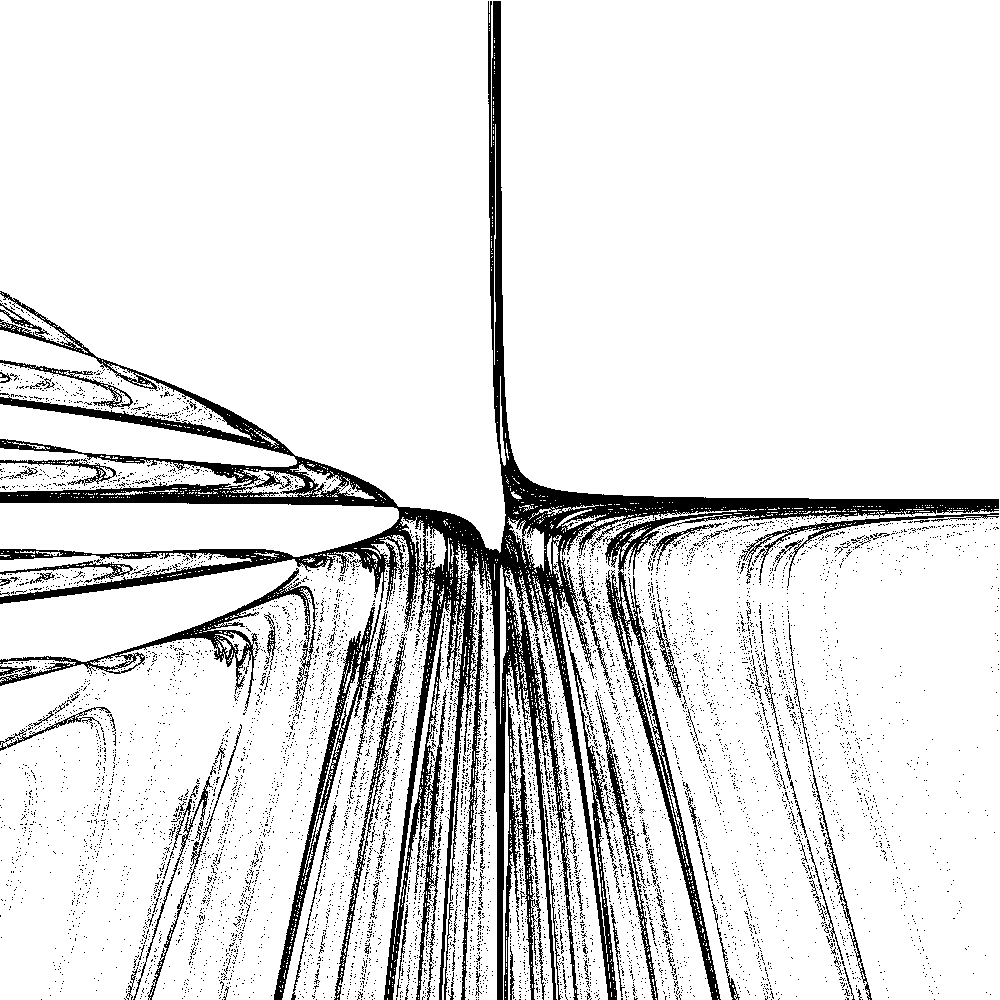}&\includegraphics[width=3cm]{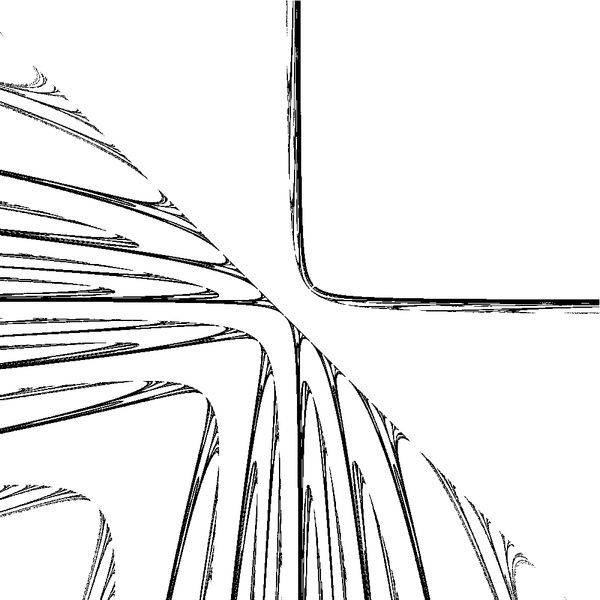}&\includegraphics[width=3cm]{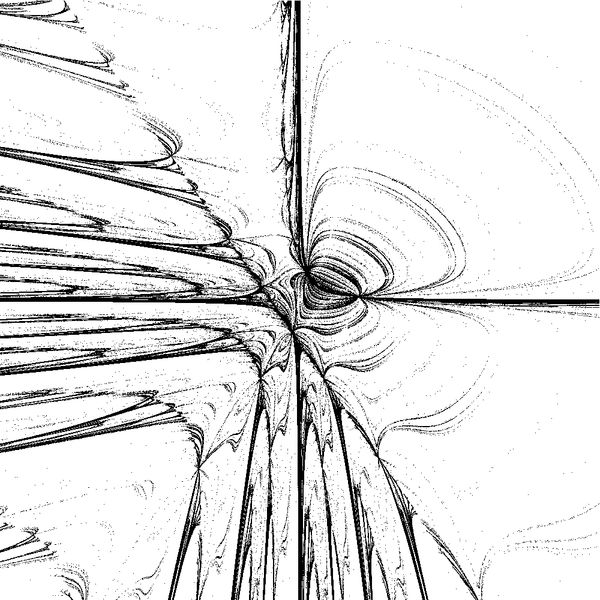}\\
    \includegraphics[width=3cm]{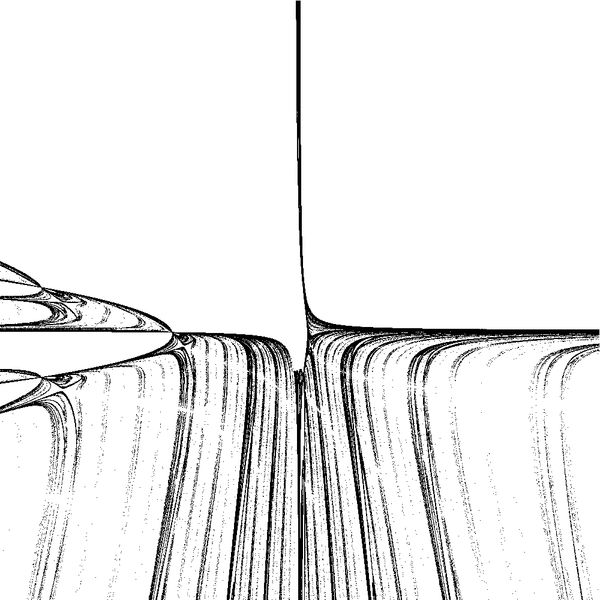}&\includegraphics[width=3cm]{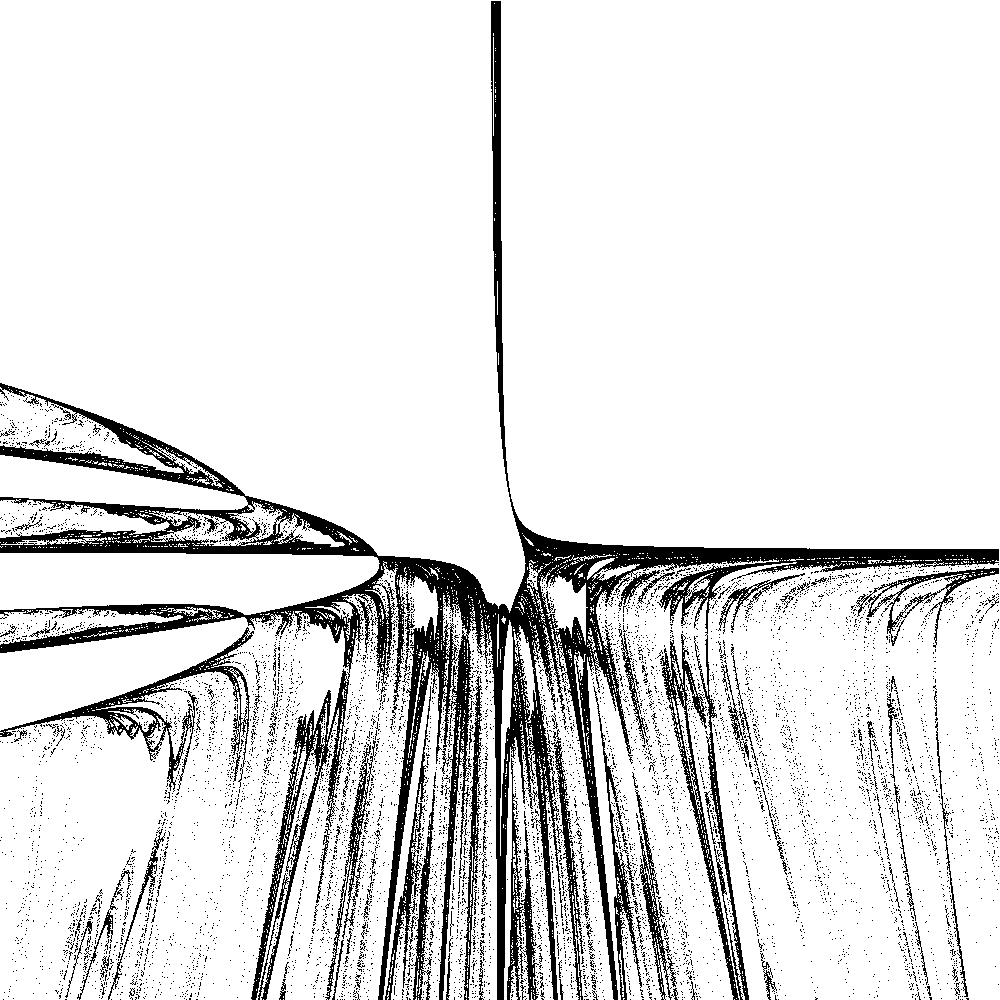}&\includegraphics[width=3cm]{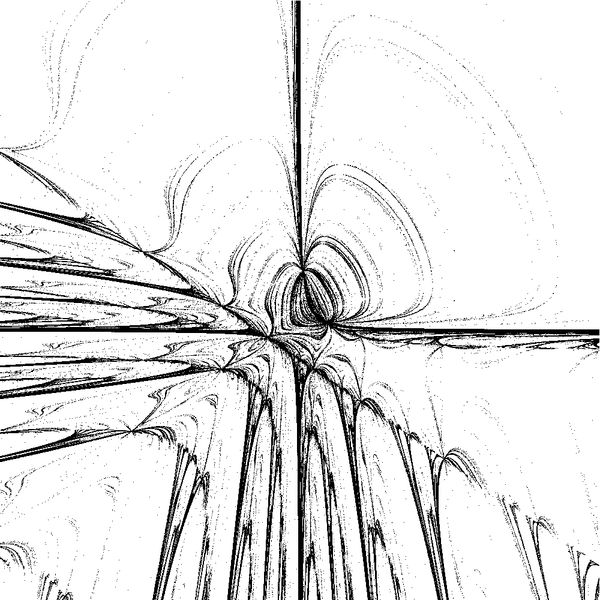}&\includegraphics[width=3cm]{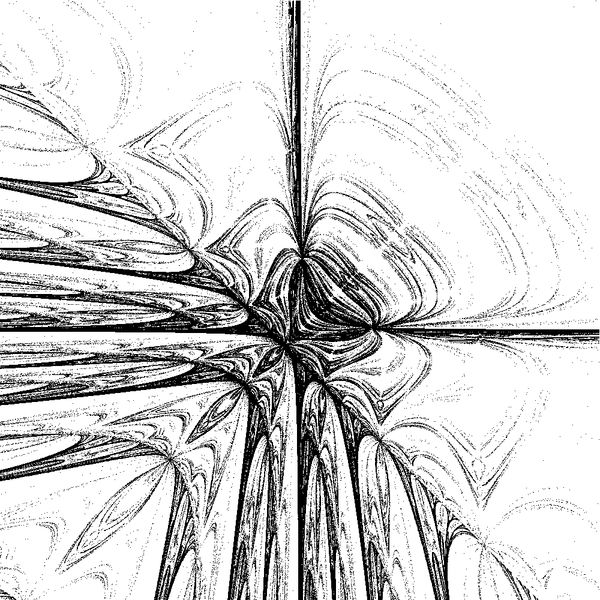}\\
    \includegraphics[width=3cm]{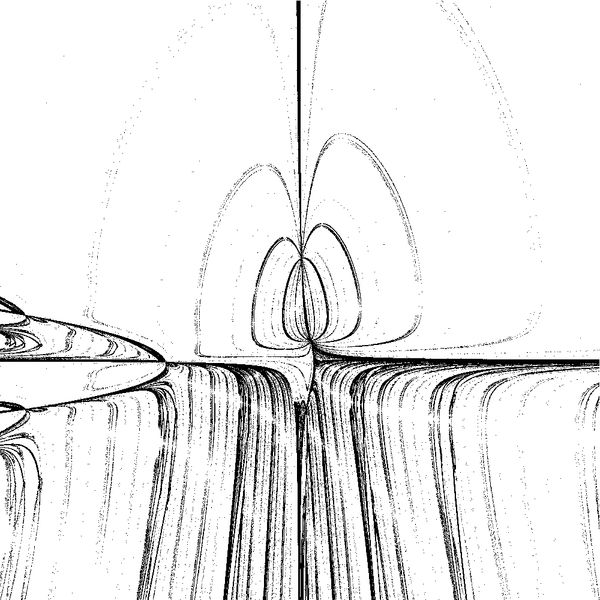}&\includegraphics[width=3cm]{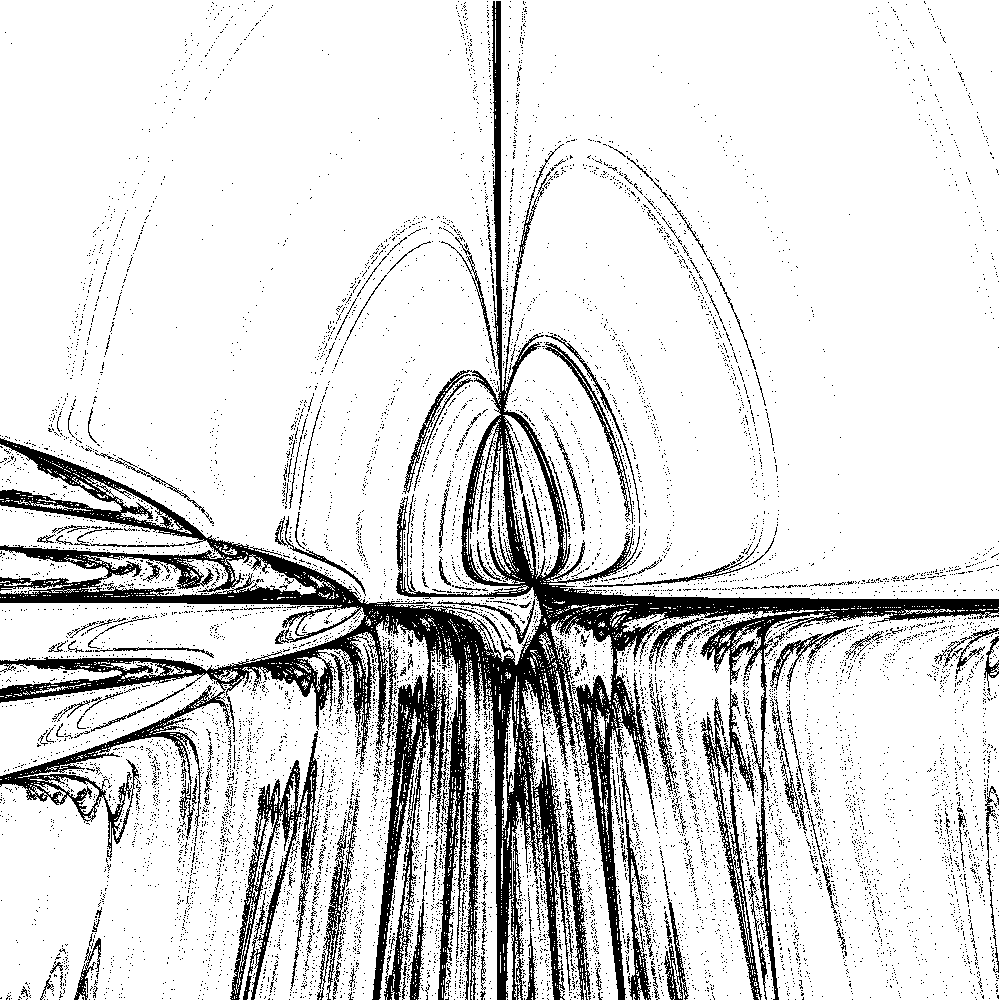}&\includegraphics[width=3cm]{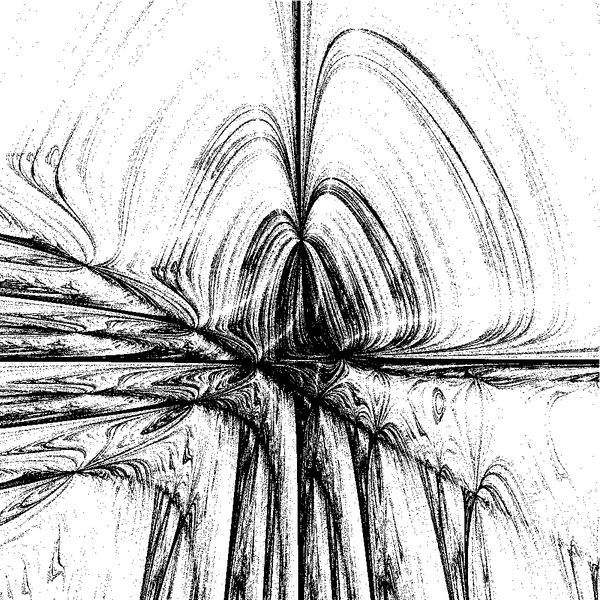}&\includegraphics[width=3cm]{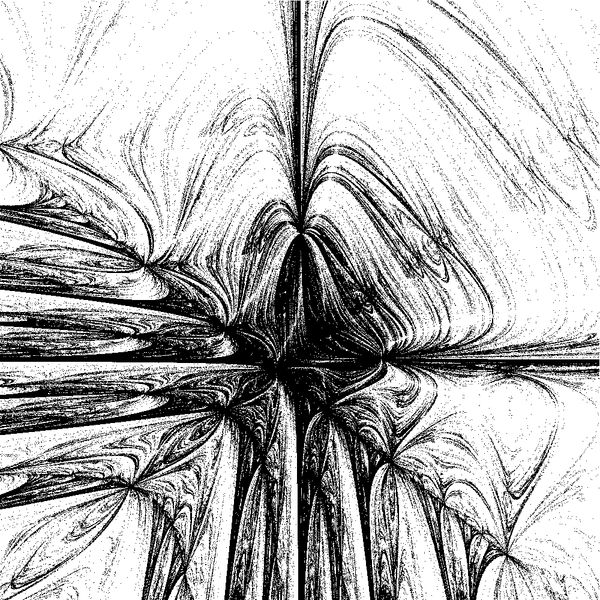}\\
  \end{tabular}
  \caption{%
    \em
    $\alpha$-limit of a generic point under  the parabolic Newton maps $\fN_{x_0,y_0}$ corresponding to
    the values $x_0=-2,-1,0,1$ and $y_0=-2,-1,0,1,2$. All these sets are also visible in the previous figure
    as the non-regular points of the basins boundaries. 
  }
  \label{fig:qq2ii}
\end{figure}
\subsection{A panorama of maps $\fN_{x_0,y_0}$}
In Fig.~\ref{fig:qq2} and~\ref{fig:qq2ii} we show, respectively, the basins of attraction and a $\alpha$-limit
of the Newton maps of twenty parabolic maps $\ff_{x_0,y_0}$ with $x_0=-2,-2,0,1$ and $y_0=-2,-1,0,1,2$
to support our conjectures~\ref{conj:alpha} and~\ref{conj:J} under several different situations.

The maps $\ff_{-2,2}$ and $\ff_{-1,1}$ have both four roots and the unions of the four relative basins of attraction
of the corresponding Newton maps appear to be full-measure. Similarly, although it has only two roots, it happens
for the map $\ff_{0,0}$ already discussed at length in Section~\ref{sec:f00}. Their $\alpha$-limits shown
in Fig.~\ref{fig:qq2ii} suggest that regular points of the Julia set are not reached and that these limit sets
coincide with the closure of the set of counterimages, under their Newton maps, of the set where these
Newton maps have a degenerate Jacobian. Note, though, that each point of the Julia set of the first two
Newton maps is a boundary point between basins of attraction corresponding to different roots while, for
$x=y=0$, almost all points of the Julia set have a neighborhood containing the basin of a single root.

For all other maps in Fig.~\ref{fig:qq2} with $y\geq-1$ and $x\leq y$, $\ff_{x,y}$ has only two real roots but
three attractors arise: two of them are the basins of attraction of the two real roots while the third one is some
subset on the corresponding ghost line. Numerics suggest that the dynamics on the basin of this third attractor
is chaotic, namely it is a subset of the Julia set rather than of the Fatou set like the other two basins. In particular,
this means that all corresponding Julia sets have non-empty interior and non-zero measure. The corresponding
$\alpha$-limits shown in Fig.~\ref{fig:qq2ii} suggest that the interior points of Julia sets cannot be reached.

The remaining maps $\ff_{x,y}$ in Fig.~\ref{fig:qq2} have no real roots and therefore the corresponding Newton
maps have two ghost lines. For $x=y=-2$, on each of these two lines lie an attractor, corresponding to the two basins
that are visible in the corresponding picture. Note that, in this case, the Julia set is the whole $\RPt$ since
on both basins the dynamics appears to be chaotic. Nevertheless, the corresponding  $\alpha$-limit seems to
be the set of boundary points between the two basins. In case of the maps corresponding to pairs $(x,y)$ with
$y=x-1$, only one attractor is visible, the one corresponding to the ghost line associated to the first pair of complex
solutions that disappears when $y$ increases. A trace of the $\alpha$-limit shown in Fig.~\ref{fig:qq2ii} can be
seen also as the darker area, namely the points that converge more slowly, in the picture showing the $\omega$-limits,
suggesting again that this $\alpha$-limit is the closure of the set of the counterimages of the set of degeneracy
of the Newton maps Jacobian. Finally, the remaining pictures show, again, a single basin of attraction. This time,
though, the attractor is not a subset of either one of the ghost lines -- the orbits of generic points seem rather to fill
up some two-dimensional region. Just like in the previous case, a trace of the $\alpha$-limit shown in Fig.~\ref{fig:qq2ii}
can be seen as the darker area in the basin of attraction. 

%that the union of the basins of attraction of the roots can fail to be full only in presence of complex roots. 
%and Fig.~\ref{fig:qq3}  we show the basin of attraction for some map conjugated to
%$\ff$ and $\fg$ for  nine values of the parameters $a$ and $b$.  The main point that we want to make with
%these plots is that it appears to be a dichotomy: when $\ff$ of $\fg$ have four real roots, no black is visible,
%namely the Julia set appears to have zero measure and the basins seem to be simply connected; on the contrary,
%when $\ff$ or $\fg$ have only two real roots, there is a large black component that replaces the ``colors'' lost
%and seems to have non-zero measure.
%Finally, in the second two columns of Fig.~\ref{fig:misc} we show
%approximations of the Julia set of some of the cases shown above obtained through backward iterations.
%
%\subsection{$\boldsymbol{u(x,y)=(xy-1,(x-5)^2-y^2-1)}$}
%
\begin{figure}
  \centering
    \includegraphics[width=13.5cm]{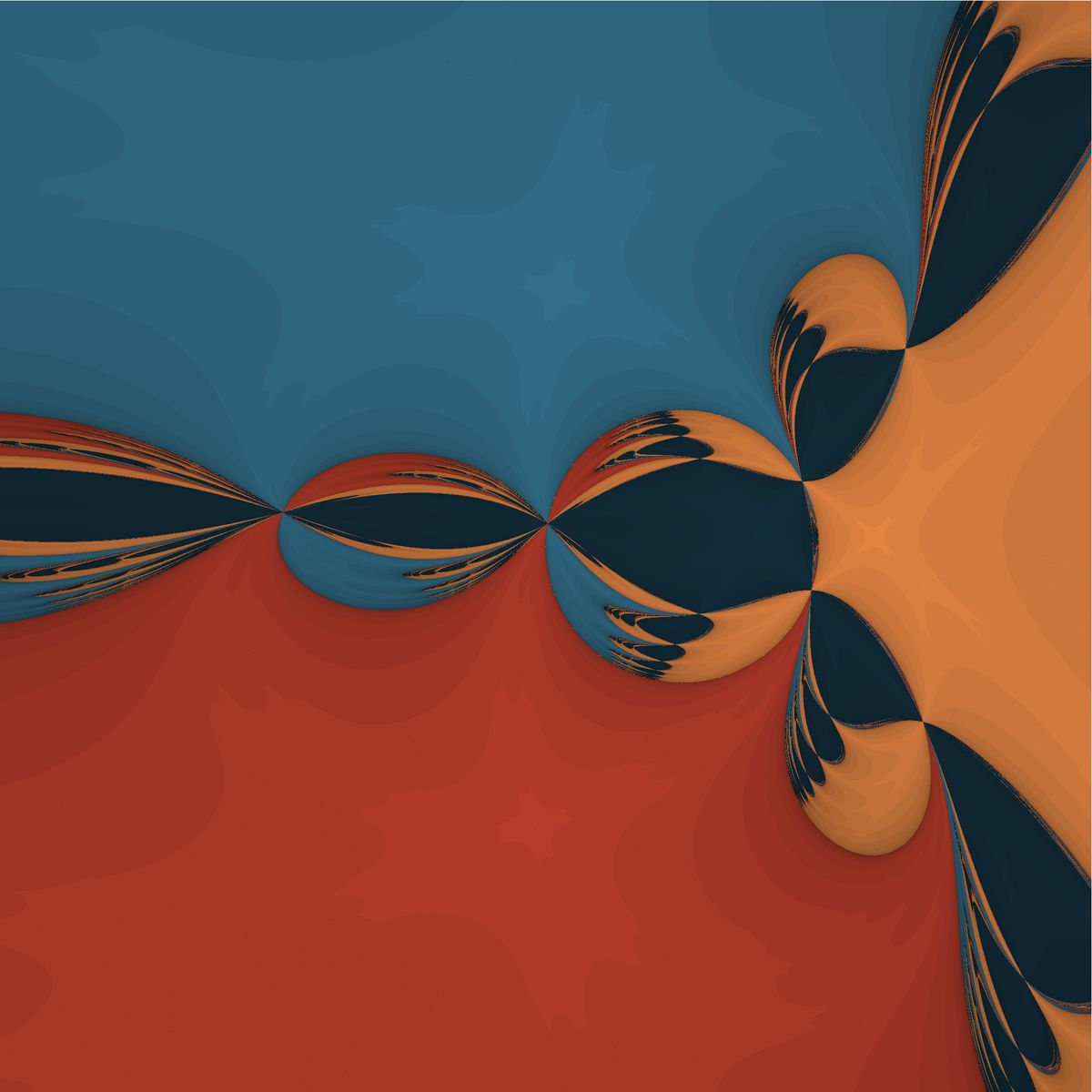}
  \caption{%                                                                                                                                                                                
    \em
    Basins of attraction of $\fN_{5,0;1}$
%    the Newton map of $u(x,y)=(xy-1,(x-5)^2-y^2-1)$
    in the square $[-10,10]^2$.
    The four real roots of $\ff_{5,0;1}$ are the points $p_1\simeq(-0.20,-5.1)$, $p_2\simeq(0.21,4.7)$,
    $p_3\simeq(4.0,0.25)$ and $p_4\simeq(6.0,-0.17)$; the corresponding basins of attraction have been colored,
    respectively, in red, cyan, blue and mustard. Darker shades correspond to higher convergence time.
  }
  \label{fig:qq3c}
\end{figure}
\begin{figure}
  \centering
  \begin{tabular}{cc}                                                                                                                                                                  
    \includegraphics[width=6cm]{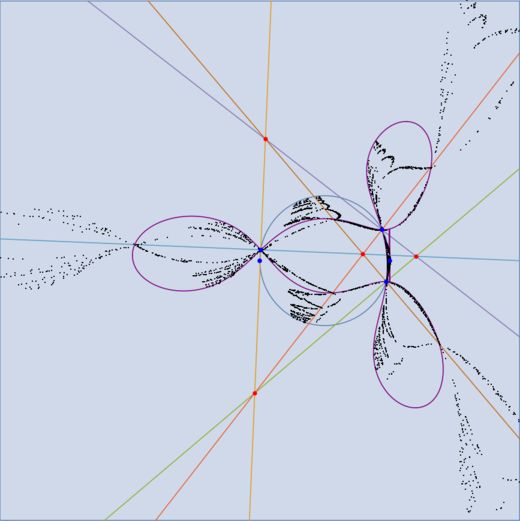}&\includegraphics[width=6cm]{qq3c}\\
    \includegraphics[width=6cm]{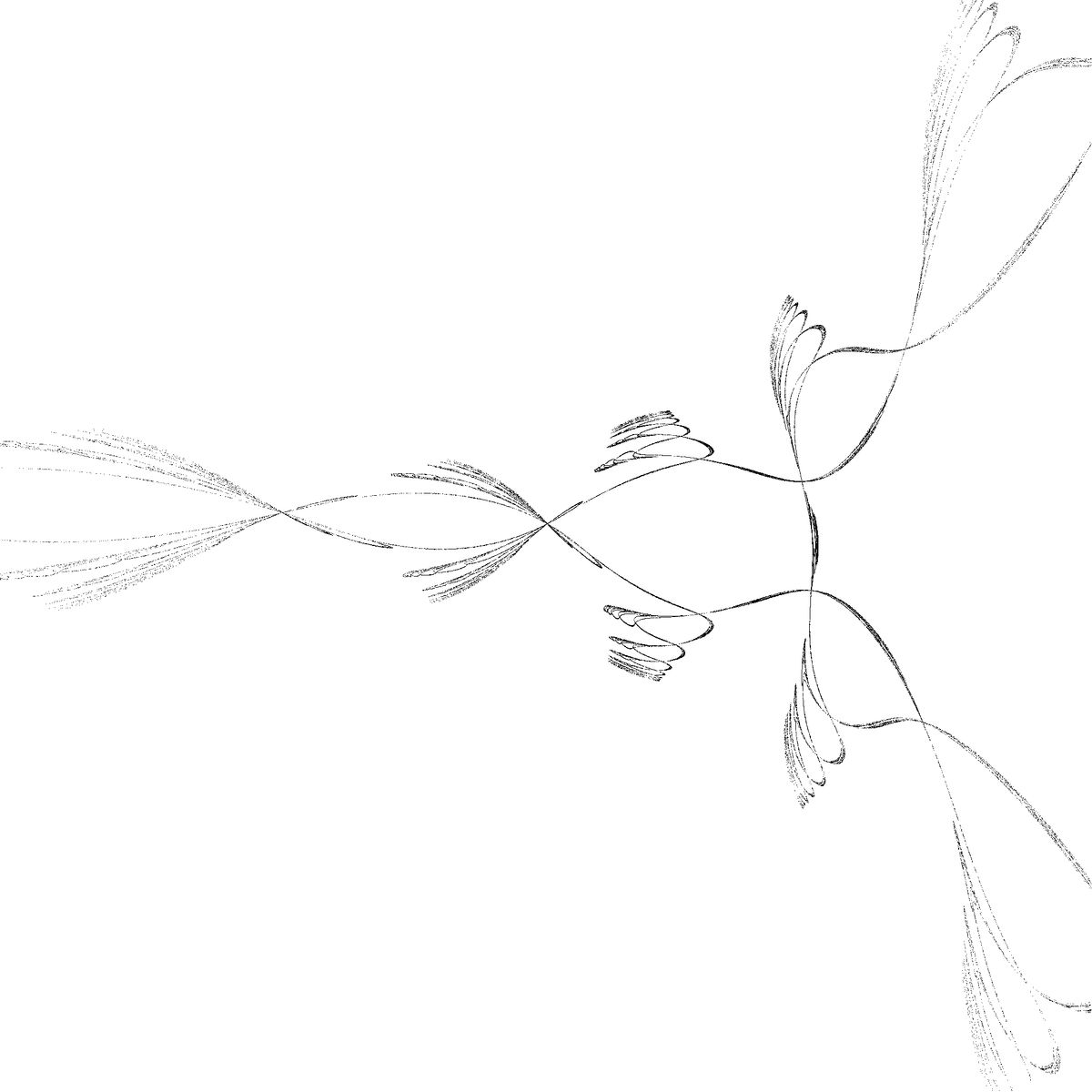}&\includegraphics[width=6cm]{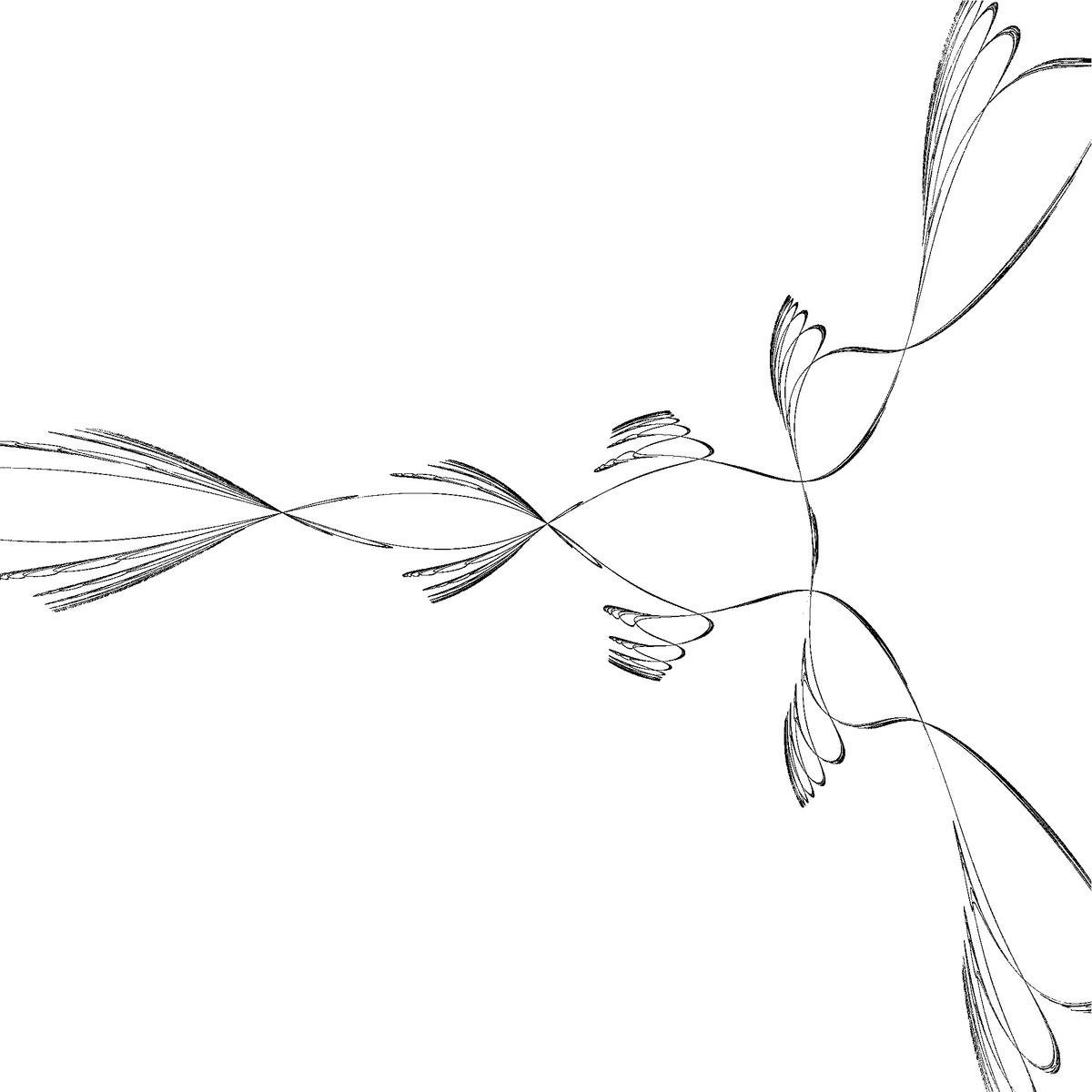}\\
  \end{tabular}
  \caption{%
    \em
    (Top, left) Some important elements of the dynamics of $\fN_{5,0;1}$ in the square $[-10,10]^2$:
    fixed points (red), indeterminacy points (blue), the six invariant lines, the set $\fZ_{5,0;1}$ (light blue)
    and its first counterimage (purple). The black points are the counterimages via $\fN_{5,0;1}$
    of a single point of the plane up to the 5th level of iteration.
    (Top, right) Basins of $\fN_{5,0;1}$ in the square $[-10,10]^2$.
    (Bottom, left) First $3\cdot10^5$ points of a random backward orbit of a random point
    under the two branches of $\fN_{5,0;1}^{-1}$.
    (Bottom, right) This picture shows the set $\fN_{5,0;1}^{-20}(-10,-3.6)$. These two bottom
    pictures suggest that through $\alpha$-limits we can only get the set of irregular points of the
    Julia set
  }
  \label{fig:qq3c1}
\end{figure}
\subsection{The map $\fN_{5,0;1}$}
%
%{\bf In Fig.~\ref{fig:qq2} we consider the case of $\boldsymbol{h(x,y)=(xy-1,(x-5)^2-y^2-1)}$.}
The Newton map
%$\ff_{5,0;1}$ has the four roots $p_1\simeq(-0.20,-5.1)$, $p_2\simeq(0.21,4.7)$, $p_3\simeq(4.0,0.25)$ and $p_4\simeq(6.0,-0.17)$. Its Newton map is
$$
\fN_{-5,0;1} = \left(\frac{x^3  + x y^2 + 2 y - 24 x}{2 (x^2 + y^2 - 5x)},\frac{y^3 + x^2 y - 10 x y + 2 x + 24 y - 10}{2 (x^2 + y^2 - 5 x)}\right)
$$
has four attractive fixed points $p_1\simeq(-0.20,-5.1)$, $p_2\simeq(0.21,4.7)$, $p_3\simeq(4.0,0.25)$ and
$p_4\simeq(6.0,-0.17)$, corresponding to the four roots of $\ff_{5,0;1}$, and restricts to the identity on the circle
at infinity. The four bounded fixed points
are super-attractive while the ones at infinity have eigenvalue 1 in the direction of the circle at infinity and a repulsive
eigenvalue 2 in some transversal direction depending on the point. The five points of indeterminacy
of this map are all bounded: $(0,0)$, $(5,0)$ and, approximately, $(4.9,-0.81)$, $(4.7,1.2)$ and $(0.035,0.42)$.
%,all belonging to the circle $\fc=\{-5 x + x^2 + y^2=0\}=Z_{N_u}$.
%
\begin{figure}
  \centering
    \includegraphics[width=13.5cm]{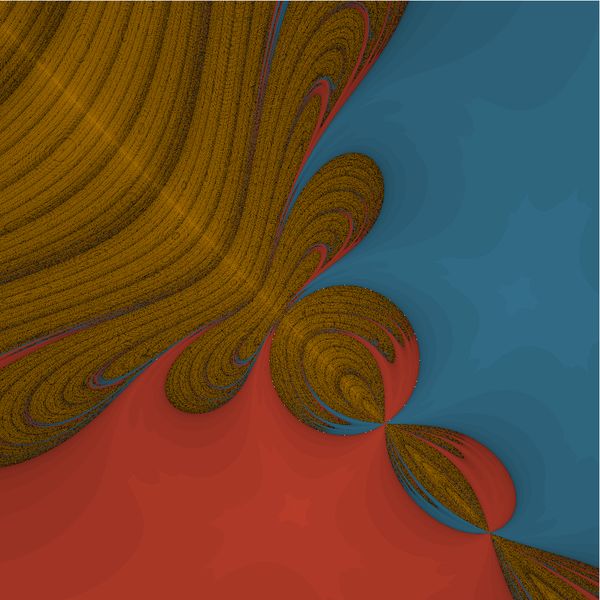}
  \caption{%                                                                                                                                                                                
    \em
    Basins of attraction of $\fN_{3,-4;1}$
%    the Newton map of $v(x,y)=(xy-1,(x-3)^2-(y+4)^2-1)$
    in the square $[-10,10]^2$.
    The two real roots of $\ff_{3,-4;1}$ are the points $p_1\simeq(7.3,0.14)$ and $p_2\simeq(-0.14,-7.0)$,
    the corresponding basins are colored in cyan and red respectively. The basin of the third attractor, a chaotic
    invariant set lying on the ghost line, is colored in gold. Darker shades correspond to higher convergence time.
    The ghost line is visible in the picture as the set of brightest points of the chaotic attractor.
    %    Just as in case of the map $\fN_{-1,2}$ in Sec.~\ref{ss:h}, numerics strongly suggest that in this case the
%    Lebesgue measure of the basins of attraction of the fixed points is not full.
  }
  \label{fig:qq3h}
\end{figure}
\begin{figure}
  \centering
  \begin{tabular}{cc}                                                                                                                                                                  
    \includegraphics[width=6cm]{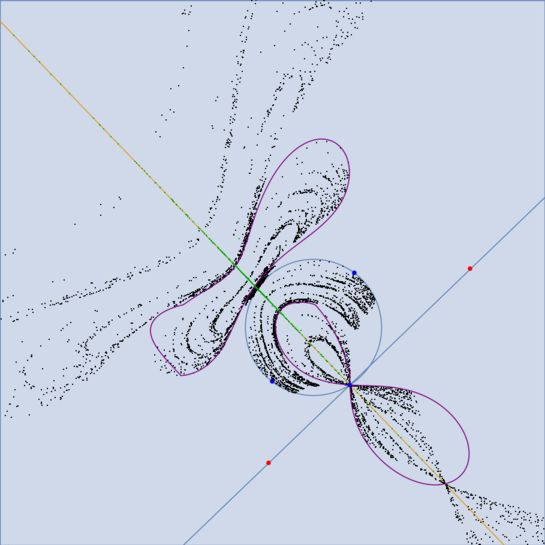}&\includegraphics[width=6cm]{qq3h}\\
    \includegraphics[width=6cm]{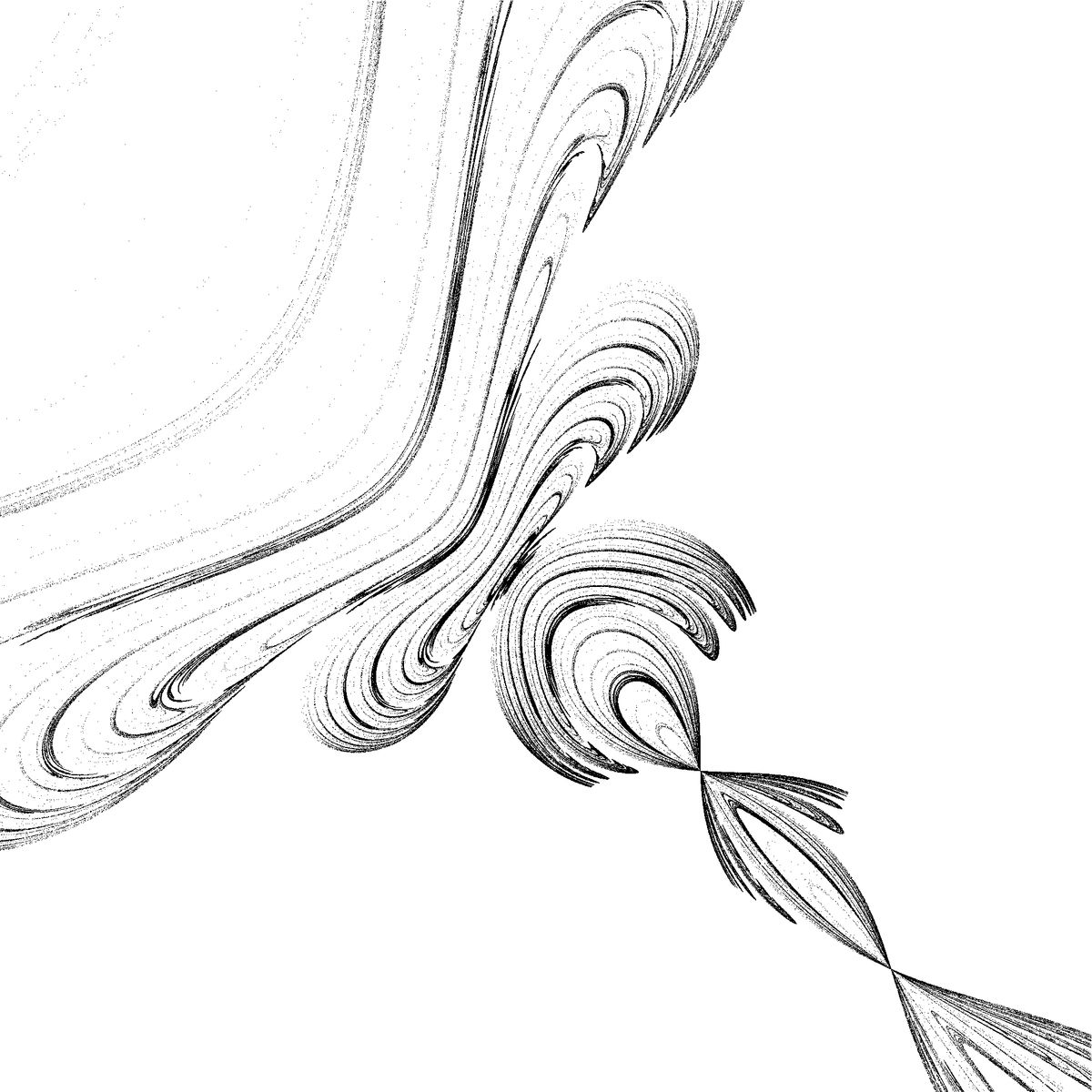}&\includegraphics[width=6cm]{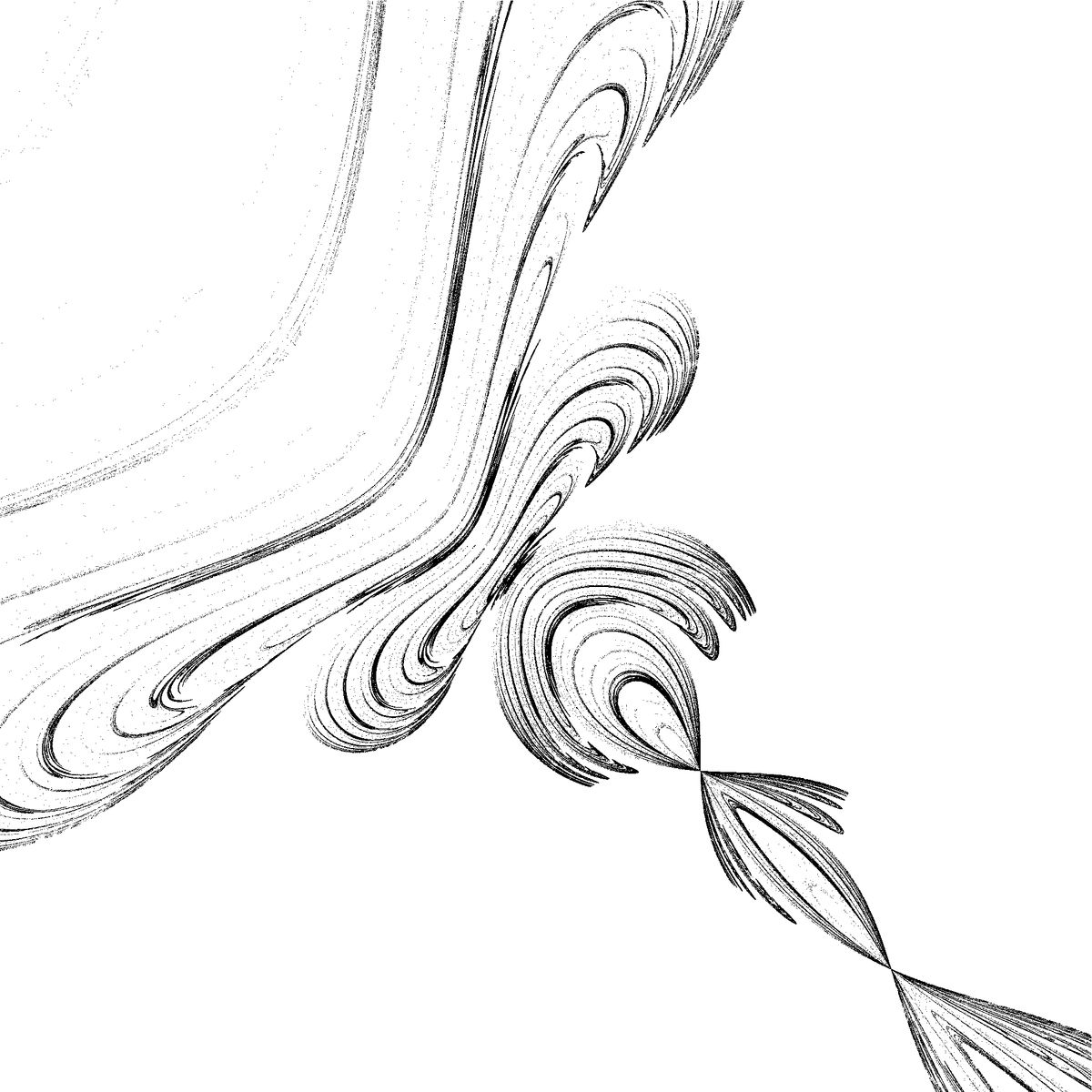}\\
  \end{tabular}
  \caption{%
    \em
     (Top, left) Main elements in the dynamics of $\fN_{3,-4;1}$: fixed points (red), point of indeterminacy (blue),
    invariant line joining the two roots (light blue) and ghost line (orange), the set $\fZ_{3,-4;1}$ (light blue)
    with its first counterimage (purple). The green and black points are, respectively, the first 500 points of
    the forward orbit of a random point in the basin of the chaotic attractor and of the backward orbit of
    a generic point on the plane. %The green points suggest the existence of an attracting Cantor set in the ghost line.
    (Top, right) Basins of attraction of $\fN_{3,-4;1}$ in the square $[-10,10]^2$.
    (Bottom, left) First $3\cdot10^5$ points of a random backward orbit of a point below the light
    blue invariant line under the two branches of $\fN_{3,-4;1}^{-1}$ (in black).
    (Bottom, right) This picture shows the set $\fN_{3,-4;1}^{-20}(-10,-3.6)$. These two bottom
    pictures suggest that through $\alpha$-limits we can only get the set of irregular points of the
    Julia set.
  }
  \label{fig:qq3h1}
\end{figure}

Unlike the parabolic case, the map $\fN_{-5,0;1}$ is surjective and open, so that
{its Fatou and Julia sets are fully invariant, as in the complex case}.
The counterimages of a point $(x_0,y_0)$ are given by 
$$
w_\pm(x_0,y_0) = \left(x_0\pm\sqrt{S_+/2},y_0\mp \sqrt{S_-/2}\right),
$$
where
$$
S_\pm = \sqrt{Q}\pm(24 - 10 x_0 + x_0^2 - y_0^2)
$$
and
$$
Q=4 + 96 x_0^2 - 40 x_0^3 - 8 x_0 y_0 + (-24 + 10 x_0 + x_0^2 + y_0^2)^2\,.
$$
Every point has exactly two distinct counterimages except for the fixed points, on which the two counterimages coincide.
In Fig~\ref{fig:qq3c1} (top, left), we show the fixed (red) and indeterminacy (blue) points, the six invariant lines joining the fixed points,
the circle of zeros of the denominator of $\fN_{-5,0;1}$ (in blue) and its first counterimage under $\fN_{-5,0;1}$ (in purple). The black
dots are the $2^{12}$ points of $N_h^{-12}(-10,-0.23)$ and show how the curves and points drawn fit with the Julia set.
Once again the picture suggests that $\fJ_{-5,0;1}$ is the alpha-limit of $\fZ_{-5,0;1}$.

Now, let $U$ be a small enough neighborhood of the four fixed points and $X=\RPt\setminus U$ . Then
$\fN_{-5,0;1}(X)\supset X$ and so the set $\lim_{n\to\infty}\fN_{-5,0;1}^{-n}(X)$  is the compact non-empty repellor of all points
that do not leave $X$ under forward iterations of $\fN_{-5,0;1}$. Just taking the $\alpha$-limit of a single point
suggests that this limit is the set of irregular points of $\fJ_{-5,0;1}$ for any non-fixed point of $\fN_{-5,0;1}$ (see Fig.~\ref{fig:qq3c1}, bottom right).
Similarly, $w_\pm(X)\subset X$ and so the these two maps define a IFS on $X$. The numerical evaluation
of random orbits under $\cI$ (see Fig.~\ref{fig:qq3c1}, bottom left) suggests that the irregular points of the Julia set is the unique compact
invariant set of $\cI$. 
%of the two segments of the hyperbola $xy=z^2$ (in homogeneous coordinates), namely the bounded one
%between $p_3$ and $p_4$ and the one between $p_1$ and $p_2$ containing the point at infinity $[0:1:0]$, have
%a single counterimage (all of which at infinity)
%and all other points have two (all finite). 

%{\bf In Fig.~\ref{fig:qq2} we consider the case of $\boldsymbol{l(x,y)=(xy-1,(x-3)-(y+4)^2-1)}$.}
%\subsection{$\boldsymbol{v(x,y)=(xy-1,(x-3)-(y+4)^2-1)}$}
%
\subsection{The map $\fN_{3,-4;1}$}
The Newton map
$$
\fN_{3,-4;1} = \left(\frac{8 + 8 x + x^3 + 2 y + 8 x y + x y^2}{2 (-3 x + x^2 + 4 y + y^2)},
\frac{-6 + 2 x - 8 y - 6 x y + x^2 y + y^3}{2 (-3 x + x^2 + 4 y + y^2)}\right)
$$
has properties very similar to those of the parabolic map $\fN_{-1,2}$.
%$h$. Like $h$

Like $\fN_{-1,2}$, also this map has only two attractive fixed points,
%two roots, the points $p_1\simeq(7.3,0.14)$ and $p_2\simeq(-0.14,-7.0)$,
namely $p_1\simeq(7.3,0.14)$ and $p_2\simeq(-0.14,-7.0)$,
and, correspondingly, only two invariant straight lines and three points of indeterminacy,
all bounded:  $(0,-4)$, $(3,0)$ and, approximately, $(2.8,-4.1)$.
%The corresponding Newton map is
%whose homogeneous version restricts to the identity on the circle at infinity.
%The two invariant lines of $N_v$, the one joining the two roots and a second one coming from the pair of complex roots,
%and only three points of indeterminacy, namely $(0,-4)$, $(3,0)$ and approximatively $(2.8,-4.1)$.
Like $\fN_{5,0;1}$, this map is surjective and open and every point except the fixed ones has two counterimages.
In Fig.~\ref{fig:qq3h1} we show all these elements plus the circle $\fZ_{3,-4;1}$ and its first counterimage under $\fN_{3,-4;1}$.

Just as in case of $\fN_{-1,2}$, numerics strongly suggests the presence of a chaotic attractor $K$ lying on the ghost line,
whose basin of attraction is shown in gold in Fig.~\ref{fig:qq3h}. The orbit of a point in $\cF(K)$ is shown in green in Fig.~\ref{fig:qq3h1}.
Like $\fN_{5,0;1}$, both random backward orbits (see Fig.~\ref{fig:qq3h1}, (bottom, left)) and the $\alpha$-limit of a generic point
(see Fig.~\ref{fig:qq3h1}, (bottom, right)) appear to converge to the set of non-regular points of the Julia set.
%numerics suggest that the $\omega$-limit of points in the large black areas
%shown in the basins picture do not correspond to any of the kinds of Fatou components that can appear in
%the complex case but rather to an attracting chaotic invariant set $K$ lying on the ghost invariant line. 
%              
\begin{figure}
  \centering
  \begin{tabular}{ccc}                                                                                                                                                                  
    \includegraphics[width=4.25cm]{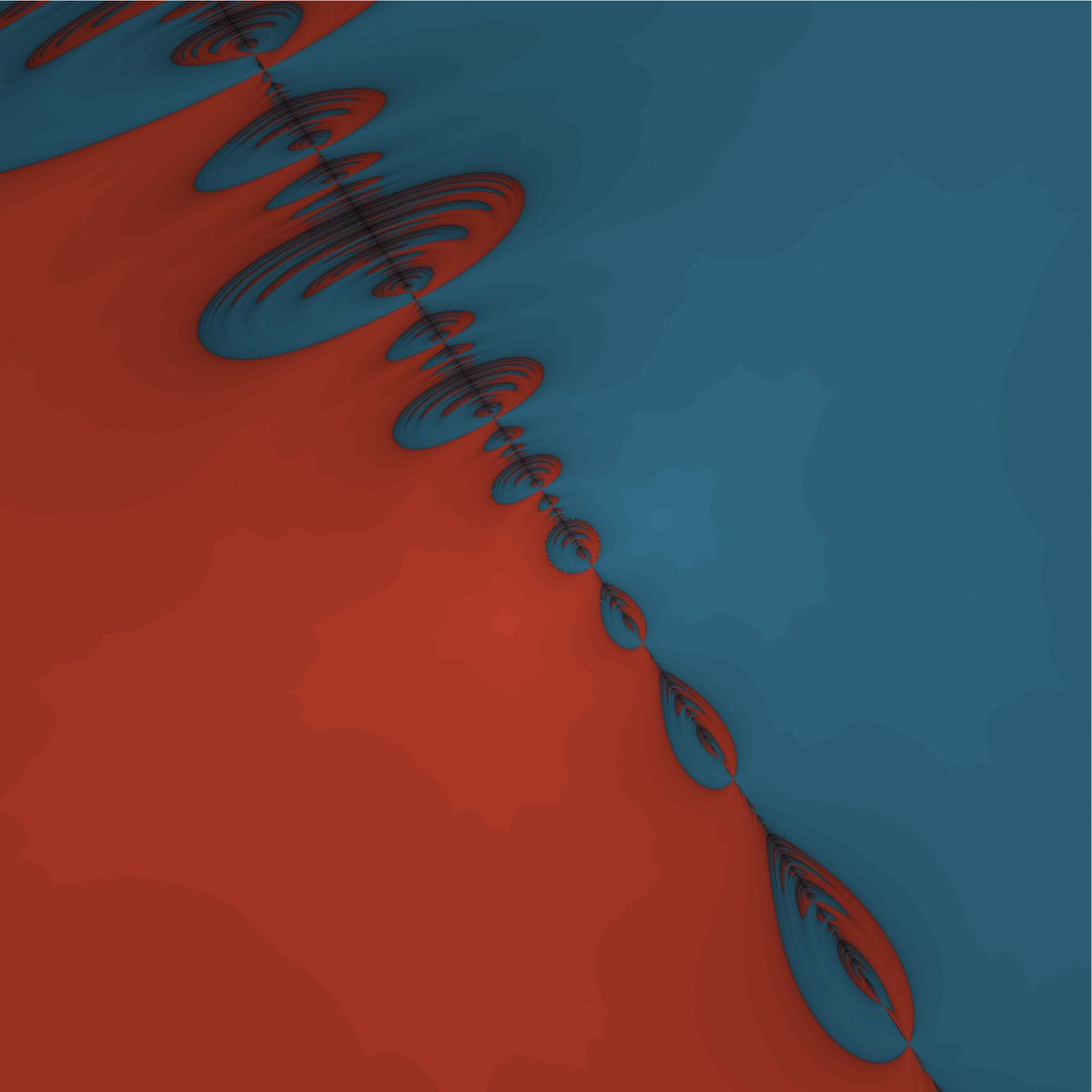}&\includegraphics[width=4.25cm]{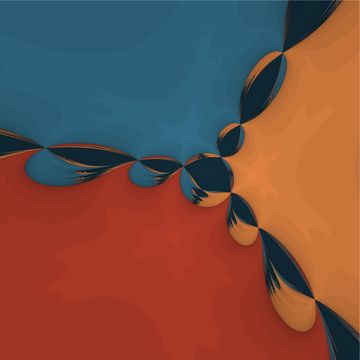}&\includegraphics[width=4.25cm]{qq3c}\\
    \includegraphics[width=4.25cm]{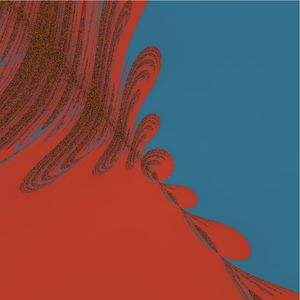}&\includegraphics[width=4.25cm]{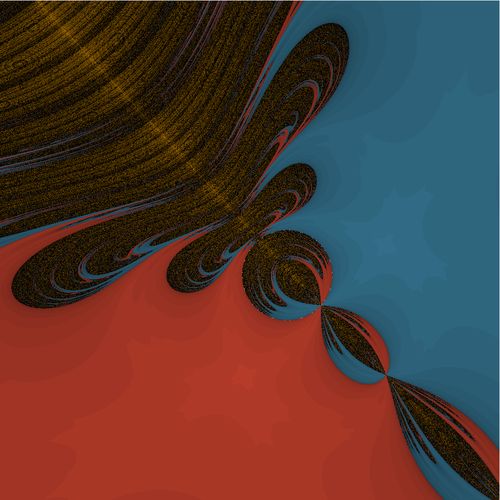}&\includegraphics[width=4.25cm]{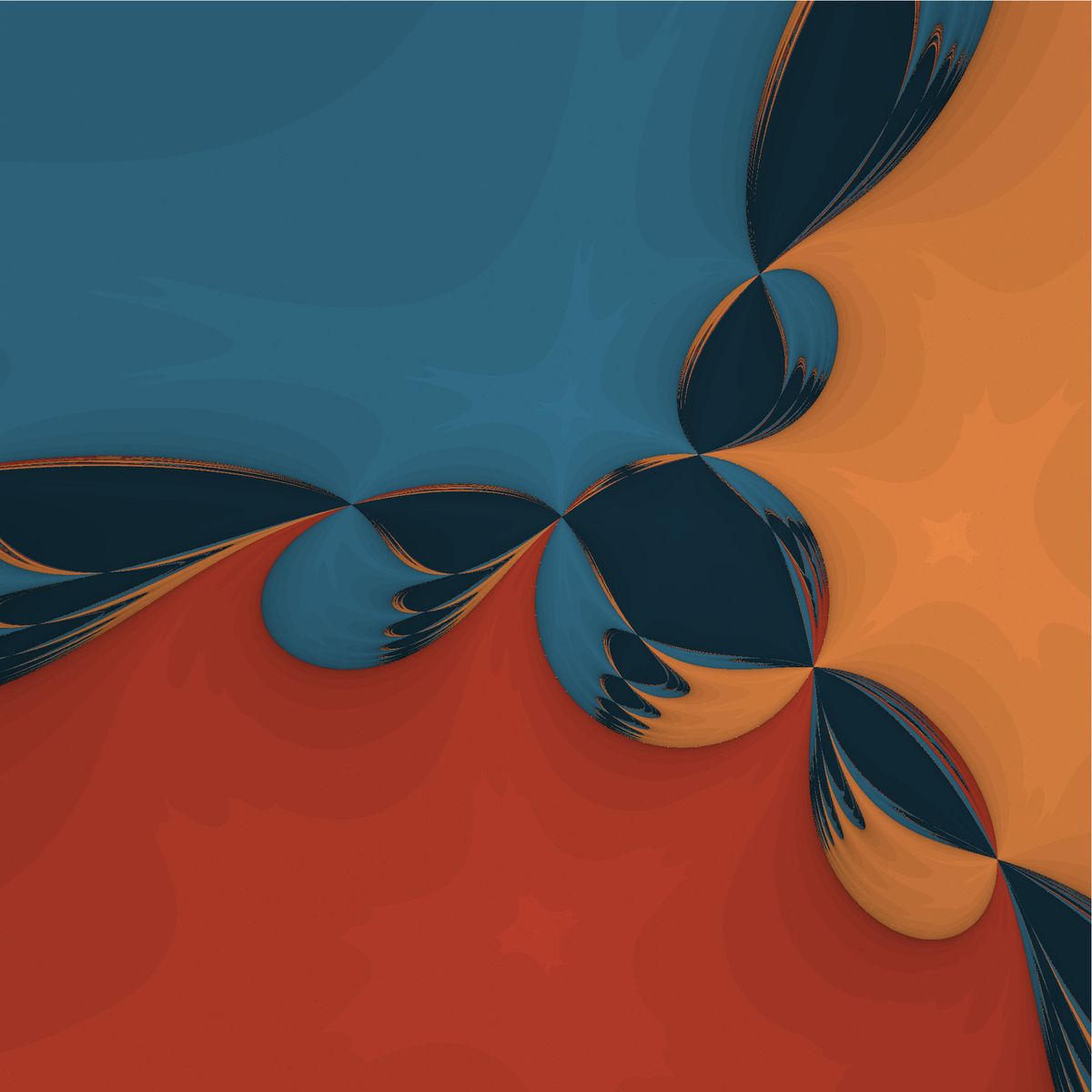}\\
    \includegraphics[width=4.25cm]{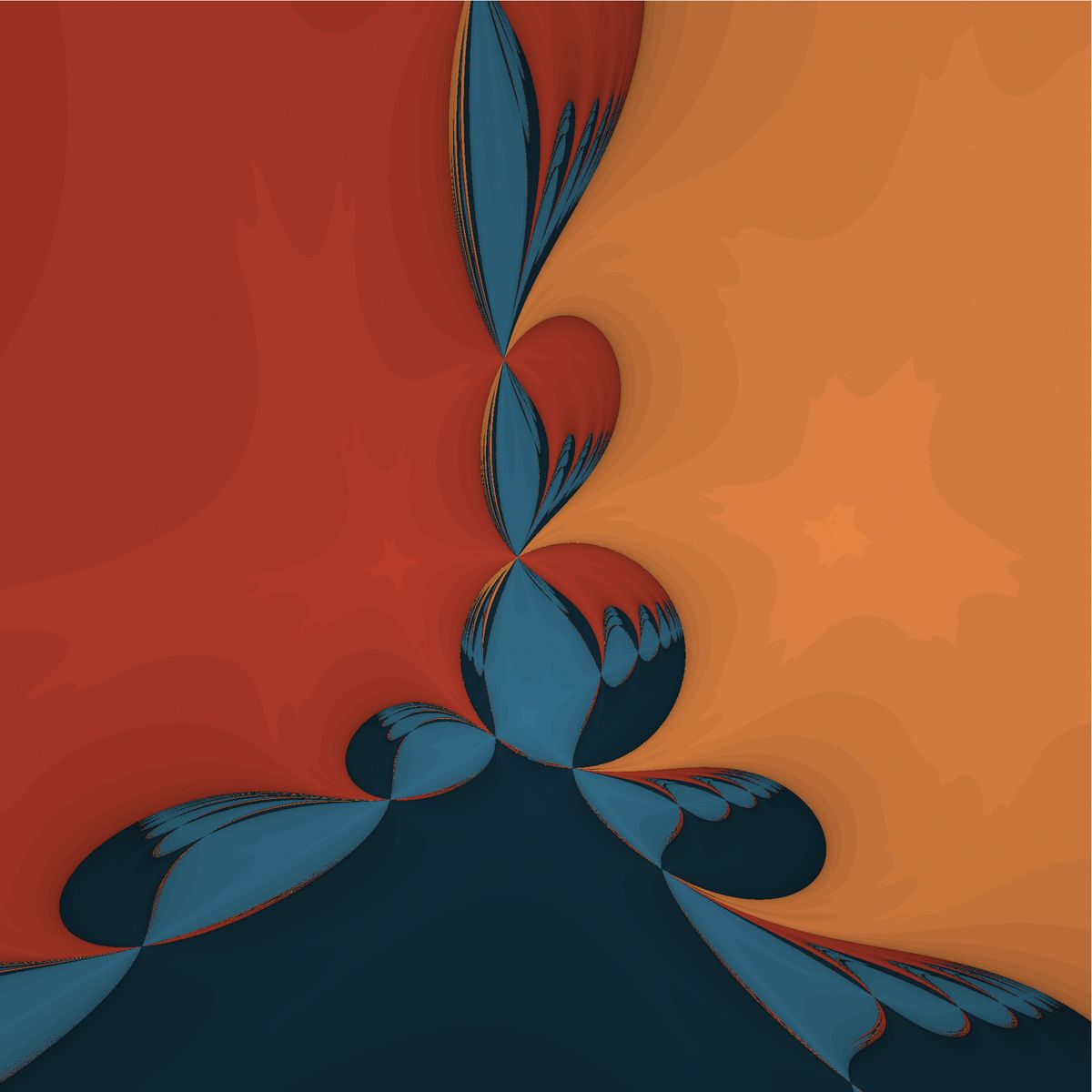}&\includegraphics[width=4.25cm]{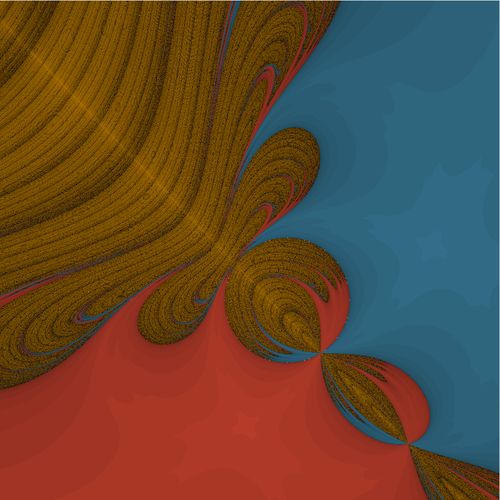}&\includegraphics[width=4.25cm]{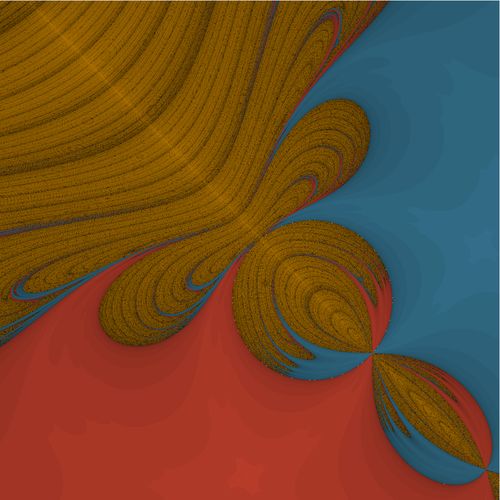}\\
  \end{tabular}
  \caption{%
    \em
    Basins of attraction for the hyperbolic Newton maps $\fN_{x_0,y_0;1}$ corresponding to the values $x_0=1,3,5$ and $y_0=-4,-2,0$.
    The numerical results strongly suggest that the union of the basins of attraction of the roots of the corresponding polynomials
    $\ff_{x_0,y_0;1}$ has full Lebesgue measure when the map has four roots and that a third, chaotic attractor can arise when the number
    of roots is non maximal. Darker shades correspond to higher convergence time.
  }
  \label{fig:qq3}
\end{figure}
\begin{figure}
  \centering
  \begin{tabular}{ccc}
    \includegraphics[width=4.25cm]{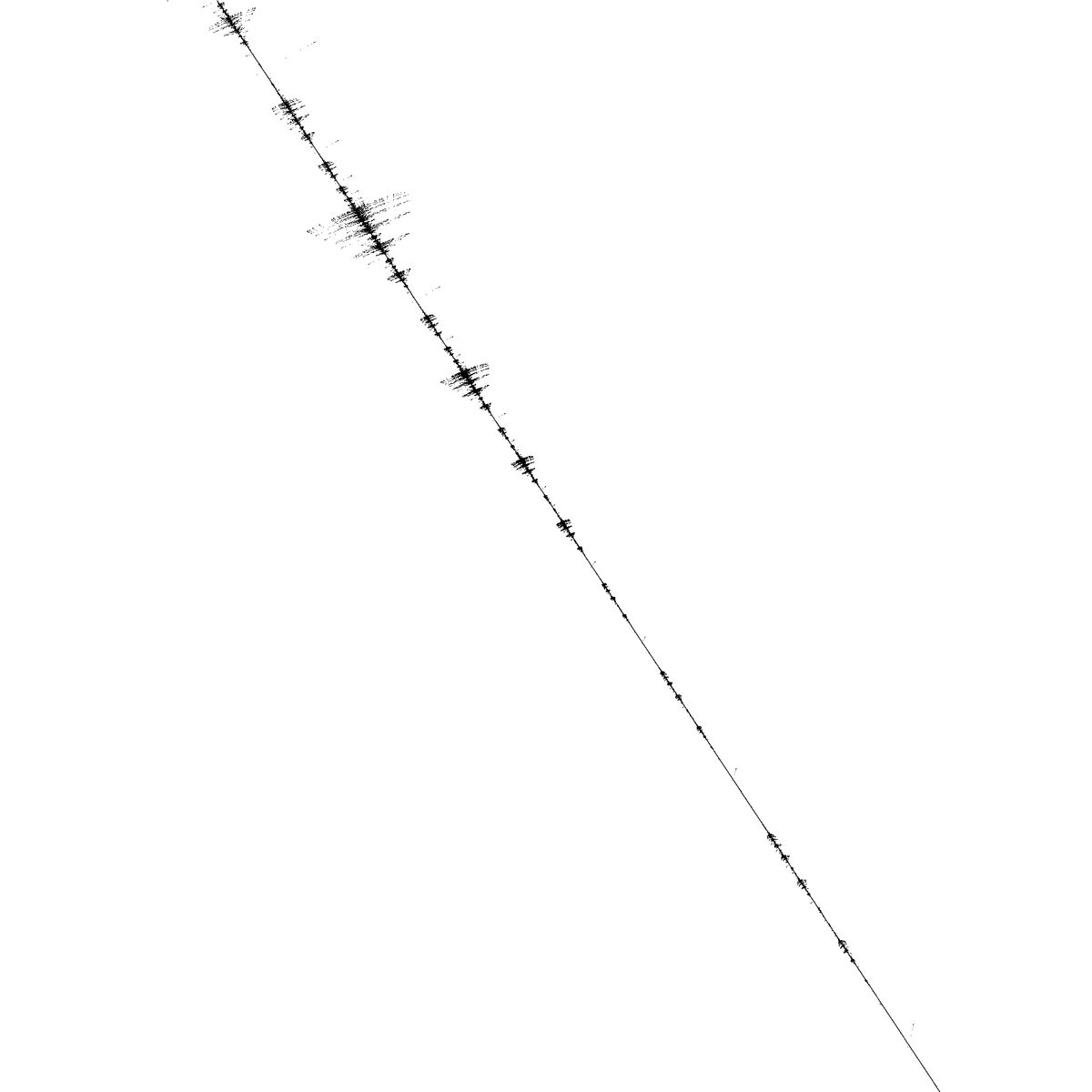}&\includegraphics[width=4.25cm]{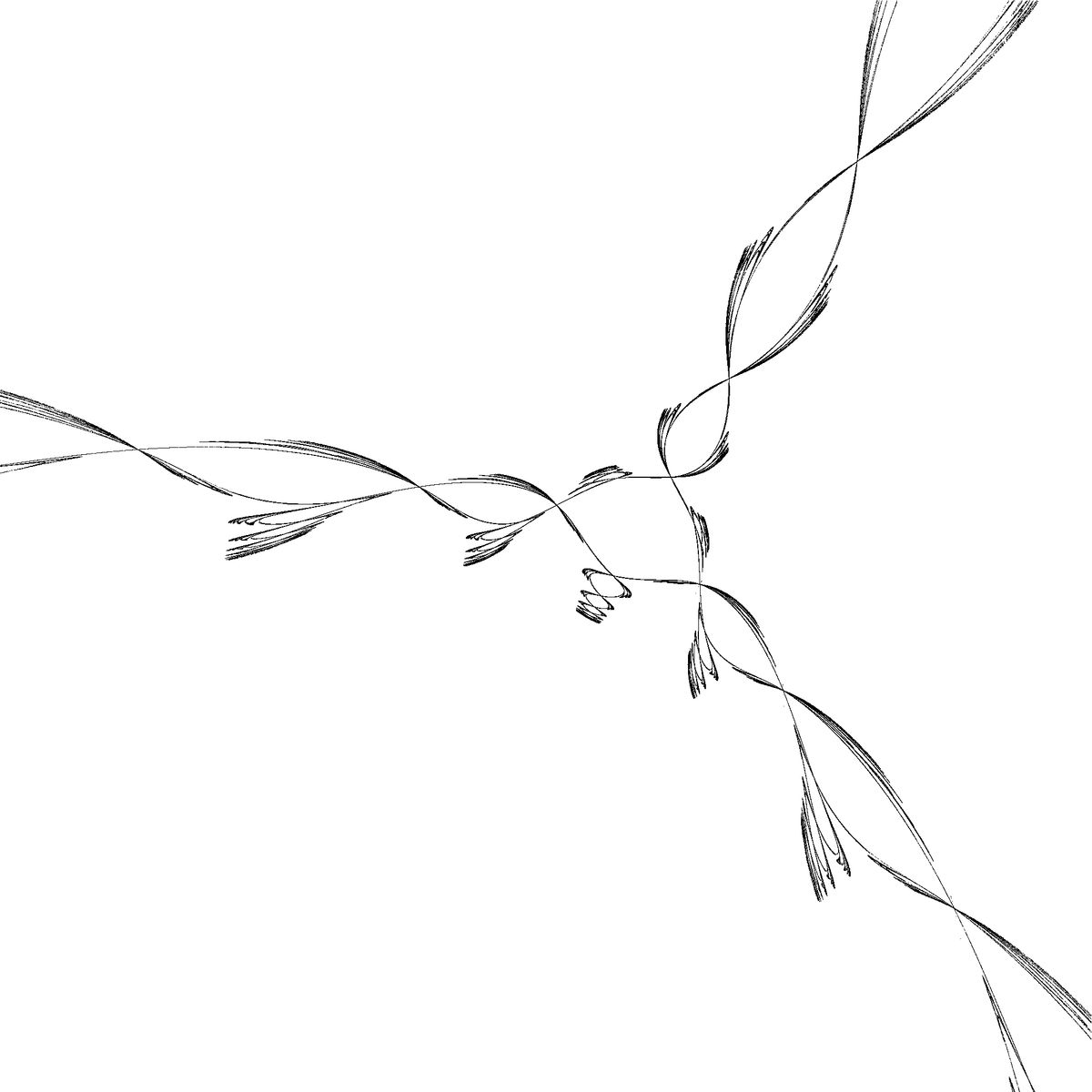}&\includegraphics[width=4.25cm]{ii-qq3_x5-y0}\\
    \includegraphics[width=4.25cm]{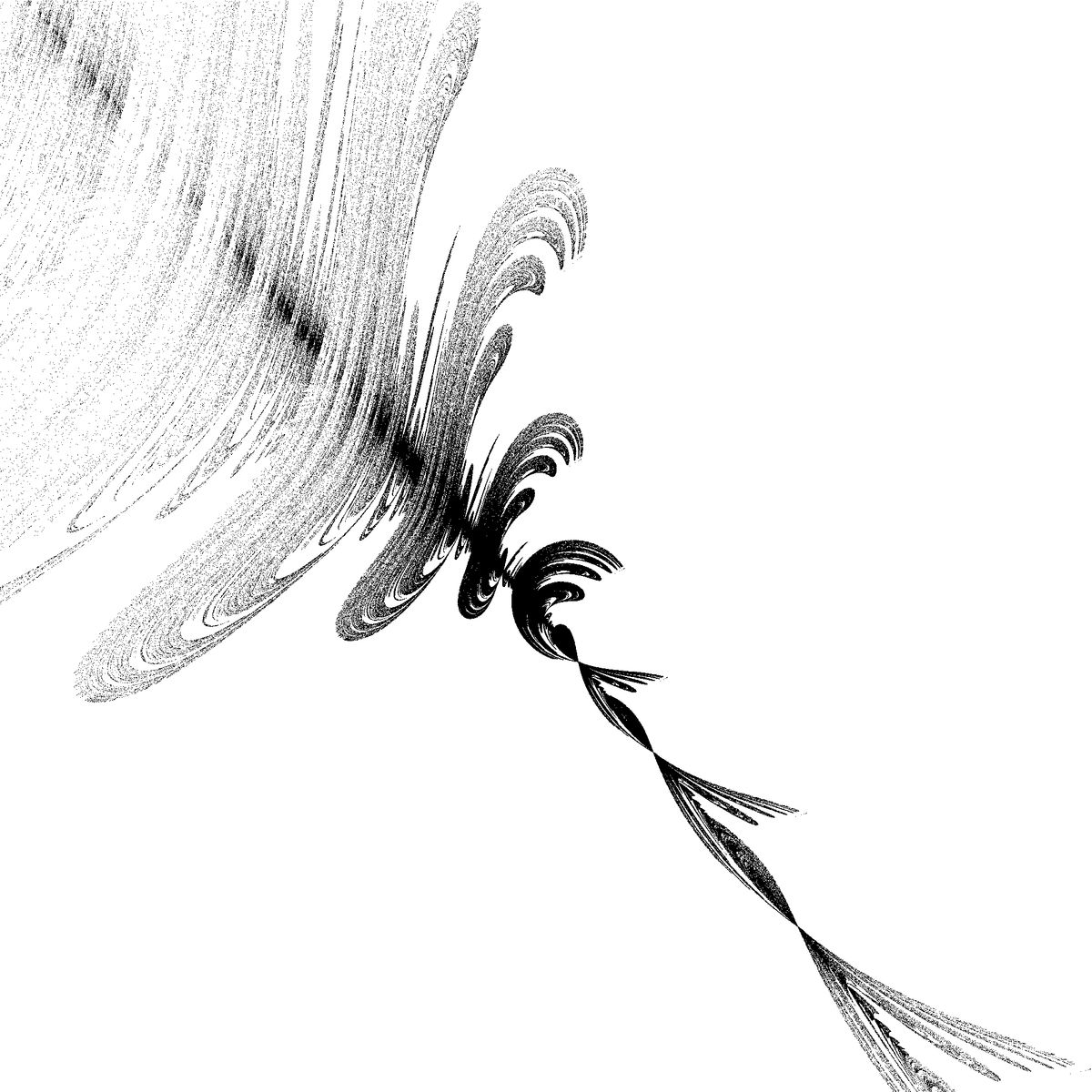}&\includegraphics[width=4.25cm]{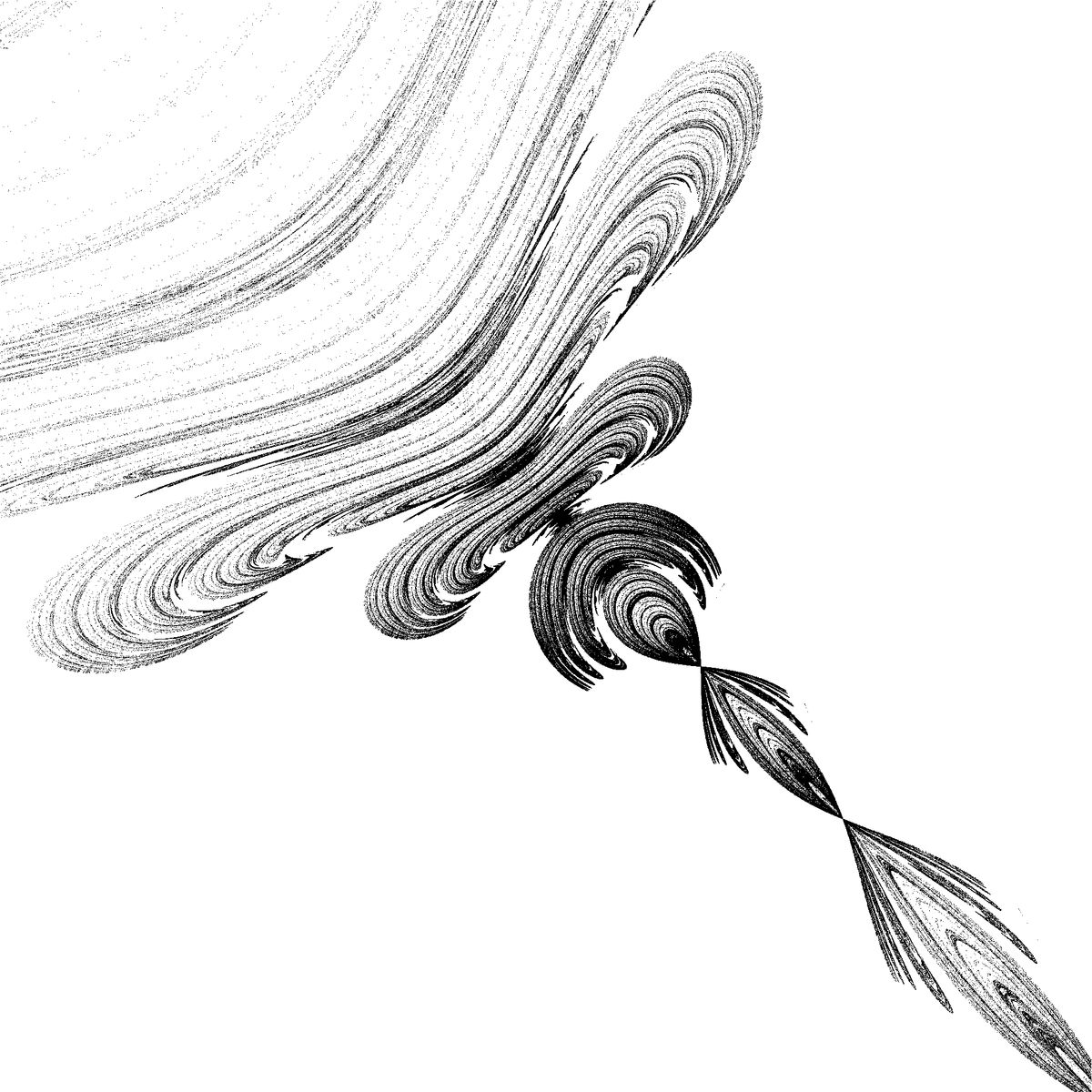}&\includegraphics[width=4.25cm]{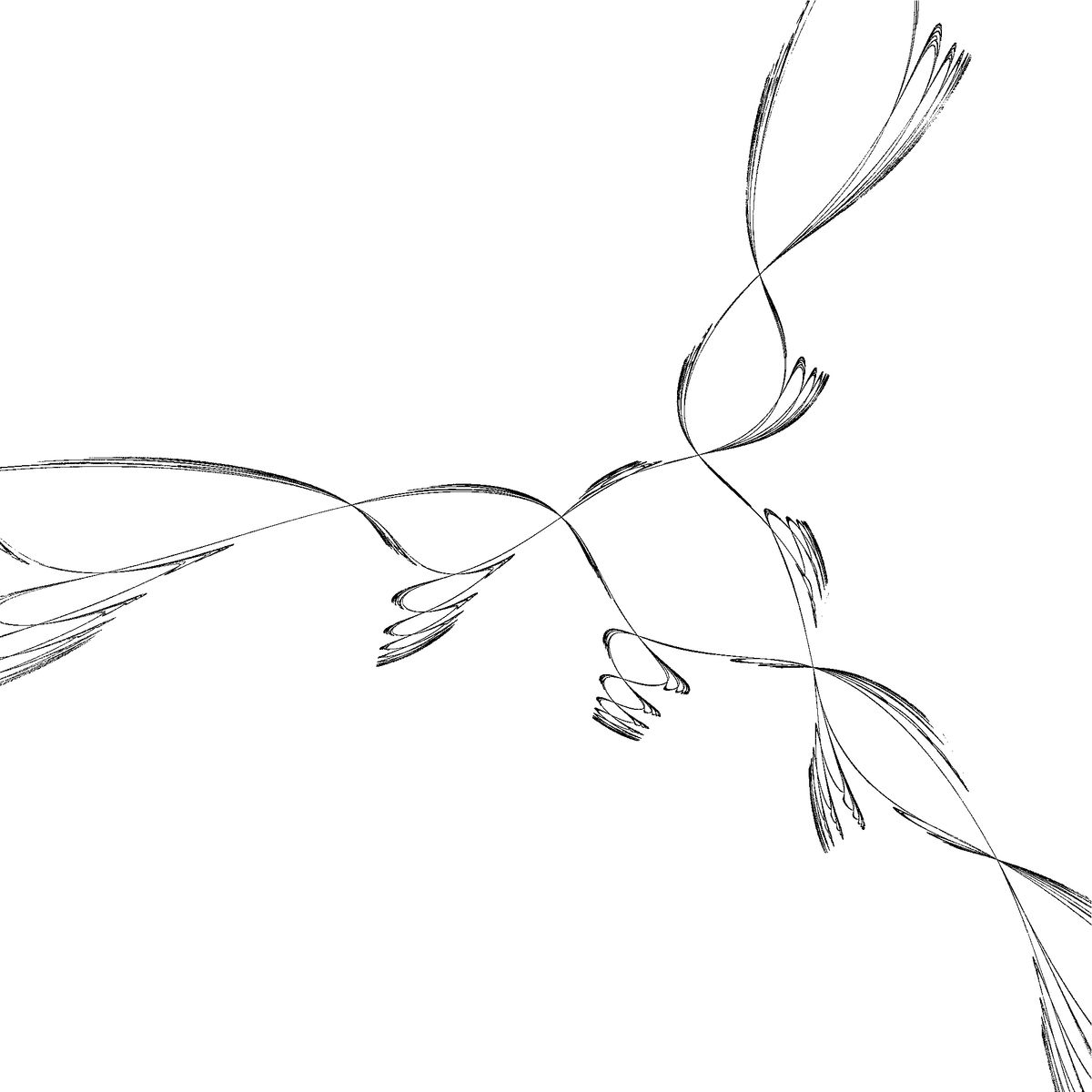}\\
    \includegraphics[width=4.25cm]{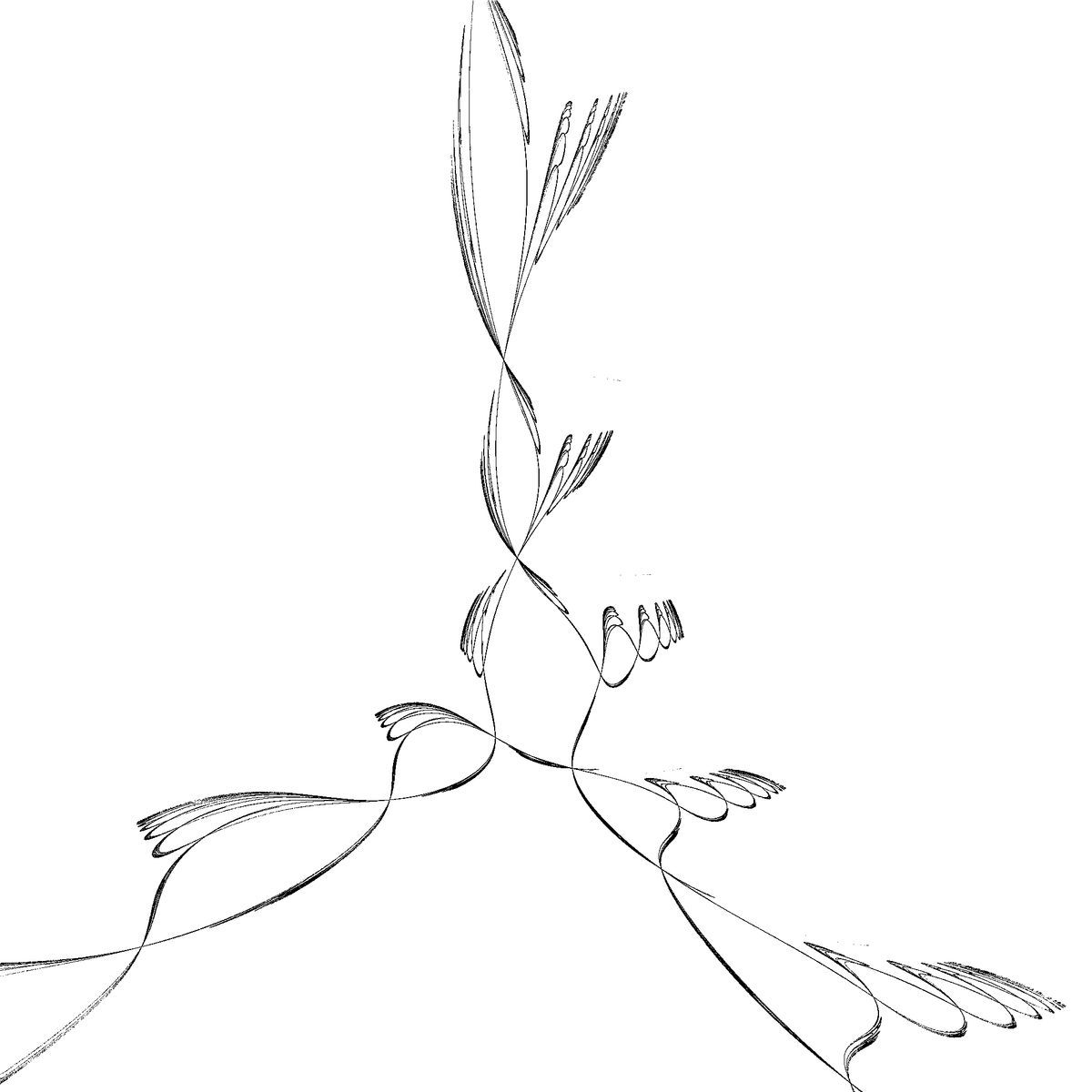}&\includegraphics[width=4.25cm]{ii-qq3_x3-y-4}&\includegraphics[width=4.25cm]{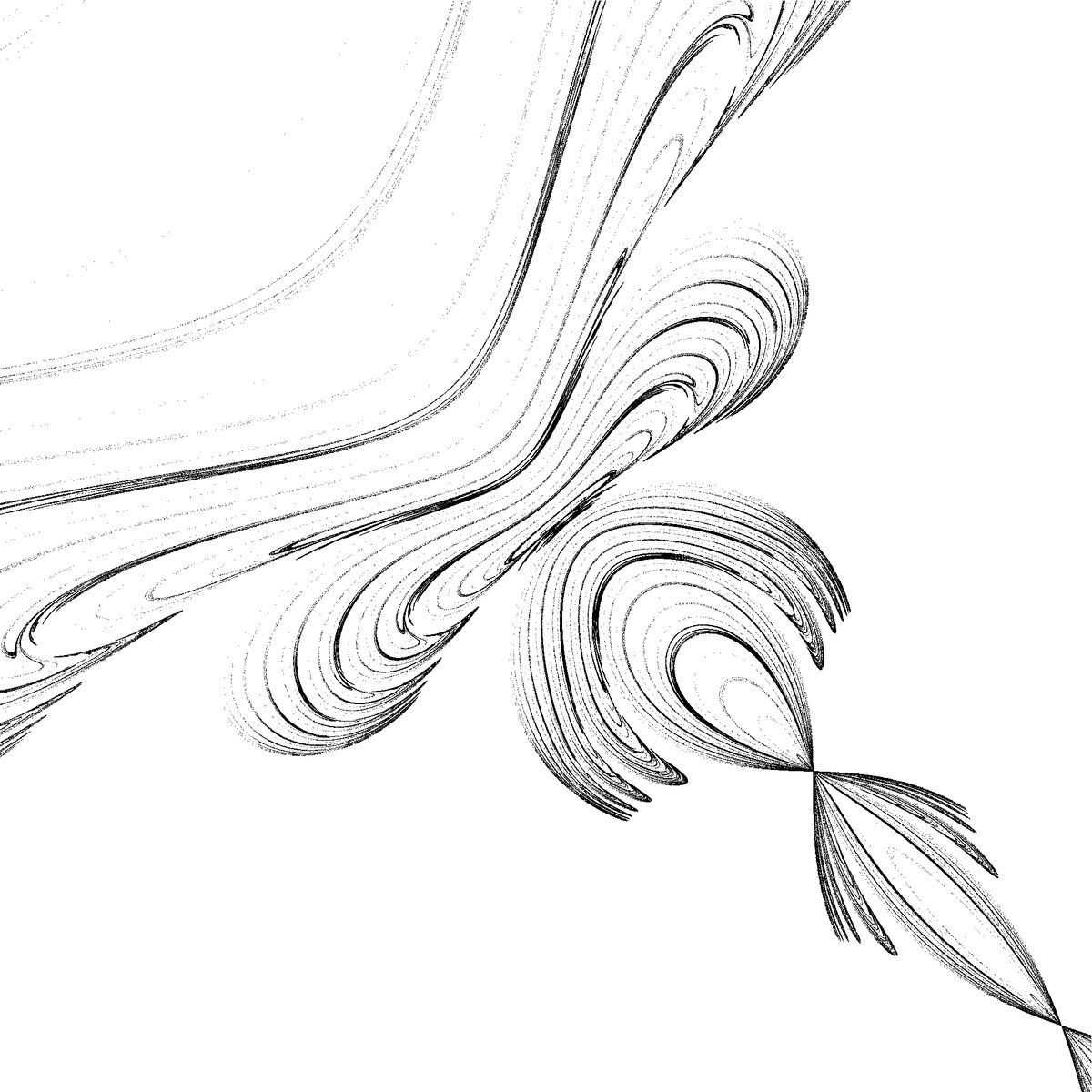}\\
  \end{tabular}
  \caption{%
    \em
    $\alpha$-limit of a generic point belonging to some suitable open subset of $\RPt$ under the parabolic Newton maps
    $\fN_{x_0,y_0;1}$  for the values $x_0=1,3,5$ and $y_0=-4,-2,0$. All this sets are also visible in the previous figure
    as the non-regular points of the basins boundaries.
  }
  \label{fig:qq3ii}
\end{figure}
\subsection{A panorama of maps $\fN_{x_0,y_0;a}$}
In Fig.~\ref{fig:qq3} and~\ref{fig:qq3ii} we show the basins of attraction and some $\alpha$-limits of
several hyperbolic Newton maps $\fN_{x_0,y_0;a}$. Similarly to what we found in the parabolic case, 
whenever $\ff_{x_0,y_0;a}$ has four roots, the union of the corresponding four basins appear to have full
Lebesgue measure in $\RPt$ while when only two roots appear (a generic map $\ff_{x_0,y_0;a}$ has at
least two roots) a third attractor often arises as an invariant subset of the map's ghost line.
While the data shown is relative only to the value $a=1$, we could not detect any particular new phenomenon
for different values of this parameter. The pictures in Fig.~\ref{fig:qq3ii} are fully consistent
with Conjecture~1 and suggest in a more evident way then their parabolic counterparts that regular
points of the Julia set cannot be reached by $\alpha$-limits.
%for  nine values of the parameters $a$ and $b$.  The main point that we want to make with
%these plots is that it appears to be a dichotomy: when $\ff$ of $\fg$ have four real roots, no black is visible,
%namely the Julia set appears to have zero measure and the basins seem to be simply connected; on the contrary,
%when $\ff$ or $\fg$ have only two real roots, there is a large black component that replaces the ``colors'' lost
%and seems to have non-zero measure. Finally, in the second two columns of Fig.~\ref{fig:misc} we show
%approximations of the Julia set of some of the cases shown above obtained through backward iterations.

%
\begin{figure}
  \centering
  \begin{tabular}{cc}                                                                                                                                                                  
    \includegraphics[width=6.3cm]{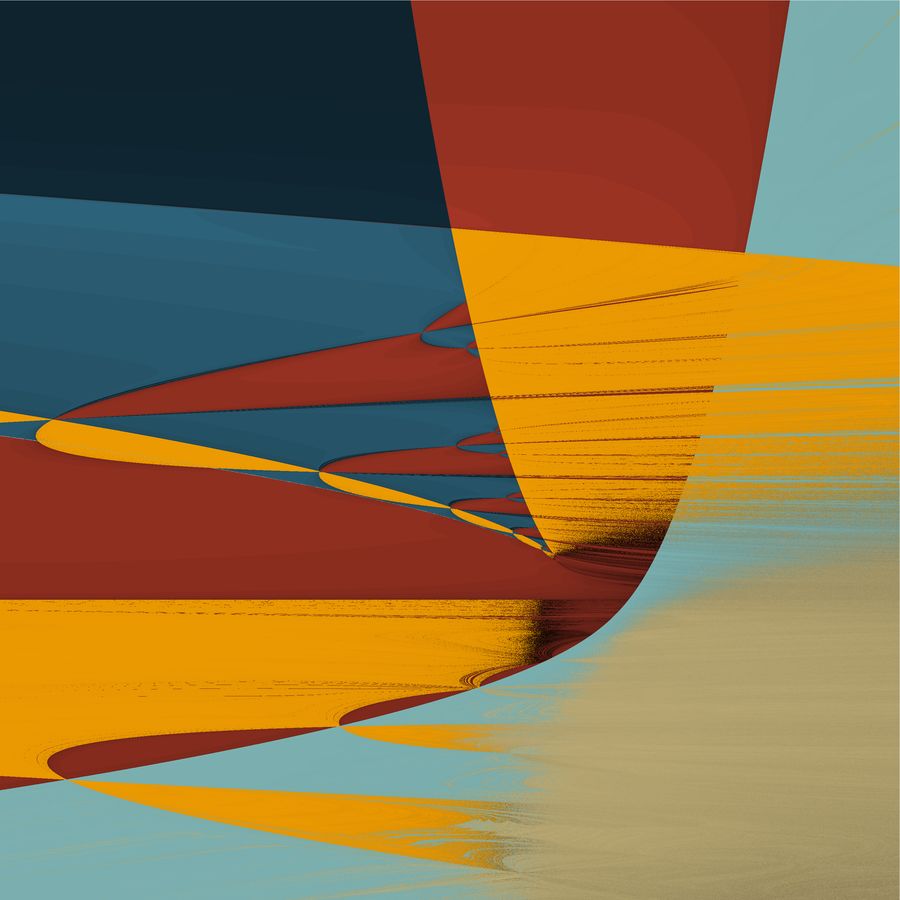}&\includegraphics[width=6.3cm]{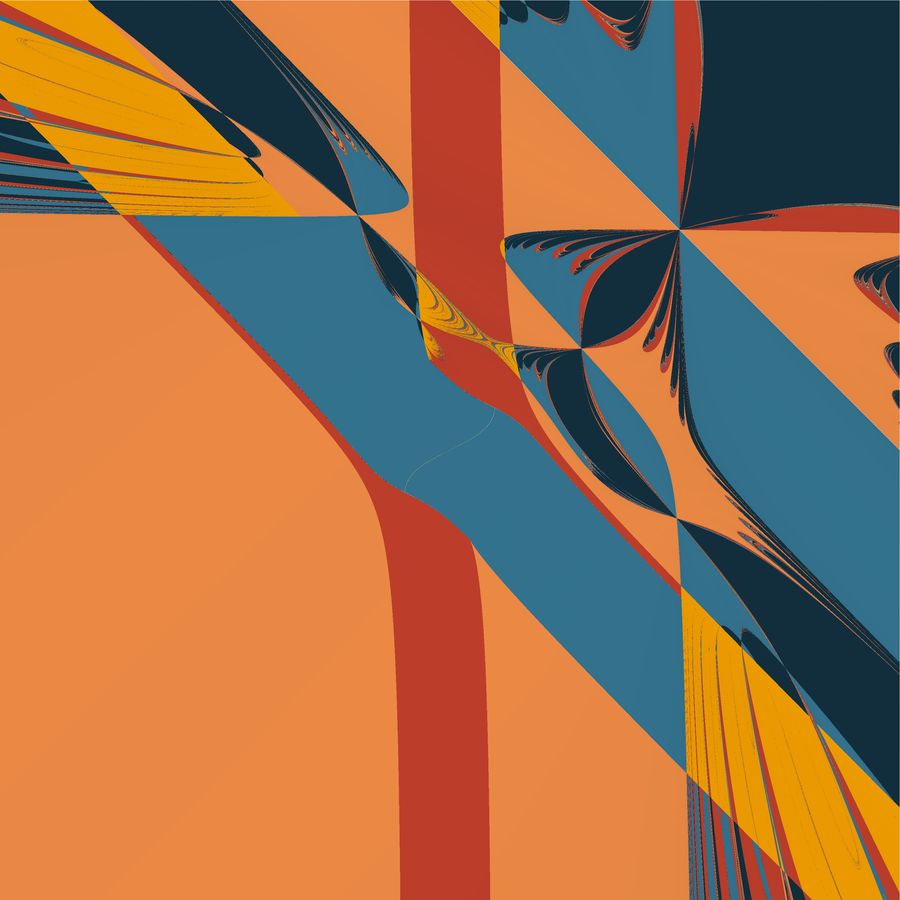}\\
    \includegraphics[width=6.3cm]{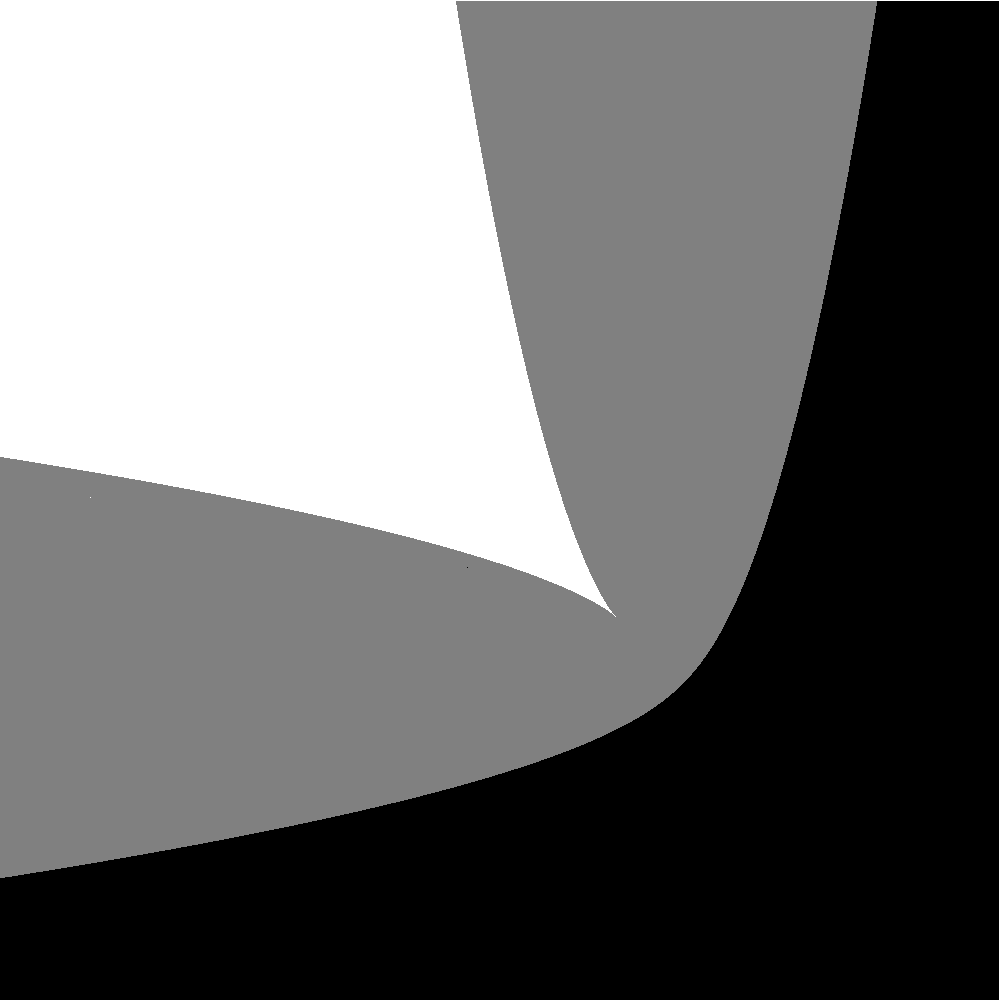}&\includegraphics[width=6.3cm]{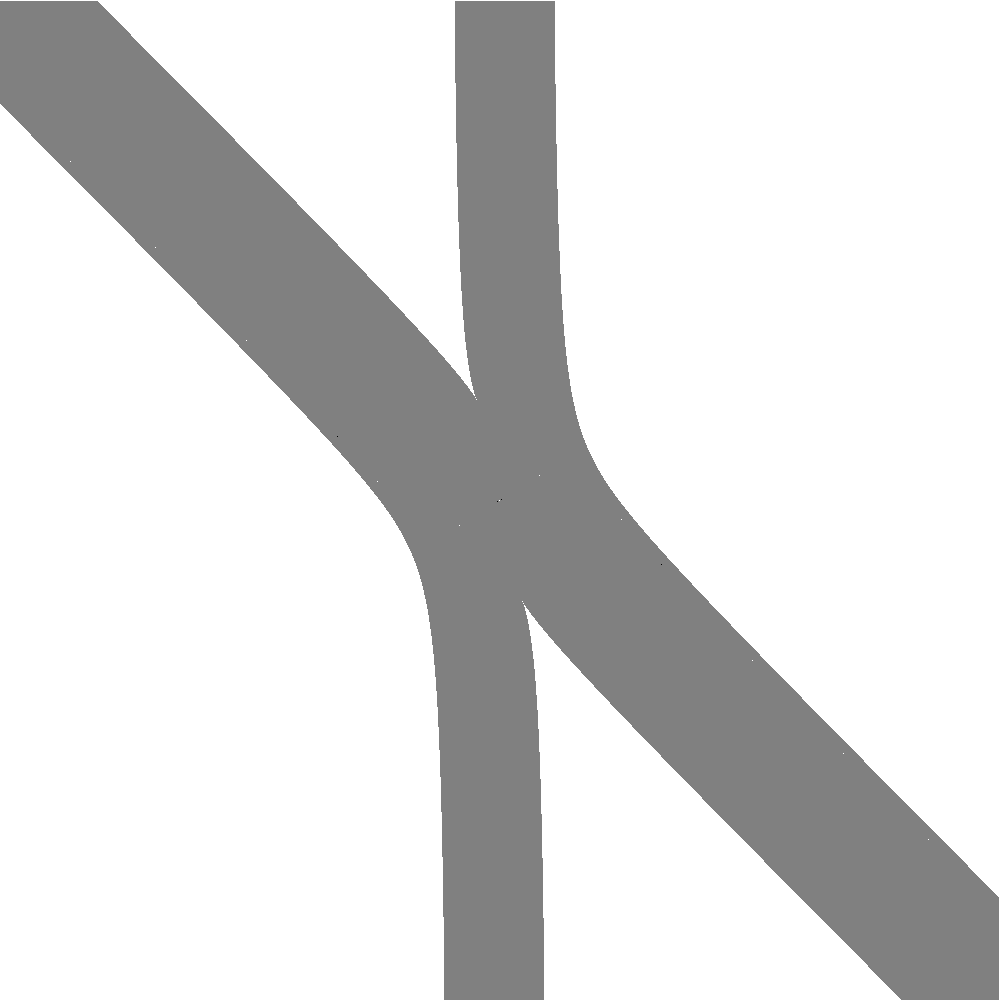}\\
  \end{tabular}
  \caption{%
    \em
    (top) Morphology in Parameter Space for the families of Newton maps $\fN_{x_0,y_0}$ and $\fN_{x_0,y_0;1}$
    in the squares $[-10,5]\times[-5,10]$ and $[-10,10]^2$ respectively. The initial point is $(0,6)$ for the
    parabolic maps and $(5,5)$ for the hyperbolic ones. (bottom) Number of solutions in the corresponding
    squares above: white means 4 real solutions, gray 2 real solutions and black no real solution. No attractor
    besides those of the attracting fixed points were detected in the white regions for the initial points above and
    several other choises of intial points, supporting Conjecture~2.
%    These two pictures show the asymptotics for a fixed point under the iterations of the Newton maps
%    of a 2-parametric family of polynomial maps of type $(2,2)$. Black dots, indicating the possible
%    presence of attractors other than the fixed points, appear only in the region where maps have only
%    two roots, supporting our Conjecture~2.
  }
  \label{fig:misc}
\end{figure}
\section{Morphology in parameter space}
In a celebrated article, Curry, Garnett and Sullivan~\cite{CGS83} found numerical evidence for the existence
of cubic holomorphic polynomials of degree 3 whose Newton map has attractive periodic cycles. The theoretical
basis of that article is an important general result of Fatou on rational holomorphic maps in one variable:
if such a map has an attracting periodic cycle, then one of its critical points must converge to it.

Newton maps of cubic holomorphic polynomials can be parametrized, modulo the equivalence described in
Proposition~\ref{prop:psi}, by 2 real parameters and only one of their critical point is not bound to converge
to any of the three bounded attractive fixed points, so checking the $\omega$-limit of this ``free'' critical point
for some large lattice of such maps is enough to detect any ``large enough'' open subset of them admitting
an attracting cycle. They called the map resulting from this study ``Morphology in Parameter Space'' (MPS). 

In the real case Fatou's result does not hold but analyzing the MPS is anyways interesting for us. For instance,
discovering some Newton map of a polynomial with maximal number of real roots having a basin of attraction
in addition to those corresponding to those roots would disprove our Conjecture~2. We explored numerically
the 2-parametric families $\fN_{x_0,y_0}$ and $\fN_{x_0,y_0;1}$ for several different initial points and found
results qualitatively equivalent to those shown in Fig.~\ref{fig:misc}. In both cases, with a striking precision,
we detected no extra basin in the region where the corresponding polynomial map has four real roots.
Whenever some pattern intersects that region, e.g. like the rocket-shaped pattern in Fig.~\ref{fig:misc} (left),
on the intersection the pattern is filled up only with the colors of the basins of attraction of the four roots,
giving reason to believe that Conjecture~2 is true for Newton maps of polynomials of type $(2,2)$.
Interestingly enough, notice that the main patterns observable in the basins of the maps $\fN_{x_0,y_0}$
and $\fN_{x_0,y_0;1}$ appear also in the corresponding MPS.

\section*{Acknowledgments}
The author gladly thanks J. Yorke for several discussions that greatly helped shaping this article and J. Hawkins for
several clarifications. All calculations were performed on the HPCC of the College of Arts and Sciences at Howard University
with Python and C++ code by the author.
\bibliography{refs}

\providecommand{\bysame}{\leavevmode\hbox to3em{\hrulefill}\thinspace}
\providecommand{\MR}{\relax\ifhmode\unskip\space\fi MR }
\providecommand{\MRhref}[2]{%
  \href{http://www.ams.org/mathscinet-getitem?mr=#1}{#2}
}
\providecommand{\href}[2]{#2}
\begin{thebibliography}{{DeL}18}

\bibitem[Bar53]{Bar53}
B{\'e}la Barna, \emph{{{\"U}ber die Divergenzpunkte des {N}ewtonschen
  Verfahrens zur Bestimmung von Wurzeln algebraischer Gleichungen. I}}, Publ.
  Math. Debrecen \textbf{3} (1953), 109--118.

\bibitem[Bar88]{Bar88}
M.~Barnsley, \emph{Fractals everywhere}, Academic Press, 1988.

\bibitem[BD85]{BD85}
M.~Barnsley and S.~Demko, \emph{Iterated functions systems and the global
  construction of fractals}, Proc. R. Soc. Lond. A \textbf{399} (1985),
  243--275.

\bibitem[CC13]{CC13}
J.T. Campbell and J.T. Collins, \emph{Specifying attracting cycles for {N}ewton
  maps of polynomials}, Journal of Difference Equations and Applications
  \textbf{19} (2013), no.~8, 1361--1379.

\bibitem[CGS83]{CGS83}
J.H. Curry, L.~Garnett, and D.~Sullivan, \emph{On the iteration of a rational
  function: computer experiments with {N}ewton's method}, Communications in
  Mathematical Physics \textbf{91} (1983), no.~2, 267--277.

\bibitem[{DeL}18]{DL18}
R.~{DeLeo}, \emph{"{S}imple {D}ynamics" conjectures for some real {N}ewton maps
  on the plane}, arXiv preprint arXiv:1812.00270 (2018).

\bibitem[Fat19]{Fat19}
Pierre Fatou, \emph{Sur les {\'e}quations fonctionnelles}, Bull. Soc. Math.
  France \textbf{47} (1919), 161--271.

\bibitem[Fat20a]{Fat20a}
\bysame, \emph{Sur les {\'e}quations fonctionnelles}, Bull. Soc. Math. France
  \textbf{48} (1920), 33--94.

\bibitem[Fat20b]{Fat20b}
\bysame, \emph{Sur les {\'e}quations fonctionnelles}, Bull. Soc. Math. France
  \textbf{48} (1920), 208--314.

\bibitem[HH15]{Hub15}
J.H. Hubbard and B.B. Hubbard, \emph{Vector calculus, linear algebra, and
  differential forms: a unified approach}, Matrix Editions, 2015.

\bibitem[HM84]{HM84}
M.~Hurley and C.~Martin, \emph{Newton’s algorithm and chaotic dynamical
  systems}, SIAM journal on mathematical analysis \textbf{15} (1984), no.~2,
  238--252.

\bibitem[HP08]{HP08}
J.H. Hubbard and P.~Papadopol, \emph{Newton's method applied to two quadratic
  equations in $\mathbb c^2$ viewed as a global dynamical system}, American
  Mathematical Soc., 2008.

\bibitem[HSS01]{HSS01}
J.~Hubbard, D.~Schleicher, and S.~Sutherland, \emph{How to find all roots of
  complex polynomials by {N}ewton’s method}, Inventiones mathematicae
  \textbf{146} (2001), no.~1, 1--33.

\bibitem[Hut81]{Hut81}
J.~E. Hutchinson, \emph{Fractals and self-similarity}, Ind. U. Math. J.
  \textbf{30} (1981), 713--747.

\bibitem[Jul18]{Jul18}
Gaston Julia, \emph{Memoire sur l'iteration des fonctions rationnelles},
  Journal de math{\'e}matiques pures et appliqu{\'e}es \textbf{1} (1918),
  47--246.

\bibitem[Kan49]{Kan49}
L.V. Kantorovich, \emph{On {N}ewton's method}, Trudy Matematicheskogo Instituta
  imeni VA Steklova \textbf{28} (1949), 104--144.

\bibitem[Lag98]{Lag98}
J.~Lagrange, \emph{Trait\'e de la r\'esolution des \'equations num\'eriques},
  Paris, 1798.

\bibitem[Lyu86]{Lyu86}
M.Yu. Lyubich, \emph{The dynamics of rational transforms: the topological
  picture}, Russian Mathematical Surveys \textbf{41} (1986), no.~4, 43.

\bibitem[Man80]{Man80}
Benoit~B Mandelbrot, \emph{Fractal aspects of the iteration of $z\to\lambda z
  (1-z)$ for complex $\lambda$ and $z$}, Annals of the New York Academy of
  Sciences \textbf{357} (1980), no.~1, 249--259.

\bibitem[Mil85]{Mil85}
J.W. Milnor, \emph{On the concept of attractor}, The Theory of Chaotic
  Attractors, Springer, 1985, pp.~243--264.

\bibitem[Mil06]{Mil06}
\bysame, \emph{Dynamics in one complex variable}, vol. 160, Springer, 2006.

\bibitem[PPS88]{PPS88}
H.O. Peitgen, M.~Pr{\"u}fer, and K.~Schmitt, \emph{Newton flows for real
  equations}, Rocky Mountain J. Math. \textbf{18} (1988), no.~2, 433--444.

\bibitem[PPS89]{PPS89}
\bysame, \emph{Global aspects of the continuous and discrete {N}ewton method: A
  case study}, Newton’s Method and Dynamical Systems, Springer, 1989,
  pp.~123--202.

\bibitem[PR86]{PR86}
H.O. Peitgen and P.H. Richter, \emph{The beauty of fractals: images of complex
  dynamical systems}, Springer Science \& Business Media, 1986.

\bibitem[Roe05]{Roe05}
R.K. Roeder, \emph{Topology for the basins of attraction of {N}ewton's method
  in two complex variables}, Cornell University, 2005.

\bibitem[SU84]{SU84}
D.G. Saari and J.B. Urenko, \emph{Newton's method, circle maps, and chaotic
  motion}, The American mathematical monthly \textbf{91} (1984), no.~1, 3--17.

\bibitem[Won84]{Won84}
S.~Wong, \emph{Newton’s method and symbolic dynamics}, Proceedings of the
  American Mathematical Society \textbf{91} (1984), no.~2, 245--253.

\bibitem[Ypm95]{Ypm95}
T.J. Ypma, \emph{Historical development of the {N}ewton--{R}aphson method},
  SIAM review \textbf{37} (1995), no.~4, 531--551.

\end{thebibliography}
\end{document}